\documentclass[nosumlimits,twoside]{amsart}
\usepackage{amsfonts, amsmath, amssymb}
\usepackage{graphicx}
\usepackage{float}
\usepackage{srcltx}
\usepackage[all]{xy}
\usepackage{version}
\usepackage[T1]{fontenc}
\usepackage{pifont}
\usepackage[dvipsnames,table]{xcolor}
\usepackage[final]{showlabels}

\usepackage{enumerate}
\usepackage[normalem]{ulem}
\usepackage{bm}
\usepackage{latexsym}
\usepackage[2emode]{psfrag}
\usepackage{yhmath}
\usepackage{array}
\usepackage{dsfont}
\usepackage{mathrsfs}
\usepackage{tikz-cd}
\usepackage{comment}
\usepackage[colorinlistoftodos,prependcaption,textsize=tiny]{todonotes}
\usepackage{mathtools}
\usepackage{hyperref}
\usepackage{enumitem}
\usepackage{stmaryrd}
\usepackage{diagbox}
\usepackage{tabularx}
\usepackage[toc,page]{appendix}

\newtheorem{theorem}{\sc Theorem}[section]
\newtheorem{proposition}[theorem]{\sc Proposition}

\newtheorem{lemma}[theorem]{\sc Lemma}
\newtheorem{corollary}[theorem]{\sc Corollary}

\theoremstyle{definition}
\newtheorem{definition}[theorem]{\sc Definition}

\newtheorem{example}[theorem]{\sc Example}

\theoremstyle{remark}
\newtheorem{remark}[theorem]{\sc Remark}

\newenvironment{invisible}{{\noindent\sc \colorbox{yellow}{Invisible:}\;}\color{gray}}{\medskip}
\excludeversion{invisible}

\allowdisplaybreaks

\excludeversion{proof?}
\newcommand{\id}{\mathsf{id}}
\newcommand{\Cc}{\mathcal{C}}
\newcommand{\Dd}{\mathcal{D}}
\newcommand{\Mm}{\mathcal{M}}
\newcommand{\mm}{\mathfrak{M}}

\newcommand{\Rr}{\mathcal{R}}
\newcommand{\Tt}{\mathcal T}
\newcommand{\Ss}{\mathcal{S}}

\newcommand{\ot}{\otimes}
\newcommand{\ep}{\varepsilon}

\newcommand{\Eq}{\operatorname{Eq}}
\newcommand{\Coeq}{\operatorname{Coeq}}
\newcommand{\tr}{\triangleright}
\newcommand{\tl}{\triangleleft}
\newcommand{\RR}{\mathfrak R}
\newcommand{\Yd}{\mathcal{YD}}

\def\Bialg{{\sf Bialg}}

\def\Vec{{\sf Vec}}

\def\Bimon{{\sf Bimon}}
\def\Comon{{\sf Comon}}
\def\Mon{{\sf Mon}}

\def\Hopfmon{{\sf Hopf}}
\newcommand{\ev}{{\rm ev}}

\newcommand{\ls}[1]{{\color{red}{#1}}}

\newcommand{\op}{\mathrm{op}}
\newcommand{\cop}{\mathrm{cop}}

\newenvironment{myproof}[1][\proofname]{%
  \proof[\bfseries\itshape #1]%
}{\endproof}

\begin{document}
\title[Braidings, duoidal categories and pre-Cartier structures for bi(co)modules]{Braidings, duoidal categories and pre-Cartier structures for bi(co)modules}

\author{Andrea Sciandra, Lukas Simons}

\address{%
\parbox[b]{0.9\linewidth}{Département de Mathématiques, Université Libre de Bruxelles, Boulevard du Triomphe, B-1050
 Bruxelles, Belgium.}}
\email{andrea.sciandra@ulb.be}
\urladdr{\url{www.andreasciandra.com}}
\address{%
\parbox[b]{0.9\linewidth}{Département de Mathématiques, Université Libre de Bruxelles, Boulevard du Triomphe, B-1050
 Bruxelles, Belgium,\\
Vakgroep Wiskunde en Data Science,
Vrije Universiteit Brussel, Pleinlaan 2, B-1050 Brussel, Belgium.}}

\email{lukas.simons@ulb.be, lukas.simons@vub.be}

\maketitle

\begin{abstract}
We study braidings and infinitesimal braidings on categories of
bi(co)modules and tetramodules, and their compatibility with duoidal
structures. For an arbitrary coalgebra, we classify braidings on its
category of bicomodules with the cotensor product in terms of canonical
$R$-forms, dualizing the classification for bimodules obtained by
Agore, Caenepeel, and Militaru. Every such braiding is a symmetry and
admits only the zero infinitesimal braiding. We construct duoidal
structures on categories of bi(co)modules and tetramodules over
bimonoids in braided monoidal categories under suitable assumptions
on equalizers and coequalizers. 
We then introduce pre-Cartier
duoidal categories and their one-sided variants. We establish
obstructions to compatible braidings for tetramodules, construct nonzero one-sided
pre-Cartier structures on bi(co)modules, and give a pre-Cartier duoidal
example with distinct monoidal products and a nonzero infinitesimal
braiding.
\end{abstract}

\tableofcontents

\section{Introduction}

Categories of bimodules and bicomodules carry monoidal structures
that reflect the algebraic data of their coefficients. Bimodules over
an algebra are composed using the balanced tensor product, while
bicomodules over a coalgebra are composed using the cotensor product.
When the coefficients form a bialgebra, its comultiplication and
multiplication also induce monoidal structures on the ordinary tensor
product. Tetramodules combine compatible module and comodule
structures and support both relative products. These constructions
lead to two related questions: which of the resulting monoidal
categories admit braidings, and how can braidings and their
infinitesimal counterparts be made compatible with the interaction
between two monoidal products?

The existence of a braiding depends on the chosen product. For the
relative tensor product of bimodules, Agore, Caenepeel, and Militaru
classified braidings by canonical $R$-matrices and proved that every
such braiding is a symmetry \cite{agore2014braidings}. This classification provides a
starting point for the corresponding problem for bicomodules and the
cotensor product. Canonical $R$-matrices and their coalgebraic
counterparts must be distinguished from universal $R$-matrices and
universal $R$-forms, which describe braidings on module and comodule
categories equipped with the ordinary tensor product. Keeping track
of this distinction is essential when both products occur in the
same category.

Duoidal categories provide a framework for studying their interaction.
They consist of two monoidal structures connected by an interchange
map and compatible unit maps \cite{Aguiar}. A braided monoidal category gives
a duoidal category by using its monoidal product twice, with interchange
induced by the braiding. Categories of bi(co)modules provide examples
involving different products. For tetramodules over a bialgebra, the
relative tensor and cotensor products form a duoidal category, as
shown by Shoikhet in the category of vector spaces \cite{Shoikhet}. In the Hopf
case, the relation between these products and Yetter--Drinfeld modules
is governed by the structure theory of Hopf bimodules \cite{Drabant,BespalovDrabant,Schauenburg}.

Infinitesimal braidings lead to a further compatibility problem.
Pre-Cartier categories equip a braided monoidal category with a
natural family of endomorphisms satisfying the infinitesimal braid
identities \cite{ABSW}. These identities arise from the first-order expansion
of the braiding axioms. In a duoidal category, the infinitesimal data
must also respect the interchange and the relevant unit maps. We
formulate these requirements for either monoidal product separately
and for both products simultaneously. This distinction is needed
for the examples considered here: a compatible braiding on one
product need not extend to a braided duoidal structure.\medskip

\noindent\textbf{Structure of the paper.} Section \ref{sec:preliminaries} recalls the required facts about bi(co)modules and
tetramodules in braided monoidal categories, together with
pre-Cartier categories and pre-Cartier bialgebras.

In Section \ref{sec:braidingbicomodules}, we classify braidings for the monoidal category $(^{C}\mm^{C},\square^{C},C)$, for an arbitrary coalgebra $C$ (Theorem \ref{theorem: braidings on bicomodules}), showing that these are all symmetries. The classification is expressed in terms of canonical $R$-forms $\Rr:C\ot C\ot C\to\Bbbk$ satisfying identities dual to those given in \cite{agore2014braidings}, where braidings for the monoidal category $(_{A}\mm_{A},\ot_{A},A)$ are classified for an arbitrary algebra $A$. Examples of braidings on $(^{C}\mm^{C},\square^{C},C)$ are provided (Example \ref{ex:braidingbicomod} and Lemma \ref{lem:rtensorproduct}), and the connection between braidings on the bi(co)modules of finite-dimensional (co)algebras and their duals is explained (Proposition \ref{cor:bijectionsimmetries}). If $H$ is a bialgebra, then it is shown that $(^{H}\mm^{H},\square^{H},H)$ nor $(_H\mm_H, \ot_H,H)$ admits a braiding, unless $H = \Bbbk$ (Proposition \ref{prop:trivialcanonicalform}). Dually to the result achieved in \cite{ASthesis} for the category $(_{A}\mm_{A},\ot_{A},A)$, the braided monoidal category $(^{C}\mm^{C},\square^{C},C,\sigma)$ has only trivial infinitesimal braiding, for any braiding $\sigma$ (Theorem \ref{thm:infbraidbicomod}). For readability, the proofs of results dual to those of \cite{agore2014braidings} are gathered in Appendix \ref{sec:appendix}.

In Section \ref{sec:duoidal}, we study duoidal structures associated to categories of bi(co)modules over a bimonoid in an arbitrary monoidal category. More precisely, we prove that, given a braided monoidal category $(\Mm,\ot,I,\sigma)$ with coequalizers which are preserved by $\ot$, then $(_H\Mm_H, \ot_H, H, \ot, I)$ is a duoidal category and, dually, if $(\Mm,\ot,I,\sigma)$ is a braided monoidal category with equalizers which are preserved by $\ot$, then $(^H\Mm^H, \ot, I, \square^H, H)$ is a duoidal category (Theorem \ref{theorem: bi(co)modules duoidal}). 
These constructions recover the known examples over vector spaces and can be applied again when the ambient category is itself a
braided category of bi(co)modules.

In Section \ref{sec:tetramodulesduoidal}, we prove that, given a braided monoidal category $(\Mm,\ot,I,\sigma)$ and $H$ in $\Bimon(\Mm)$, if $\Mm$ has equalizers and coequalizers which are both preserved by $\ot$, then $({}^{H}_{H}\Mm^H_H, \ot_H, H, \square^H, H)$ is duoidal (Theorem \ref{thm:duoidaltetramodules}), generalizing the result achieved in \cite{Shoikhet} in the case $\Mm=\mathsf{Vec}_{\Bbbk}$. When $H$ is an object in $\mathsf{Hopf}(\Mm)$, we show that the relative tensor
and cotensor products admit a common realization in which
their induced monoidal structures coincide. Assuming that the antipode is bijective,
we show that the equivalence with Yetter--Drinfeld modules
identifies the tetramodule duoidal structure with the one induced
by the canonical Yetter--Drinfeld braiding (Proposition \ref{prop:duoidalequivalenceYetterDrinfeld}). We also exhibit a
nonzero infinitesimal braiding on a category of Yetter--Drinfeld
modules (Example \ref{ex:infbraidYD}).

Finally, in Section \ref{sec:preCartier} we study braidings compatible with an interchange
map and introduce $\circ$-pre-Cartier and $\bullet$-pre-Cartier
duoidal categories, as well as pre-Cartier duoidal categories
satisfying both sets of conditions (Definition \ref{def:preCartier}). We show that the canonical
duoidal structure on tetramodules over a $\Bbbk$-bialgebra is
$\circ$-braided or $\bullet$-braided only when $H\cong\Bbbk$ (Theorem \ref{prop:tetramodule-braided-duoidal-obstruction}).
We also prove that the duoidal category obtained by using the
same symmetric monoidal product twice admits only zero
infinitesimal braidings satisfying these compatibility conditions (Proposition \ref{prop:trivialpreCartier}).
We construct one-sided pre-Cartier structures on bi(co)modules
from pre-Cartier bialgebras (Proposition \ref{prop:bimodulespreCartier} and Corollary \ref{cor:otherinfbraidingbimod}), give explicit nonzero examples (Examples \ref{ex:1}, \ref{ex:2}),
and conclude with a pre-Cartier duoidal category built from
truncated Cauchy and Hadamard products whose infinitesimal
braiding is nonzero (Example \ref{ex:fullpreCartier}). \medskip

\noindent\textit{Notations and conventions}. We fix a field $\Bbbk$ and all vector spaces are understood to be $\Bbbk$-vector spaces unless
otherwise specified. By a linear map we mean a $\Bbbk$-linear map and all linear maps whose domain is a tensor product will usually be defined on generators and understood to be extended by linearity. The category of vector spaces will be denoted by $\mathsf{Vec}_{\Bbbk}$, $\Vec$, or $\mm$. (Co-/Bi-/Hopf) Algebras will be understood to be over $\Vec_{\Bbbk}$ unless the context suggests otherwise. Algebras will be associative and unital and coalgebras will be coassociative and counital. For a $\Bbbk$-coalgebra $C$, we will employ the Sweedler notation for the coproduct $\Delta(x)=x_{1}\ot x_{2}$, omitting the summation symbol. Similarly, working with left (resp.\ right) $\Bbbk$-comodules $V$ we will employ a Sweedler type notation $\rho^{L}(v)=v_{(-1)}\ot v_{(0)}$ (resp.\ $\rho^{R}(v)=v_{(0)}\ot v_{(1)}$). We will sometimes omit writing parentheses for coactions.

Given an object $X$ in a category $\Mm$, the identity morphism on $X$ will be denoted either by $\id_{X}$ or $X$ for short. We will usually denote the equalizer of two parallel morphisms $f,g:X\to Y$ in a category $\Mm$ by $\mathrm{eq}(f,g):\Eq(f,g)\to X$ and the coequalizer by $\mathrm{coeq}(f,g):Y\to\mathrm{Coeq}(f,g)$. When considering the cotensor product over a coalgebra $C$, we will sometimes omit the superscript if the context makes it clear over what coalgebra the cotensor product is taken, and hence simply write $\square$ instead of $\square^C$.

Comforted by the Mac Lane Coherence theorem \cite[page 169]{MacLane}, we shall consistently be sloppy on associativity and unit constraints.

\section{Preliminaries}\label{sec:preliminaries}
In this section, we recall some notions and results that will be useful throughout this paper. We refer the reader to e.g.\ \cite[XIII.1]{Kassel-book} for the notion of braided monoidal category $(\Mm,\ot,I,\sigma)$. For limits and colimits and other basic notions of category theory we refer the reader to e.g.\ \cite{Handbook1}.

\subsection{Bi(co)modules and tetramodules in braided monoidal categories}

Given a monoidal category $(\Mm,\ot,I)$, one can define the categories $\mathsf{Mon}(\Mm)$ and $\mathsf{Comon}(\Mm)$ of monoids and comonoids in $\Mm$, respectively. Given an object $A$ in $\mathsf{Mon}(\Mm)$, one can define the category $_{A}\Mm$ of left $A$-modules in $\Mm$ and, similarly, right $A$-modules and $A$-bimodules in $\Mm$. The latter being left and right $A$-modules for which the order of acting on the left and right is unimportant. Dually, given an object $C$ in $\mathsf{Comon}(\Mm)$, one can define the category of left $C$-comodules in $\Mm$ and, similarly, the categories of right $C$-comodules and $C$-bicomodules. Let $(\Mm,\ot,I)$ be a monoidal category with coequalizers, $(A,m,u)$ be an object in $\mathsf{Mon}(\Mm)$, and $N\in {}_A\Mm$, $M\in \Mm_{A}$ with structure morphisms $\alpha^{L}_N\colon A\otimes N\to N$ and $\alpha^{R}_M:M\otimes A\to M$, respectively. 
Then we define the \emph{$A$-balanced tensor product} of $M$ and $N$ as
\begin{equation}\label{def:balancedtensor}
    (M\ot_A N, q_{M,N}) = \mathrm{coeq}\left(
\begin{tikzcd}
	{M\ot A\ot N} && {M\ot N}
	\arrow["{M\ot \alpha^L_N}"', shift right=1.5, from=1-1, to=1-3]
	\arrow["{\alpha^R_M\ot N}", shift left=1.5, from=1-1, to=1-3]
\end{tikzcd}\right).
\end{equation}
For any morphism $f\colon M\to X$ in $\Mm_A$ and $g\colon N\to Y$ in ${}_{A}{\Mm}$, there is a unique morphism $f\ot_A g\colon M\ot_A N\to X\ot_A Y$ such that $q_{X,Y}(f\ot g) = (f\ot_A g)q_{M,N}$. 
This construction provides functors $M\ot_{A}(-):{}_{A}\Mm\to\Mm$ and $(-)\ot_{A}N:\Mm_{A}\to\Mm$ for any $M$ in $\Mm_{A}$ and any $N$ in $_{A}\Mm$. 
For $M\in\Mm_A$ and $N\in {}_A\Mm$, there are canonical (natural) isomorphisms $\Upsilon_M$, $\Upsilon'_N$ in $\Mm$:
\begin{itemize}
    \item $\Upsilon_M:M\otimes_AA\to M$, uniquely determined by $\Upsilon_Mq_{M,A}=\alpha^{R}_M$; 
   \item $\Upsilon'_N:A\otimes_AN\to N$, uniquely determined by  $\Upsilon'_Nq_{A,N}=\alpha^{L}_N$. 
\end{itemize}
One can check that $\Upsilon^{-1}_M=q_{M,A}(\id_M\otimes u)$ and $\Upsilon'^{-1}_N=q_{A,N}(u\otimes \id_N)$. \begin{invisible}
Indeed, $\Upsilon_{M}\Upsilon_{M}^{-1}=\Upsilon_{M}q_{M,A}(\id_{M}\ot u)=\alpha^{R}_{M}(\id_{M}\ot u)=\id$ and $\Upsilon_{M}^{-1}\Upsilon_{M}q_{M,A}=\Upsilon_{M}^{-1}\alpha^{R}_{M}=q_{M,A}(\id_{M}\ot u)\alpha^{R}_{M}=q_{M,A}$.
\end{invisible}
Following \cite{MR2696373}, a monoidal category $(\Mm,\ot,I)$ is said to be \textit{coregular} if it has coequalizers and these are preserved by $\ot$, i.e.\ the functors $A\ot(-):\Mm\to\Mm$ and $(-)\ot B:\Mm\to\Mm$ preserve coequalizers, for any $A,B$ in $\Mm$. This means that, for any pair of parallel morphisms $f,g:X\to Y$ in $\Mm$, 
the induced morphism $\phi$ in the following diagram
\[\begin{tikzcd}
	A\ot X\ot B & A\ot Y\ot B & \mathrm{Coeq}(\id\ot f\ot \id,\id\ot g\ot\id) \\\\
	& A\ot\mathrm{Coeq}(f,g)\ot B
	\arrow[shift left=1.5, from=1-1, to=1-2,"\id\ot f\ot\id"]
	\arrow[shift right=1.5, from=1-1, to=1-2,"\id\ot g\ot\id"']
	\arrow[from=1-2, to=1-3,"\mathrm{coeq}"]
	\arrow[from=1-2, to=3-2,"{\id\ot\mathrm{coeq}(f,g)\ot\id}"']
	\arrow[dashed, from=1-3, to=3-2,"\phi"]
\end{tikzcd}\]
is an isomorphism in $\Mm$.

Let $(\Mm,\otimes,I)$ be a monoidal category with equalizers and $(C,\Delta,\varepsilon)$ be an object in $\mathsf{Comon}(\Mm)$. Recall from e.g.\ \cite[Definition 2.2.1]{MR2696373} that, given a right $C$-comodule $(V, \rho^{R}_V)$  and a left $C$-comodule $(W,\rho^{L}_W)$ in $\Mm$, their \textit{cotensor product over $C$} in $\Mm$ is defined as
\begin{equation}\label{def:cotensor}
    (V\square^C W, e_{V,W}) = \mathrm{eq}\left( 
\begin{tikzcd}
	{V\ot W} && {V\ot C\ot W}
	\arrow["{\rho^R_V\ot W}", shift left=1.5, from=1-1, to=1-3]
	\arrow["{V\ot\rho^L_W}"', shift right=1.5, from=1-1, to=1-3]
\end{tikzcd} \right).
\end{equation}
For any morphisms $f\colon V\to X$ in $\Mm^C$ and $g\colon W\to Y$ in ${}^{C}{\Mm}$ there is a unique morphism $f\square^C g\colon V\square^C W \to X\square^C Y$ such that $(f\ot g)e_{V,W} = e_{X,Y}(f\square^C g)$.
This construction provides functors $V\square^{C}(-):{}^{C}\Mm\to\Mm$ and $(-)\square^{C}W:\Mm^{C}\to\Mm$ for any $V$ in $\Mm^{C}$ and any $W$ in $^{C}\Mm$.
For $V\in\Mm^{C}$ and $W\in{}^{C}\Mm$, we have the canonical (natural) isomorphisms $\Lambda_V$, $\Lambda'_{W}$ in $\Mm$: 
\begin{itemize}
    \item $\Lambda _V:V\to V\square^CC$, uniquely determined by $\rho^{R}_V= e_{V,C} \Lambda_V$;
    \item $\Lambda'_W: W\to C\square^CW$, uniquely determined by the property $\rho^{L}_W=e_{C,W}\Lambda'_W$.
\end{itemize}
One can easily check that $\Lambda_M^{-1}=(\id_M\otimes\varepsilon)e_{M,C}$ and $(\Lambda'_Y)^{-1}=(\varepsilon\otimes\id_Y)e_{C,Y}$. 

A monoidal category $(\Mm,\otimes, I)$ is said to be \emph{regular} if it has equalizers and $\ot$ preserves them, i.e.\ the functors $A\ot(-)$ and $(-)\ot B$ preserve equalizers, for any $A,B$ in $\Mm$. This means that, for any pair of parallel morphisms $f,g:X\to Y$ in $\Mm$, 
the induced morphism $\psi$ in the following diagram
\[\begin{tikzcd}
	\Eq(\id\ot f\ot \id,\id\ot g\ot \id) & A\ot X\ot B & A\ot Y\ot B \\\\
	& A\ot\Eq(f,g)\ot B
	\arrow[from=1-1, to=1-2,"\mathrm{eq}"]
	\arrow[shift left=1.5, from=1-2, to=1-3,"\id\ot f\ot \id"]
	\arrow[shift right=1.5, from=1-2, to=1-3,"\id\ot g\ot \id"']
	\arrow[dashed, from=3-2, to=1-1,"\psi"]
	\arrow[from=3-2, to=1-2, "{\id\ot\mathrm{eq}(f,g)\ot \id}"']
\end{tikzcd}\]
    is an isomorphism in $\Mm$. 

The classical example of a regular and coregular category is the monoidal category $(\Vec_{\Bbbk},\ot_{\Bbbk},\Bbbk)$ of vector spaces over a field $\Bbbk$. 

The category $_{A}\Mm_{A}$ of $A$-bimodules in $\Mm$ over an algebra $A\in\mathsf{Mon}(\Mm)$ has objects given by triples $(V,\alpha^{L}_{V},\alpha^{R}_{V})$ such that $(V,\alpha^{L}_{V})$ is a left $A$-module, $(V,\alpha^{R}_{V})$ is a right $A$-module and $\alpha^{L}_{V}(\id_{A}\ot\alpha^{R}_{V})=\alpha^{R}_{V}(\alpha^{L}_{V}\ot\id_{A})$. Morphisms in $_{A}\Mm_{A}$ are left $A$-linear maps which are also right $A$-linear”. Dually, given $C$ in $\Comon(\Mm)$, a $C$-bicomodule is a triple $(V,\rho_V^L,\rho_V^R)$ with $V$ in $\Mm$ such that $(V,\rho_V^L)$ is a left $C$-comodule, $(V,\rho_V^R)$ is a right $C$-comodule, and $(\rho_V^L\otimes \id_C)\rho_V^R = (\id_C\otimes\rho_V^R)\rho_V^L$. A morphism of $C$-bicomodules is a left $C$-colinear morphism which is also right $C$-colinear. We denote the category of $C$-bicomodules in $\Mm$ by ${}^C\Mm^C$.


The following known result shows that the category of bimodules over an algebra in an arbitrary coregular monoidal category is monoidal, equipped with the tensor product over the algebra and the algebra itself as unit object. Dually, the category of bicomodules over a coalgebra in a regular monoidal category is monoidal, equipped with the cotensor product over the coalgebra and the coalgebra itself as unit object. We point out that these results are known in the literature, see e.g.\ \cite[page 329]{Schbimod} for bimodules. 
However, we report some details for the sake of completeness, showing where the assumption of (co)regularity is needed.

\begin{theorem}\label{theorem: bi(co)modules monoidal}
    Let $(\Mm, \ot, I)$ be a monoidal category.
    \begin{enumerate}[label=\arabic*)]
        \item If $A$ is in $\Mon(\Mm)$ and $\Mm$ is coregular, then $({}_{A}{\Mm}_A, \ot_A, A)$ is a monoidal category.
        \item If $C$ is in $\Comon(\Mm)$ and $\Mm$ is regular, then $({}^{C}{\Mm}^C, \square^C, C)$ is a monoidal category.
    \end{enumerate}
\end{theorem}
\begin{proof}
    1). 
    Let $(M, \alpha^L_M, \alpha^R_M)$ and $(X, \alpha^L_X, \alpha^R_X)$ be $A$-bimodules in $\Mm$, and define $M\ot_{A}X$
    with canonical morphism $q_{M,X}\colon M\ot X\to M\ot_A X$ as in \eqref{def:balancedtensor}. Consider the diagram
\[\begin{tikzcd}
	{A\ot M\ot A\ot X} && {A\ot M \ot X} && {M\ot_A X.}
	\arrow["{A\ot \alpha^R_M\ot X}", shift left=1.5, from=1-1, to=1-3]
	\arrow["{A\ot M \ot \alpha^L_X}"', shift right, from=1-1, to=1-3]
	\arrow["{q_{M,X}(\alpha^L_M\ot X)}", from=1-3, to=1-5]
\end{tikzcd}\]
One immediately verifies both compositions to agree, since $M$ is an $A$-bimodule and $q_{M,X}$ coequalizes the pair $(\alpha^{R}_{M}\ot\id,\id\ot\alpha^{L}_{X})$. Hence, by the universal property of the coequalizer, there exists a unique morphism in $\Mm$ from the coequalizer of the first two parallel morphisms above to $M\ot_A X$ satisfying the universal property. Since $\Mm$ is coregular, this coequalizer is isomorphic to $A\ot (M\ot_A X)$ with canonical epimorphism $\id\ot q_{M,X}$, whence there is a unique morphism $\alpha_{M\ot_A X}^L \colon A\ot (M\ot_A X)\to M\ot_A X$ in $\Mm$ satisfying $\alpha_{M\ot_A X}^L(\id\ot q_{M,X}) = q_{M,X}(\alpha_M^L \ot \id)$. This makes $M\ot_A X$ into a left $A$-module. 
\begin{invisible}
Indeed, notice that
\[
\alpha^L_{M\ot_A X}(A\ot\alpha^L_{M\ot_A X}), \alpha^L_{M\ot_A X}(\mu\ot M\ot_A X)\colon A\ot A\ot M\ot_A X \to M\ot_A X.
\]
Since both are morphisms from a coequalizer, it suffices that they are equal after composing with its canonical epimorphism to conclude their equality. By coregularity we compute that
\begin{align*}
    \alpha^L_{M\ot_A X}(A\ot\alpha^L_{M\ot_A X})(A\ot A \ot q_{M,X}) &= \alpha^L_{M\ot_A X}(A\ot q_{M,X})(A\ot \alpha^L_M\ot X)
    \\&= q_{M,X}(\alpha^L_M\ot X)(A\ot \alpha^L_M\ot X)
    \\&=q_{M,X}(\alpha^L_M\ot X)(\mu\ot M\ot X)
    \\&= \alpha^L_{M\ot_A X}(A\ot q_{M,X})(\mu\ot M\ot X)
    \\&= \alpha^L_{M\ot_A X}(\mu\ot M\ot_A X)(A\ot A\ot q_{M,X}).
\end{align*}
Similarly, we find the following equalities
\begin{align*}
    \alpha^L_{M\ot_A X}(\eta\ot M\ot_A X)(I\ot q_{M,X}) &= \alpha^L_{M\ot_A X}(A\ot q_{M,X})(\eta\ot M\ot X)
    \\&= q_{M,X}(\alpha^L_M \ot X)(\eta\ot M\ot X)
    \\&= q_{M,X}(l_M \ot X)
    \\&= q_{M,X}l_{M\ot X}
    \\&= l_{M\ot_A X}(I\ot q_{M,X}),
\end{align*}
where we utilized moreover properties of the left unitor. Hence $(M\ot_A X, \alpha^L_{M\ot_A X})$ is a left $A$-module. 
\end{invisible}
Similarly, there is a unique morphism $\alpha^R_{M\ot_A X}\colon (M\ot_A X) \ot A \to M\ot_A X$ in $\Mm$ satisfying $\alpha^R_{M\ot_A X}(q_{M,X}\ot A) = q_{M,X}(M\ot \alpha^R_X)$, and this makes $M\ot_A X$ into a right $A$-module. 
\begin{invisible}
To conclude, notice that
\begin{align*}
    \alpha^L_{M\ot_A X}(A\ot \alpha^R_{M\ot_A X})(A\ot q_{M,X}\ot A) &= \alpha^L_{M\ot_A X}(A\ot q_{M,X})(A\ot M\ot \alpha_X^R)
    \\&= q_{M,X}(\alpha_M^L\ot X)(A\ot M \ot \alpha^R_X)
    \\&= q_{M,X}(M\ot \alpha_X^R)(\alpha^L_M\ot X\ot A)
    \\&= \alpha^R_{M\ot_A X}(q_{M,X}\ot A)(\alpha^L_M \ot X\ot A)
    \\&= \alpha^R_{M\ot_A X}(A\ot \alpha^L_{M\ot_A X})(A\ot q_{M,X}\ot A).
\end{align*}
\end{invisible}
Finally, $(M\ot_A X, \alpha^L_{M\ot_A X}, \alpha^R_{M\ot_A X})$ becomes an $A$-bimodule.
\begin{invisible}
    If $f\colon M\to N$ and $g\colon X\to Y$ are $A$-bimodule morphisms in $\Mm$, it holds that the following compositions of morphisms are equal
\[\begin{tikzcd}
	{M\ot A\ot X} && {M\ot X} && {N\ot Y} && {N\ot_A Y.}
	\arrow["{M\ot \alpha_X^L}"', shift right, from=1-1, to=1-3]
	\arrow["{\alpha^R_M\ot X}", shift left=1.5, from=1-1, to=1-3]
	\arrow["{f\ot g}", from=1-3, to=1-5]
	\arrow["{q_{N,Y}}", from=1-5, to=1-7]
\end{tikzcd}\]
By the universal property of the coequalizer, there is a unique morphism $f\ot_A g\colon M\ot_A X\to N\ot_A Y$ in $\Mm$ such that $q_{N,Y}(f\ot g)=(f\ot_{A}g)q_{M,X}$.
One can verify that $f\ot_A g$ is an $A$-bimodule morphism. 
Notice that
\begin{align*}
    (f\ot_A g)\alpha^L_{M\ot_A X}(A\ot q_{M,X}) &= (f\ot_A g)q_{M,X}(\alpha^L_M\ot X)
    \\&= q_{N,Y}(f\ot g)(\alpha^L_M\ot X)
    \\&= q_{N,Y}(\alpha^L_N\ot Y)(A\ot f\ot g)
    \\&= \alpha_{N\ot_A Y}^L(A\ot q_{N,Y})(A\ot f\ot g)
    \\&= \alpha_{N\ot_A Y}^L (A\ot f\ot_A g)(A\ot q_{M,X}).
\end{align*}
Thus, $f\ot_A g$ is a left $A$-module morphism and, analogously, one proves that it is a right $A$-module morphism as well. 
The functoriality of $\ot_A$ is easily verified, again by the universal property of the coequalizer. 
\end{invisible}
Moreover, $(A,\mu,\mu)$ is immediately an $A$-bimodule by the associativity of the multiplication.\newline
For the monoidal structure we need left and right unitors, as well as an associator. Given $M$ in $_{A}\Mm_{A}$, left and right unitors are provided by $\Upsilon'_{M}$ and $\Upsilon_{M}$ defined before. Moreover, if $a$ is the associator of $\Mm$, the associator $\alpha_{M,N,P}\colon (M\ot_A N)\ot_A P\to M\ot_A (N\ot_A P)$ is defined as the unique morphism in $\Mm$ such that
\[
\alpha_{M,N,P}q_{M\ot_A N,P}(q_{M,N}\ot P)= q_{M, N\ot_A  P}(M \ot q_{N,P})a_{M,N,P}.
\]
\begin{invisible}
 A priori, $A\ot_A M = (M, \alpha^L_M)$ and $M\ot_A A = (M,\alpha^R_M)$ for any $A$-bimodule $M$.
    If $M$ is an $A$-bimodule, we claim that
    \begin{equation}\label{eq: A unit for ot_A}
    A\ot_A M = (M, \alpha^L_M), \hspace{3em} M\ot_A A = (M, \alpha^R_M).
    \end{equation}
    This implies the left and right unitors to be identities. Notice that $(M,\alpha^L_M)$ does have the property that it coequalizes the left action of $M$ and the right action of $A$;
\[\begin{tikzcd}
	{A\ot A\ot M} && {A\ot M} && {M.}
	\arrow["{A\ot \alpha^L_M}"', shift right, from=1-1, to=1-3]
	\arrow["{\mu\ot M}", shift left=1.5, from=1-1, to=1-3]
	\arrow["{\alpha^L_M}", from=1-3, to=1-5]
\end{tikzcd}\]
For the universal property, assume $N$ to be an $A$-bimodule and $f\colon A\ot M \to N$ an $\Mm$-morphism coequalizing said actions. Then define
\[
\bar{f} \colon M\xrightarrow{l_M^{-1}}I\ot M\xrightarrow{\eta\ot M}A\ot M\xrightarrow{f} N.
\]
This is an $\Mm$-morphism as composition thereof. It remains to be verified that $f = \bar{f}\alpha^L_M$, and that $\bar{f}$ is the unique $\Mm$-morphism with this property. For the former, utilizing properties of the left unitor and the fact that $f$ coequalizes said actions, we compute that
\begin{align*}
    \bar{f}\alpha^L_M&= f(\eta\ot M)l_M^{-1}\alpha^L_M
    \\&=f(\eta\ot M)(I\ot \alpha^L_M)l_{A\ot M}^{-1}
    \\&= f(A\ot \alpha^L_M)(\eta\ot A\ot M)l_{A\ot M}^{-1}
    \\&= f(\mu\ot M)(\eta\ot A\ot M)l_{A\ot M}^{-1}
    = fl_{A\ot M}l_{A\ot M}^{-1}
    = f.
\end{align*}
For the unicity, if $g\colon M \to N$ is a morphism such that $f = g\alpha^L_M = \bar{f}\alpha^L_M$, then a priori $gl_M = g\alpha^L_M(\eta\ot M) = \bar{f}\alpha^L_M(\eta\ot M) = \bar{f}l_M$, which implies that $g = \bar{f}$. Likewise one proves the second identity, rendering the right unitor trivial as well.
To construct the associator, we consider the following diagram 
\begin{equation}\label{eqn: first of associator}
\begin{tikzcd}
	{M\ot N\ot P} && {(M\ot_A N)\ot P} && {(M\ot_A N)\ot_A P}
	\arrow["{q_{M,N}\ot P}", from=1-1, to=1-3]
	\arrow["{q_{M\ot_A N, P}}", from=1-3, to=1-5]
\end{tikzcd}
\end{equation}
One verifies that the morphism in \eqref{eqn: first of associator} coequalizes the pair $(M\ot\alpha^{R}_{N}\ot P,M\ot N\ot\alpha^{L}_{P})$, by definition of the $A$-bimodule structure on $M\ot_A N$.
Thus, there exists a unique $\Mm$-morphism $f\colon M\ot (N\ot_A P)\to (M\ot_A N)\ot_A P$ such that
\begin{equation}\label{eqn: up of f}
    q_{M\ot_A N, P}(q_{M,N}\ot P) = f(M\ot q_{N,P}).
\end{equation}
One can show that $f$ coequalizes the pair $(\alpha^R_M\ot (N\ot_A P),M\ot \alpha^L_{N\ot_A P})$.
As the canonical morphisms of coequalizers are always epimorphisms, this is equivalent to stating that
\[
f(\alpha^R_M\ot N\ot_A P)(M\ot A\ot q_{N,P}) = f(M\ot \alpha^L_{N\ot_A P})(M\ot A\ot q_{N,P}),
\]
which is easily verified by condition (\ref{eqn: up of f}) and the definition of the $A$-bimodule structure on $N\ot_A P$.
Hence there exists a unique $\Mm$-morphism $g\colon M\ot_A (N\ot_A P)\to (M\ot_A N)\ot_A P$ such that
\begin{equation}\label{eqn: up of g}
    gq_{M, N\ot_A P} = f.
\end{equation}
Completely analogously, the morphism $q_{M, N\ot_A  P}(M \ot q_{N,P})\colon M\ot N \ot P \to M\ot_A (N\ot_A P)$ induces unique $\Cc$-morphisms
\begin{align*}
    \bar{f}&\colon (M\ot_A N)\ot P\to M\ot_A (N\ot_A P),
    \\\bar{g}&\colon (M\ot_A N)\ot_A P\to M\ot_A (N\ot_A P)
\end{align*}
satisfying
\begin{equation}\label{eqn: ups of bar f bar g}
    \bar{f} (q_{M, N}\ot P) = q_{M, N\ot_A  P}(M \ot q_{N,P}),\hspace{3em} \bar{g}q_{M\ot_A N,P} = \bar{f}.
\end{equation}
One can prove that $g$ and $\bar{g}$ are inverse morphisms, and the latter provide the associator. 
Utilizing the identities (\ref{eqn: up of f}), (\ref{eqn: up of g}), and (\ref{eqn: ups of bar f bar g}) one proves immediately that
\begin{align*}
&\bar{g}gq_{M,N\ot_A P}(M\ot q_{N,P}) = q_{M,N\ot_A P}(M\ot q_{N,P}),
\\&g\bar{g}q_{M\ot_A N, P}(q_{M, N}\ot P) = q_{M\ot_A N, P}(q_{M, N}\ot P),
\end{align*}
which proves that $g$ and $\bar{g}$ are inverse morphisms since the canonical morphisms of coequalizers are epimorphic. Hence the category $(_A\Mm_A, \ot_A, A)$ is a monoidal category.
\end{invisible}
\noindent 2). The case of $C$-bicomodules is exactly the dual of the previous one. 
\begin{invisible}
We briefly recall some details. Given two $C$-bicomodules $V$ and $W$, one defines $V\square^{C}W$ with the canonical morphism $e_{V,W}:V\square^{C}W\to V\ot W$ as in \eqref{def:cotensor}. There exist unique morphisms $\rho^{L}_{V\square W}$ and $\rho^{R}_{V\square W}$ in $\Mm$ such that the diagrams in (\ref{diagram: bicomodule on V square W}) commute.
    \begin{equation}\label{diagram: bicomodule on V square W}
\begin{tikzcd}
	{V\square^C W} & {V\otimes W} && {V\square^C W} & {V\otimes W} \\
	{C\otimes (V\square^C W)} & {C\otimes V\otimes W,} && {(V\square^C W)\otimes C} & {V\otimes W \otimes C}
	\arrow["{e_{V,W}}", from=1-1, to=1-2]
	\arrow["{\rho_{V\square W}^L}"', dashed, from=1-1, to=2-1]
	\arrow["{\rho_V^L\ot\id}", from=1-2, to=2-2]
	\arrow["{e_{V,W}}", from=1-4, to=1-5]
	\arrow["{\rho_{V\square W}^R}"', dashed, from=1-4, to=2-4]
	\arrow["{\id\ot\rho_W^R}", from=1-5, to=2-5]
	\arrow["{\id\ot e_{V,W}}"', from=2-1, to=2-2]
	\arrow["{e_{V,W}\ot\id}"', from=2-4, to=2-5]
\end{tikzcd}
\end{equation}
This makes $(V\square^C W, \rho_{V\square W}^L, \rho_{V\square W}^R)$ into a $C$-bicomodule. The existence is due to the universal property of equalizers, while the $C$-bicomodule structure can be pulled back through the monomorphism $e_{V,W}$, see \cite[Proposition 2.2.1]{MR2696373}. Given morphisms $f\colon V\to W$, $g\colon X\to Y$ in $^{C}\Mm^{C}$,

Then, by definition of $C$-bicomodule morphisms, the top (respectively bottom) two compositions of morphisms are equal in (\ref{eqn: towards cotensor product of morphisms}) and (\ref{eqn: towards cotensor product of morphisms 2}).
    \begin{align}
        \label{eqn: towards cotensor product of morphisms}&V\square^C X \xrightarrow{i_{V,X}} V\otimes X \xrightarrow{f\otimes g}
\begin{tikzcd}[ampersand replacement=\&]
	{W\ot Y} \&\& {W\ot C\ot Y,}
	\arrow["{\rho^R_W\ot Y}", shift left=2, from=1-1, to=1-3]
	\arrow["{W\ot \rho^L_Y}"', shift right=2, from=1-1, to=1-3]
\end{tikzcd}
        \\&\label{eqn: towards cotensor product of morphisms 2} V\square^C X \xrightarrow{i_{V,X}} 
\begin{tikzcd}[ampersand replacement=\&]
	{V\ot X} \&\& {V\ot C\ot X}
	\arrow["{\rho^R_V\ot X}", shift left=2, from=1-1, to=1-3]
	\arrow["{V\ot \rho^L_X}"', shift right=2, from=1-1, to=1-3]
\end{tikzcd} \xrightarrow{f\otimes C\ot g} W \ot C \ot Y.
    \end{align}
    Yet by definition of $V\square^C X$, the top and bottom compositions in (\ref{eqn: towards cotensor product of morphisms 2}) are equal, and hence so too are those in (\ref{eqn: towards cotensor product of morphisms}). Hence there is a unique $\Mm$-morphism $f\square^C g\colon V\square^C X \to W\square^C Y$ such that $(f\ot g)e_{V,X}=e_{W,Y}(f\square^{C}g)$.
By uniqueness of the construction of the cotensor product of $C$-bicomodule morphisms, one verifies the functoriality of the cotensor product immediately.
Moreover, since $\ot$ preserves equalizers, one has that $(^{C}\Mm^{C},\square^{C},C)$ is a monoidal category, considering $(C,\Delta, \Delta)$ as a $C$-bicomodule. In fact, given $M$ in $^{C}\Mm^{C}$, left and right unitors are given by $(\Lambda'_{M})^{-1}$ and $\Lambda_{M}^{-1}$ defined as before, respectively. Dually to the previous case, one can prove that $M\square^{C}(N\square^{C}P)\cong(M\square^{C}N)\square^{C}P$.
\end{invisible}
\end{proof}

\begin{remark}\label{rmk:forgetfulbimod}
Given $A$ in $\mathsf{Mon}(\Mm)$, it is known that the forgetful functor $U:{}_{A}\Mm_{A}\to\Mm$ creates equalizers (in fact, any limit)
, i.e. if $\Mm$ is a category with equalizers, also the category $_{A}\Mm_{A}$ has equalizers and these are constructed as in $\Mm$, and come equipped with the unique $A$-bimodule structure such that the canonical map becomes an $A$-bimodule morphism (see e.g.\ \cite[Corollary 2.4]{PareigisI} where this is done for $_{A}\Mm$). The functor $U$ does not create colimits in general, so one cannot say that $_{A}\Mm_{A}$ has coequalizers when $\Mm$ has them. This happens when $\Mm$ is coregular in which case the coregularity of the category $\Mm$ lifts to the monoidal category $(_{A}\Mm_{A},\ot_{A},A)$, see again e.g.\ \cite[page 329]{Schbimod} (or \cite[Theorem 1.12]{AMS} in the case of abelian monoidal categories). Dually, given $C$ in $\mathsf{Comon}(\Mm)$, the forgetful functor $U:{}^{C}\Mm^{C}\to\Mm$ creates colimits, which coincide with those in $\Mm$ equipped with the unique $H$-bicomodule structure rendering the canonical map a morphism thereof. This functor does not create limits in general, so one cannot say that $^{C}\Mm^{C}$ has equalizers when $\Mm$ has them. This happens when $\Mm$ is regular in which case the regularity of $\Mm$ lifts to the monoidal category $(^{C}\Mm^{C},\square^{C},C)$. These results are summarized in the following result, and we include some details of the proof for the sake of completeness.
\end{remark}

\begin{proposition}\label{prop: (co)regularity of bi(co)modules}
    Let $(\Mm, \ot, I)$ be a (strict) monoidal category. The following statements hold:
\begin{itemize}
    \item[1)] Let $A$ be an object in $\Mon(\Mm)$. If $\Mm$ has equalizers then $_{A}\Mm_{A}$ has equalizers. If $\Mm$ is coregular then $_{A}\Mm_{A}$ has coequalizers and $(_{A}\Mm_{A},\ot_{A},A)$ is coregular.
    \item[2)] Let $C$ be an object in $\Comon(\Mm)$. If $\Mm$ has coequalizers then $^{C}\Mm^{C}$ has coequalizers. If $\Mm$ is regular then $^{C}\Mm^{C}$ has equalizers and $(^{C}\Mm^{C},\square^{C},C)$ is regular.
\end{itemize}
\end{proposition}

\begin{proof}
    1) Let $A$ be an object in $\Mon(\Mm)$. The first statement was observed in Remark \ref{rmk:forgetfulbimod}. We assume $\Mm$ is coregular. Given parallel morphisms $\alpha,\beta:M\to N$ in $_{A}\Mm_{A}$, we denote its coequalizer in $\Mm$ by $(T,q)$. There is a canonical way of endowing $T$ with an $A$-bimodule structure, namely uniquely so that its canonical epimorphism $q$ is moreover a morphism of $A$-bimodules. One verifies immediately that $q\alpha^L_N$ coequalizes $A\ot \alpha$ and $A\ot \beta$, whence, as $\ot$ preserves coequalizers, there exists a unique $\Mm$-morphism $\alpha^L_T\colon A\ot T \to T$ satisfying $q\alpha^L_N = \alpha^L_T(A\ot q)$. Similarly, we find a unique $\Mm$-morphism $\alpha^R_T\colon T\ot A\to T$ satisfying $q\alpha^R_N = \alpha^R_T(q\ot A)$. One immediately verifies that $(T,\alpha^L_T, \alpha^R_T)$ becomes an object in $_{A}\Mm_{A}$, by precomposing with the correct canonical morphism of a coequalizer (since these are epimorphisms). By construction, this bimodule structure on $C$ makes $q$ a morphism in $_{A}\Mm_{A}$. Moreover, the uniqueness in $_{A}\Mm_{A}$ follows from the uniqueness in $\Mm$. Hence $_A\Mm_A$ admits coequalizers. 

Moreover, we show that $\ot_A$ preserves these. This means that for all $A$-bimodules $X$ and $Y$, 
    \[
    X\ot_A T \ot_A Y = \Coeq\left( 
\begin{tikzcd}
	{X\ot_A M\ot_A Y} && {X\ot_A N\ot_A Y}
	\arrow["{X\ot_A \alpha\ot_A Y}", shift left=1.5, from=1-1, to=1-3]
	\arrow["{X\ot_A\beta\ot_A Y}"', shift right=1.5, from=1-1, to=1-3]
\end{tikzcd} \right).
    \]
One verifies immediately that $X\ot_A q$ coequalizes $X\ot_A \alpha$ and $X\ot_A \beta$. For the universal property, assume we have an $A$-bilinear morphism $\gamma:X\ot_{A}N\to P$ coequalizing $X\ot_A\alpha$ and $X\ot_A\beta$.
Since $\ot$ preserves coequalizers, $X\ot T$ is the coequalizer of the pair of morphisms $X\ot \alpha$ and $X\ot\beta$ in $\Mm$. 
Therefore, given $\gamma q_{X,N}:X\ot N\to P$, there is a unique $A$-bilinear morphism $\psi\colon X\ot T\to P$ such that $\psi(X\ot q) = \gamma q_{X,N}$. This moreover coequalizes $\alpha^R_X\ot T$ and $X\ot \alpha^L_T$.
\begin{invisible}
In fact
\begin{align*}
    \psi(X\ot\alpha^L_T)(X\ot A\ot q) &= \psi(X\ot q)(X\ot \alpha^L_N)
    \\&=\gamma q_{X,N}(X\ot \alpha^L_N)
    \\&= \gamma q_{X,N}(\alpha^R_X\ot N)
    \\&= \psi(X\ot q)(\alpha^R_X\ot N)
    = \psi(\alpha^R_X\ot N)(X\ot A\ot q)
\end{align*}
and $X\ot A \ot q$ is an epimorphism.
\end{invisible}
Hence there is a unique $A$-bilinear morphism $\rho \colon X\ot_A T\to P$ such that $\psi = \rho q_{X,T}$. 
We have that
\[
    \rho(X\ot_A q)q_{X,N} = \rho q_{X,T}(X\ot q) = \psi(X\ot q) = \gamma q_{X,N}.
\]
Moreover, $\rho$ is unique for this property since $X\ot_A q$ is an epimorphism. Proving that $\ot_A$ preserves coequalizers on the left is completely analogously, and thus $(_A\Mm_A, \ot_A, A)$ is coregular.

    Part 2) follows by dual arguments.
\end{proof}
\begin{remark}\label{rmk:regularityotA}
In general, the regularity of $(\Mm,\ot,I)$ does not imply the regularity of $(_{A}\Mm_{A},\ot_{A},A)$. We know that $_{A}\Mm_{A}$ has equalizers but they may not be preserved by $\ot_{A}$. Given an object $X$ in $_{A}\Mm_{A}$, if the functors $-\ot_{A}X,X\ot_{A}-:{}_{A}\Mm_{A}\to{}_{A}\Mm_{A}$ preserve equalizers then $X$ is said to be left respectively right $A$-\textit{flat}, see e.g.\ \cite[\S 2.3]{Schbimod}. 

Dually, the coregularity of $(\Mm,\ot,I)$ does not imply the coregularity of $(^{C}\Mm^{C},\square^{C},C)$. We know that $^{C}\Mm^{C}$ has coequalizers but they may not be preserved by $\square^{C}$.  Given an object $X$ in $^{C}\Mm^{C}$, if the functors $-\square^{C}X,X\square^{C}-:{}^{C}\Mm^{C}\to{}^{C}\Mm^{C}$ preserve coequalizers then $X$ is said to be left respectively right $C$-\textit{coflat}.
\end{remark}
Once the monoidal category $(\Mm,\ot,I)$ has a (pre-)braiding, one can define the categories $\mathsf{Bimon}(\Mm)$ and $\mathsf{Hopf}(\Mm)$ of bimonoids and Hopf monoids in $\Mm$, as well as the categories $\mathsf{Mon}_{c}(\Mm)$ and $\mathsf{Comon}_{cc}(\Mm)$ of commutative monoids and cocommutative comonoids.

We recall that, given a braided monoidal category $(\Mm,\ot,I,\sigma)$, the categories $\mathsf{Mon}(\Mm)$ and $\mathsf{Comon}(\Mm)$ are monoidal with $\ot$ and $I$ (and the same constraints of $\Mm$). Moreover, we have the following equivalences of categories
\begin{align*}
    &\mathsf{Mon}(\mathsf{Comon}(\Mm))\cong\mathsf{Bimon}(\Mm)\cong\mathsf{Comon}(\mathsf{Mon}(\Mm))\\
    \mathsf{Comon}&(\mathsf{Comon}(\Mm))\cong\mathsf{Comon}_{cc}(\Mm),\qquad \mathsf{Mon}(\mathsf{Mon}(\Mm))\cong\mathsf{Mon}_{c}(\Mm),
\end{align*}
see e.g.\ \cite[page 12]{Aguiar}. 
\begin{remark}
Note that the monoidal categories $(\mathsf{Mon}(\Mm),\ot,I)$ and $(\mathsf{Comon}(\Mm),\ot,I)$ may fail to be braided, and the category $\mathsf{Bimon}(\Mm)$ may fail to be monoidal. The latter is monoidal if $\sigma$ acts like a symmetry on the objects underlying (co)monoids, i.e. $\sigma_{M,N}^{-1} = \sigma_{N,M}$ if $M$ and $N$ are (co)monoids. \begin{invisible}However, if $\sigma$ is a \textit{symmetry}, i.e.\ $\sigma^{-1}_{A,B}=\sigma_{B,A}$ for all objects $A,B$ in $\Mm$, then $\sigma_{A,B}$ is a morphism of monoids and comonoids.\end{invisible} It follows that both $(\mathsf{Mon}(\Mm),\ot,I,\sigma)$ and $(\mathsf{Comon}(\Mm),\ot,I,\sigma)$ are then symmetric monoidal categories. Iterating these results and applying the previous equivalences, one can deduce that $(\mathsf{Bimon}(\Mm),\ot,I,\sigma)$, $(\mathsf{Mon}_{c}(\Mm),\ot,I,\sigma)$ and $(\mathsf{Comon}_{cc}(\Mm),\ot,I,\sigma)$ are symmetric monoidal categories as well. Moreover, if $(\Mm,\ot,I,\sigma)$ is a symmetric monoidal category, then also $(\mathsf{Hopf}(\Mm),\ot,I,\sigma)$ is a symmetric monoidal category, see e.g.\ \cite[page 12]{Aguiar}.
\end{remark} 

It is known that, given a braided monoidal category $(\Mm,\ot,I,\sigma)$ and an object $H$ in $\mathsf{Bimon}(\Mm)$, the categories $(_{H}\Mm,\ot,I)$ and $(^{H}\Mm,\ot,I)$ are monoidal with the same constraints of $\Mm$, see e.g.\ \cite[Proposition 3.2.7, Proposition 3.2.9]{HecSch}. If the category $\Mm$ is symmetric then $\mathsf{Bimon}(\Mm)$ is monoidal and, by identifying $_{H}\Mm_{H}$ with $_{H\ot H^{\mathrm{op}}}\Mm$ and $^{H}\Mm^{H}$ with $^{H\ot H^{\mathrm{cop}}}\Mm$ where $H^{\mathrm{op}}=(H,m_{H}\sigma_{H,H},u_{H},\Delta_{H},\varepsilon_{H})$ and $H^{\mathrm{cop}}=(H,m_{H},u_{H},\sigma_{H,H}\Delta_{H},\varepsilon_{H})$, one obtains the monoidal structure of bi(co)modules using \cite[Proposition 3.2.7, Proposition 3.2.9]{HecSch}. In fact, these monoidal structures hold even when $\Mm$ is simply braided, as recalled in the following result.

\begin{proposition}\label{prop:bi(co)modmonoidalot}
    Let $(\Mm, \ot, I, a, l, r, \sigma)$ be a braided monoidal category and $H$ be an object in $\mathsf{Bimon}(\Mm)$. Then $(_H\Mm_H, \ot, I, a, l, r)$ and  $(^H\Mm^H, \ot, I, a, l, r)$ are monoidal categories. If $\Mm$ is (co)regular, then so are $({}_{H}{\Mm}_H, \ot, I)$ and $({}^{H}{\Mm}^H, \ot, I)$. 
\end{proposition}

\begin{proof}
If $(M,\alpha^L_M, \alpha^R_M)$ and $(N,\alpha^L_N, \alpha^R_N)$ are object in $_{H}\Mm_{H}$, one defines on $M\ot N$ the \emph{diagonal actions}
\begin{align*}
&\alpha^L_{M\ot N} = (\alpha^L_M \ot \alpha^L_N)(\id\ot\sigma_{H,M}\ot \id)(\Delta\ot\id\ot\id)\colon H\ot M\ot N \to M\ot N, 
\\ &\alpha^R_{M\ot N} = (\alpha^R_M\ot \alpha^R_N)(\id\ot\sigma_{N,H}\ot\id)(\id\ot\id\ot\Delta) \colon M\ot N\ot H\to M\ot N,
\end{align*}
\begin{invisible}
It is immediate to verify that this makes $M\ot N$ into an $H$-bimodule, and that for $f\colon M\to X$ and $g\colon N\to Y$ $H$-bimodule morphisms, $f\ot g\colon M\ot N\to X\ot Y$ becomes an $H$-bimodule morphism again. Hence $\ot\colon _H\Mm_H\times{} _H\Mm_H\to{} _H\Mm_H$ is a well-defined functor. It follows easily that 
\end{invisible}
and $(I,\ep\ot I, I\ot \ep)$ is the $H$-bimodule structure on $I$. 

Dually, given $(M,\rho^L_M, \rho^R_M)$ and $(N,\rho^L_N, \rho^R_N)$ in $^{H}\Mm^{H}$, one defines on $M\ot N$ the \emph{diagonal coactions}
\begin{align*}
&\rho^L_{M\ot N} =(m\ot\id\ot\id)(\id\ot\sigma_{M,H}\ot \id) (\rho^L_M \ot \rho^L_N)\colon M\ot N \to H\ot M\ot N, 
\\ &\rho^R_{M\ot N} = (\id\ot\id\ot m)(\id\ot\sigma_{H,N}\ot\id) (\rho^R_M\ot \rho^R_N)\colon M\ot N\to M\ot N\ot H,
\end{align*}
and $(I,u\ot I,I\ot u)$ is in $^{H}\Mm^{H}$. It is an easy verification that the constraints are $H$-bi(co)linear.

If $\Mm$ is regular then ${}_{H}{\Mm}_H$ and ${}^{H}{\Mm}^H$ admit equalizers by Proposition \ref{prop: (co)regularity of bi(co)modules}. It suffices to prove that $\ot$ preserves these. Consider a pair of morphisms $\alpha,\beta:M\to N$ in $_{H}\Mm_{H}$, and denote its equalizer in ${}_{H}{\Mm}_H$ as $(E, i)$. Recall that this is the equalizer in $\Mm$ equipped with the unique $H$-bimodule structure such that its canonical morphism is an $H$-bimodule morphism. Given $H$-bimodules $V,W$, we already know that the $H$-bimodule $V\ot E\ot W$ is the equalizer of the pair $(\id\ot\alpha\ot\id,\id\ot\beta\ot\id)$ in $\Mm$. The universal property follows by the one in $\Mm$, using that $\id\ot i\ot\id$ is a morphism in $_{H}\Mm_{H}$ and a monomorphism in $\Mm$. Similarly for $H$-bicomodules and coregularity. \qedhere
\end{proof}
Given a braided monoidal category $(\Mm, \ot, I,\sigma)$ and an object $H$ in $\mathsf{Bimon}(\Mm)$, an \emph{$H$-tetramodule} is a quintuple $(V,\alpha^L, \alpha^R, \rho^L, \rho^R)$ such that $(V,\alpha^L, \alpha^R)$ is an $H$-bimodule, $(V,\rho^L, \rho^R)$ is an $H$-bicomodule, and these are compatible in the sense that $\rho^L$ and $\rho^R$ are $H$-bimodule morphisms (or, equivalently, $\alpha^L$ and $\alpha^R$ are $H$-bicomodule morphisms). A morphism of $H$-tetramodules in $\Mm$ is a morphism in $_{H}\Mm_{H}$ which is also a morphism in $^{H}\Mm^{H}$. We denote the category of $H$-tetramodules in $\Mm$ by $^H_H\Mm^H_H$.
    We have that $H$ is in $\Mon({}^H\Mm^H)$ and $\Comon({}_H\Mm_H)$ for the (co)multiplication, and the following equivalences of categories
    ${}^H_H\Mm^H_H \cong {}_H({}^H\Mm^H)_H \cong  {}^H({}_H\Mm_H)^H$,
see e.g.\ \cite[Lemma 3.6]{Schauenburg}.
We recall the following result, in which the monoidal structures of bimodules and bicomodules given in Theorem \ref{theorem: bi(co)modules monoidal} are lifted to tetramodules:
\begin{proposition}[cf.\ {\cite[Theorem 1.1]{BespalovDrabant}}]\label{prop:monoidaltetramodules}
    Let $(\Mm, \ot, I,\sigma)$ be a braided monoidal category and $H$ be an object in $\mathsf{Bimon}(\Mm)$. If $\Mm$ is coregular respectively regular, then $(^H_H\Mm^H_H, \ot_H, H)$ respectively $(^H_H\Mm^H_H, \square^H, H)$ are monoidal categories, for $H$ equipped with the (co)multiplication. 
\end{proposition}

\begin{invisible}
\begin{proof}
    By Proposition \ref{prop:bi(co)modmonoidalot}, $({}_{H}{\Mm}_H, \ot, I)$ is regular with the same equalizers as $\Mm$. Since $H$ is in $\Comon({}_{H}{\Mm_H}, \ot, I)$, by Theorem \ref{theorem: bi(co)modules monoidal} the category $({}^{H}({}_{H}{\Mm}_H)^H, \square^H, (H,\Delta, \Delta))\cong ({}^{H}_{H}{\Mm}^H_H, \square^H, H)$ is a monoidal category. 
    Analogously, $({}^{H}{\Mm}^H, \ot, I)$ is coregular and $H$ is in $\Mon({}^{H}{\Mm}^H, \ot, I)$, whence $({}_{H}{({}^{H}\Mm^H})_H, \ot_H, (H,\mu,\mu)) \cong ({}^{H}_{H}{\Mm}_H^H, \ot_H, H)$ is monoidal.
\end{proof}
\end{invisible}
\begin{remark}\label{remark: constraints of tetramodules}
    For completeness, let us explicitly write down the constraints on ${}^H_H\Mm^H_H$ for both monoidal structures. Denote the constraints of $\Mm$ by $l,r,$ and $a$. If $(\Mm,\ot,I)$ is regular, then so is $({}_H\Mm_H, \ot, I)$ by Proposition \ref{prop:bi(co)modmonoidalot}, with the same equalizers and constraints as in $\Mm$. Since ${}_H\Mm_H$ has the same equalizers, by Theorem \ref{theorem: bi(co)modules monoidal}, the constraints of $({}^H({}_H\Mm_H)^H\cong {}^H_H\Mm^H_H, \square^H, H)$ are $\Lambda_V^{-1}, (\Lambda_V')^{-1}$, and $a^{\square}_{V,W,U}$ for any $V,W,U$ in ${}^H_H\Mm^H_H$, which are defined as the unique morphisms satisfying
    \begin{gather*}
        e_{V,H}\Lambda_V = \rho^R_V, \hspace{3em} e_{H,V}\Lambda'_V = \rho^L_V,
        \\ (V\ot e_{W,U})e_{V,W\square U}a^{\square}_{V,W,U} = a_{V,W,U}(e_{V,W}\ot U)e_{V\square W, U}.
    \end{gather*}
    Likewise, if $\Mm$ is coregular, then $({}_H({}^H\Mm^H)_H\cong {}^H_H\Mm^H_H, \ot_H, H)$ is monoidal with constraints $\Upsilon_M, \Upsilon'_M$, and $a^{\ot_H}_{M,N,P}$ for any $M,N,P$ in ${}_H^H\Mm^H_H$, which are defined as the unique morphisms satisfying
    \begin{gather*}
        \Upsilon_Mq_{M,H} = \alpha^R_M, \hspace{3em} \Upsilon'_{M}q_{H, M} = \alpha^L_M,
        \\a^{\ot_H}_{M,N,P}q_{M\ot_H N, P}(q_{M,N}\ot P) = q_{M,N\ot_H P}(M\ot q_{N,P})a_{M,N,P}.
    \end{gather*}
\end{remark}

\subsection{Pre-Cartier categories and bialgebras}

Here we recall pre-Cartier categories introduced in \cite{ABSW} as infinitesimal versions of a braided category, and their algebraic counterparts for the categories of modules and comodules, namely pre-Cartier bialgebras. First, we recall that a category is \textit{pre-additive} when the set of morphisms between any two objects is an abelian group and the composition of morphisms satisfies the distributive law. A pre-additive braided monoidal category is a braided monoidal category which is also pre-additive and the tensor functor is additive in each entry, i.e.\ the tensor product and the addition of morphisms satisfy the distributive laws.

\begin{definition}[{\cite[Definition 1.1]{ABSW}}]
A pre-additive braided monoidal category $(\Mm,\otimes,I,\sigma)$ is called \textit{pre-Cartier} if it is equipped with a natural transformation $t:\otimes\to\otimes$ such that the identities
\begin{align}
t_{X,Y\otimes Z}&=t_{X,Y}\otimes\id_{Z}+(\sigma^{-1}_{X,Y}\otimes\id_{Z})(\id_{Y}\otimes t_{X,Z})(\sigma_{X,Y}\otimes\id_{Z}),\label{eqn: inf braid 1} \\
t_{X\otimes Y,Z}&=\id_{X}\otimes t_{Y,Z}+(\id_{X}\otimes\sigma^{-1}_{Y,Z})(t_{X,Z}\otimes\id_{Y})(\id_{X}\otimes\sigma_{Y,Z}),\label{eqn: inf braid 2}
\end{align}
hold true for all objects $X,Y,Z$ in $\Mm$. In this case, $t$ is said to be an \textit{infinitesimal braiding} of $(\Mm,\otimes,I,\sigma)$. A pre-Cartier category $(\Mm,\otimes,I,\sigma,t)$ is called \textit{Cartier} if in addition
\[
\sigma_{X,Y}t_{X,Y}=t_{Y,X}\sigma_{X,Y}
\]
holds for all objects $X,Y$ in $\Mm$. 
\end{definition}

The notion of symmetric Cartier category recovers that of \textit{infinitesimal symmetric category} which goes back to \cite{Ca}. 

We now recall how infinitesimal braidings are described for the braided monoidal category of modules over a quasitriangular bialgebra. In the following, for an element $T=\sum_{i}{T^{i}\ot T_{i}}\in H\ot H$, we will adopt the standard notation $T_{12}=\sum_{i}T^{i}\ot T_{i}\ot 1=T\ot 1$, $T_{23}=\sum_{i}1\ot T^{i}\ot T_{i}=1\ot T$, $T_{13}=\sum_{i}T^{i}\ot1\ot T_{i}$, $T^{\mathrm{op}}=\tau_{H,H}(T)=\sum_iT_{i}\ot T^{i}$. We omit writing the summation, leaving it implicit.

\begin{definition}[{\cite{Drinfeld}}]\label{definition: quasitriangular}
A bialgebra $H$ is \textit{quasitriangular} if there is an invertible element $\Rr=\Rr^i \otimes \Rr_i \in H\otimes H$ \--- called the \textit{universal $\Rr$-matrix} or \textit{quasitriangular structure} \--- such that the following axioms hold:
\begin{align}
    \Delta^\mathrm{op}(\cdot)&=\Rr\Delta(\cdot)\Rr^{-1},\label{qtr1}\\
(\id_H\otimes\Delta)(\Rr)&=\Rr_{13}\Rr_{12},\label{qtr2}\\
(\Delta\otimes\id_H)(\Rr)&=\Rr_{13}\Rr_{23},\label{qtr3}
\end{align}
If in addition $\Rr^{-1}=\Rr^\op$, then $(H,\Rr)$ is called \textit{triangular}.
\end{definition}

We recall that $(\varepsilon\ot\id)(\Rr)=1_{H}=(\id\ot\varepsilon)(\Rr)$. Moreover, if $H$ is a Hopf algebra, the following equalities are satisfied $(S\ot\id)(\Rr)=\Rr^{-1}$, $(\id\ot S)(\Rr^{-1})=\Rr$, and hence $(S\ot S)(\Rr)=\Rr$, $(S\ot S)(\Rr^{-1})=\Rr^{-1}$, see e.g. \cite[Lemma 2.1.2]{Majid-book}. Quasitriangular structures produce solutions to the so-called \emph{quantum Yang--Baxter equation}:
\begin{equation}\label{eqn: qybe}
        \Rr_{12}\Rr_{13}\Rr_{23} = \Rr_{23}\Rr_{13}\Rr_{12},
\end{equation}
see e.g.\ \cite[Theorem VIII.2.4]{Kassel-book}.

It is known (see \cite[Theorem 9.2.4 and paragraph thereafter]{Majid-book}) that a bialgebra $H$ is quasitriangular if and only if the monoidal category ${}_H\mm$ of left $H$-modules is braided, with triangular structures corresponding to ${}_{H}\mm$ being symmetric, with braiding determined on objects $M,N\in{}_H\mm$ by
\begin{equation*}
    \sigma^\Rr_{M,N}\colon M\otimes N\rightarrow N\otimes M,\qquad
    m\otimes n\mapsto\Rr^\op\cdot(n\otimes m)=(\Rr_i\cdot n)\otimes(\Rr^i\cdot m).
\end{equation*}
Note that $(\sigma^\Rr_{M,N})^{-1}( n\otimes m)= \Rr^{-1}\cdot (m\otimes n)=(\overline{\Rr}^i\cdot m)\otimes (\overline{\Rr}_{i}\cdot n)$, where $\Rr^{-1}=\overline{\Rr}^{i}\otimes\overline{\Rr}_{i}$.

\begin{definition}[{\cite[Definition 2.1]{ABSW}}]
A triple $(H,\Rr, \chi)$ is called a \emph{pre-Cartier quasitriangular bialgebra} if $(H,\Rr)$ is a quasitriangular bialgebra and $\chi\in H\otimes H$ such that
\begin{align}
    \chi\Delta(\cdot)&=\Delta(\cdot)\chi,\label{cqtr1}\\
    (\id_H\otimes\Delta)(\chi)&=\chi_{12}+\Rr^{-1}_{12}\chi_{13}\Rr_{12},\label{cqtr2}\\
    (\Delta\otimes\id_H)
    (\chi)&=\chi_{23}+\Rr^{-1}_{23}\chi_{13}\Rr_{23}.\label{cqtr3}
\end{align}
The element $\chi$ is called an \emph{infinitesimal $\Rr$-matrix}. A quasitriangular bialgebra $(H,\Rr)$ is \emph{Cartier} if it is pre-Cartier and the corresponding infinitesimal $\Rr$-matrix $\chi$ satisfies in addition that
\begin{equation}\label{eq:ctr2}
    \mathcal{R}\chi=\chi^\mathrm{op}\mathcal{R}.
\end{equation}
\end{definition}
By \cite[Remark 2.2, (iii)]{ABSW} the following equalities
\[
(\id\otimes \varepsilon)(\chi)=0=(\varepsilon\otimes\id)(\chi)
\]
hold for every infinitesimal $\Rr$-matrix $\chi$. 
\begin{invisible}
We recall that, in \cite[Theorem 2.21]{ABSW}, it is proven that any infinitesimal $\Rr$-matrix is also a $2$-cocycle in the cohomology for coalgebras defined as in \cite[XVIII.5]{Kassel-book}. More precisely, $\chi\in H\otimes H$ is such that
\begin{equation}\label{eq:Hoch2cocy}
\chi_{12}+(\Delta\otimes\id)(\chi)=\chi_{23}+(\id\otimes\Delta)(\chi).
\end{equation}
\end{invisible}
In \cite[Theorem 2.6]{ABSW} it is shown that a quasitriangular bialgebra $(H,\Rr)$ is pre-Cartier (resp.\ Cartier) if, and only if, the category $(_{H}\mm,\ot,\Bbbk)$ is braided pre-Cartier (resp.\ Cartier). The same bijective correspondence holds if one considers pre-Cartier (resp.\ Cartier) triangular bialgebras $H$ and symmetric pre-Cartier (resp.\ Cartier) structures on $({}_H\mm,\ot,\Bbbk)$. The corresponding infinitesimal braiding on $({}_H\mm,\ot,\Bbbk)$ is given for all objects $M,N$ in ${}_H\mm$ by
\begin{equation}\label{t-chi}
    t_{M,N}\colon M\otimes N\rightarrow M\otimes N,\qquad
    m\otimes n\mapsto\chi\cdot(m\otimes n)=(\chi^i\cdot m)\otimes(\chi_i\cdot n),
\end{equation}
where $\chi=\chi^i\otimes\chi_i\in H\otimes H$ is the infinitesimal $\Rr$-matrix for $H$. 
\begin{invisible}
We point out that the description of $t$ given in \eqref{t-chi} is obtained by using the naturality of the infinitesimal braiding. Indeed, for any $M$ in ${}_{H}\mm$ and $m\in M$, we can consider a map $l_{m}\colon H\to M$ in $_{H}\mm$ by setting $l_{m}(h)=h\cdot m$, so that, defining $\chi=t_{H,H}(1_{H}\ot1_{H})$, we have
\[
t_{M,N}(m\ot n)=t_{M,N}(l_{m}\ot l_{n})(1_{H}\ot1_{H})=(l_{m}\ot l_{n})t_{H,H}(1_{H}\ot1_{H})=(l_{m}\ot l_{n})(\chi^{i}\ot\chi_{i})=(\chi^{i}\cdot m)\ot(\chi_{i}\cdot n). 
\]
\end{invisible}

A similar characterization can be given for the braided monoidal category of comodules over a coquasitriangular bialgebra. 

\begin{definition}[{\cite[Definition 2.1]{Larson1991}}]
    Let $H$ be a bialgebra, then a \emph{universal $R$-form} of $H$ is a linear map $\Rr\colon H\ot H \to \Bbbk$ which is invertible with respect to the convolution product, and satisfying, for all $a,b,c\in H$, the following identities:
    \begin{align}
        &\tag{CQT1} \Rr(c\ot ab) = \Rr(c_1\ot b)\Rr(c_2\ot a),\label{CQT1}
        \\&\tag{CQT2} \Rr(ab\ot c)= \Rr(a\ot c_1)\Rr(b\ot c_2),\label{CQT2}
        \\&\tag{CQT3} \Rr(a_1\ot b_1)a_2b_2 = \Rr(a_2\ot b_2)b_1a_1.\label{CQT3}
    \end{align}
    If so, we call $(H,\Rr)$ a coquasitriangular bialgebra. If in addition $\Rr^{-1}=\Rr\tau$, then $(H,\Rr)$ is called \textit{cotriangular}.
\end{definition}
If $(H,\Rr)$ is a coquasitriangular bialgebra, then $\Rr^{-1}$, the convolution inverse of $\Rr$, satisfies for all $a,b,c\in H$
\begin{align}
    &\tag{CQT1'} \Rr^{-1}(c\ot ab) = \Rr^{-1}(c_1\ot a)\Rr^{-1}(c_2\ot b),\label{CQT1'}
    \\&\tag{CQT2'} \Rr^{-1}(ab\ot c) = \Rr^{-1}(b\ot c_1)\Rr^{-1}(a\ot c_2),\label{CQT2'}
    \\&\tag{CQT3'} a_1b_1\Rr^{-1}(a_2\ot b_2) = \Rr^{-1}(a_1\ot b_1)b_2a_2.\label{CQT3'}
\end{align}

Given a bialgebra $H$, there is a bijection between the braidings $\sigma$ on $(\mm^H,\ot, \Bbbk)$ and universal $R$-forms on $H$. Explicitly, a universal $R$-form $\Rr$ uniquely determines a braiding on right $H$-comodules $M$ and $N$ via
   \[
    \sigma^{\Rr}_{M,N}( m\ot n)= n_{(0)}\ot m_{(0)}\Rr(m_{(1)}\ot n_{(1)}),
    \]
see 
\cite[Exercise 9.2.9]{Majid-book}.

We also recall the dual setting for infinitesimal structures.

\begin{definition}[{\cite[Definition 3.1]{ABSW}}]
A triple $(H,\Rr,\chi)$ is called a \textit{pre-Cartier coquasitriangular bialgebra} if $(H,\Rr)$ is a coquasitriangular bialgebra and $\chi\colon H\ot H\to\Bbbk$ is such that for all $a,b,c\in H$
\begin{align}
    &\label{pcct1}\chi(a_{1}\ot b_{1})a_{2}b_{2}=a_{1}b_{1}\chi(a_{2}\ot b_{2}),
    \\&\label{pcct2}\chi(a\ot bc)=\chi(a\ot b)\varepsilon(c)+\Rr^{-1}(a_{1}\ot b_{1})\chi(a_{2}\ot c)\Rr(a_{3}\ot b_{2}),
    \\&\label{pcct3}\chi(ab\ot c)=\varepsilon(a)\chi(b\ot c)+\Rr^{-1}(b_{1}\ot c_{1})\chi(a\ot c_{2})\Rr(b_{2}\ot c_{3}).
\end{align}
The element $\chi$ is called an \emph{infinitesimal $\Rr$-form}.
We say that $(H,\Rr,\chi)$ is Cartier if in addition it satisfies 
\[
\Rr(a_{1}\ot b_{1})\chi(a_{2}\ot b_{2})=\chi(b_{1}\ot a_{1})\Rr(a_{2}\ot b_{2}).
\]
\end{definition}

As proven in \cite[Theorem 3.2]{ABSW}, given a coquasitriangular bialgebra $(H,\Rr)$, there is a bijection between pre-Cartier structures of $(H,\Rr)$ and pre-Cartier structures of $(\mm^{H},\ot, \sigma^{\Rr})$. The corresponding infinitesimal braiding on $\mm^{H}$ is defined for all objects $M,N$ in $\mm^{H}$ by
\[
t_{M,N}\colon M\ot N\to M\ot N,\ m\ot n\mapsto m_{(0)}\ot n_{(0)}\chi(m_{(1)}\ot n_{(1)}),
\]
where $\chi\colon H\ot H\to\Bbbk$ is the infinitesimal $\Rr$-form for $H$. Moreover, there is a bijective correspondence between Cartier structures of $(H,\Rr)$ and Cartier structures of $(\mm^{H},\ot,\sigma^{\Rr})$.

\section{Braidings and infinitesimal braidings for the category of bi(co)modules}\label{sec:braidingbicomodules}
In this section, we work over the coregular and regular monoidal category $(\Vec_{\Bbbk},\ot_{\Bbbk},\Bbbk)$ of vector spaces over $\Bbbk$, and we let $A$ and $C$ be an algebra and a coalgebra therein, respectively.

We briefly recall the classification of braidings for the monoidal category $({}_{A}{\mm}_{A},\otimes_{A},A)$, obtained in \cite{agore2014braidings}. Here we adopt a simple juxtaposition for the actions.

\begin{theorem}[{\cite[Theorem 3.1]{agore2014braidings}}]\label{thm:braidingsbimodtensorA}
    Let $A$ be an algebra. Then there is a bijective correspondence between the class of all braidings $\sigma$ on $({}_{A}{\mm}_{A},\otimes_{A},A)$ and the set of all invertible elements $R=R^{1}\otimes R^{2}\otimes R^{3}\in A\otimes A\otimes A$ satisfying the following conditions, for all $a\in A$:
\begin{enumerate}
    \item[1)] $R^{1}\otimes R^{2}\otimes aR^{3}=R^{1}a\otimes R^{2}\otimes R^{3}$, 
    \item[2)] $aR^{1}\otimes R^{2}\otimes R^{3}=R^{1}\otimes R^{2}a\otimes R^{3}$,
    \item[3)] $R^{1}\otimes aR^{2}\otimes R^{3}=R^{1}\otimes R^{2}\otimes R^{3}a$, 
    \item[4)] $R^{1}\otimes R^{2}\otimes1\otimes R^{3}=r^{1}R^{1}\otimes r^{2}\otimes r^{3}R^{2}\otimes R^{3}$, 
    \item[5)] $R^{1}\otimes1\otimes R^{2}\otimes R^{3}=R^{1}\otimes R^{2}r^{1}\otimes r^{2}\otimes R^{3}r^{3}$,
\end{enumerate}
where $r=r^{1}\otimes r^{2}\otimes r^{3}=R$. Under the above correspondence, the braiding $\sigma$ corresponding to $R$ is given by the formula
\[
\sigma_{M,N}\colon M\otimes_{A}N\to N\otimes_{A}M,\ m\otimes_{A}n\mapsto R^{1}nR^{2}\otimes_{A}mR^{3},
\]
for all $M,N$ in ${}_{A}{\mm}_{A}$.
\end{theorem}

An invertible element $R\in A\ot A\ot A$ satisfying the previous equations is called a \textit{canonical $R$-matrix} of $A$. We also recall that the classification of the canonical $R$-matrices can be simplified.

\begin{theorem}[{\cite[Theorem 3.2]{agore2014braidings}}]\label{thm:semplifiedconditions}
    Let $A$ be an algebra, then there is a bijection between the set of canonical $R$-matrices of $A$ and the set of all elements $R\in A\otimes A\otimes A$ satisfying, for any $a\in A$, 
\begin{equation}\label{cond1braiding}
R^{1}\otimes aR^{2}\otimes R^{3}=R^{1}\otimes R^{2}\otimes R^{3}a
\end{equation}  
and the normalizing condition
\begin{equation}\label{cond2braiding}
R^{1}R^{2}\otimes R^{3}=R^{2}\otimes R^{3}R^{1}=1\otimes1.
\end{equation}
Furthermore, $R$ is invariant under cyclic permutation of the tensor factors, $R=R^{2}\otimes R^{3}\otimes R^{1}=R^{3}\otimes R^{1}\otimes R^{2}$,
and we have the additional normalizing condition
\begin{equation}\label{cond3braiding}
R^{1}\otimes R^{2}R^{3}=1\otimes1.
\end{equation}
In particular, every braiding on ${}_{A}{\mm}_{A}$ is a symmetry.
\end{theorem}

Moreover, there is no non-trivial infinitesimal braiding for the braided monoidal category of $A$-bimodules, with the tensor product $\ot_{A}$.

\begin{theorem}[{\cite[Theorem 3.2.11]{ASthesis}}]\label{noinfbraidbim}
    For any braiding $\sigma$ on the monoidal category $({}_{A}{\mm}_{A},\ot_{A},A)$ there is no non-trivial infinitesimal braiding.
\end{theorem}

The purpose of this section is to achieve the corresponding result in the dual setting. These two will be put together in the duoidal setting of Section \ref{sec:duoidal}. The proofs of those results are moved to Appendix \ref{sec:appendix}, this to improve readability and because they are dual to the results of \cite{agore2014braidings}.

We consider the monoidal category $(^{C}\mm^{C},\square^{C},C)$. Given $V,W$ in $^{C}\mm^{C}$, we recall that $V\square^C W$ contains elements $v\ot w\in V\ot W$ such that $v_{(0)}\otimes v_{(1)}\ot w = v\ot w_{(-1)}\otimes w_{(0)}$. We will denote them by $v\square w$. Moreover, given morphisms $f\colon V\to W$ and $g\colon X\to Y$ in ${}^C\mm^C$, then $f\square^C g=(f\ot g)\vert_{V\square^C X}$.

\begin{remark}\label{lemma: left right unitors for vec}
As recalled in the preliminaries, there exists a $C$-bicomodule isomorphism $C\square^C C \cong C$. More precisely, this is given by $\Delta\colon C\to C\square^{C}C\colon x\mapsto x_{1}\square x_{2}$ whose inverse is given by $x\square y\mapsto \varepsilon(x)y=\varepsilon(y)x$. 
\begin{invisible}
Notice that $\Delta$ has image in $C\square^C C = \ker(\id\ot\Delta-\Delta\ot\id)$ by coassociativity. This is moreover easily verified to be a $C$-bicomodule morphism. Conversely, the restrictions of $\ep\ot\id$ and $\id\ot\ep$ to $C\square^C C$ coincide, and are inverse to $\Delta$ (as linear maps), whence they are the inverse $C$-bicomodule morphism.
\end{invisible}
In particular, for all $C$-bicomodules $V,W$, $\id_{V}\ot\Delta\ot\id_{W}$ is an isomorphism of $C$-bicomodules $V\ot C\ot W\cong(V\ot C)\square^C (C\ot W)$, since $\ot$ preserves equalizers.
\end{remark}
\begin{invisible}
Let $v\square x\in V\square^C X$, then $v_{(0)}\ot v_{(1)} \ot x = v \ot x_{(-1)}\ot x_{(0)}$ and
    \begin{align*}
        f(v)_{(0)}\ot f(v)_{(1)}\ot g(x) &= f(v_{(0)}) \ot v_{(1)}\ot g(x)
        \\&= f(v)\ot x_{(-1)}\ot g(x_{(0)})
        = f(v)\ot g(x)_{(-1)}\ot g(x)_{(0)}.
    \end{align*}
    Hence $(f\ot g)(V\square^C X)\subseteq W\square^C Y$. As $f\ot g$ is moreover a $^C\mm^C$-morphism its restriction to $V\square^C X$ equals $f\square^C g$.
\end{invisible}

\begin{invisible}
The next lemma provides a bicomodule morphism which will be useful in the classification of braidings on $({}^C\mm^C,\square^{C},C)$.

\begin{lemma}\label{lemma: induced bicomodule morphism}
    Let $V$ be a $C$-bicomodule, and $f\in V^*$. Define
    \begin{equation}\label{eqn: induced bimodule morphism}
        \hat{f}\colon V\to C\otimes C\colon v\mapsto v_{(-1)}\otimes f(v_{(0)})v_{(1)}.
    \end{equation}
    Then $(\ep\ot\ep)\hat{f} = f$ and $\hat{f}$ is a $C$-bicomodule morphism, considering $C\ot C$ in ${}^C\mm^{C}$ with coactions $\Delta\ot\id_{C}$ and $\id_{C}\ot\Delta$.
\end{lemma}
\begin{proof}
    The first claim is immediate by definition of left and right coactions. Let us verify that $\hat{f}$ is a morphism of $C$-bicomodules. 
By definition of the tensor $C$-bicomodule structure it follows that
    \begin{align*}
        \hat{f}(v)_{(-1)}\otimes \hat{f}(v)_{(0)} &= (v_{(-1)}\otimes v_{(1)})_{(-1)} \otimes (v_{(-1)}\otimes v_{(1)})_{(0)} f(v_{(0)})
        \\&= v_{(-1)1}\otimes v_{(-1)2}\otimes v_{(1)}f(v_{(0)})
        \\&= v_{(-1)}\otimes v_{(0)(-1)}\otimes v_{(0)(1)}f(v_{(0)(0)})
        \\&= v_{(-1)}\otimes \hat{f}(v_{(0)})
\end{align*}  
and, similarly, $\hat{f}(v)_{(0)}\otimes \hat{f}(v)_{(1)}=\hat{f}(v_{(0)})\ot v_{(1)}$.
\end{proof}
\end{invisible}

\begin{invisible}
The following remark is immediate, yet useful.

\begin{remark}\label{lemma: left right unitors for vec}
As recalled in the preliminaries, there exists a $C$-bicomodule isomorphism $C\square^C C \cong C$. More precisely, this is given by $\Delta\colon C\to C\square^{C}C\colon x\mapsto x_{1}\square x_{2}$ whose inverse is given by $x\square y\mapsto \varepsilon(x)y=\varepsilon(y)x$. 

Notice that $\Delta$ has image in $C\square^C C = \ker(\id\ot\Delta-\Delta\ot\id)$ by coassociativity. This is moreover easily verified to be a $C$-bicomodule morphism. Conversely, the restrictions of $\ep\ot\id$ and $\id\ot\ep$ to $C\square^C C$ coincide, and are inverse to $\Delta$ (as linear maps), whence they are the inverse $C$-bicomodule morphism.

In particular, for all $C$-bicomodules $V,W$, $\id_{V}\ot\Delta\ot\id_{W}$ is an isomorphism of $C$-bicomodules $V\ot C\ot W\cong(V\ot C)\square^C (C\ot W)$, since $\ot$ preserves equalizers.
\end{remark}
\end{invisible}


The next technical result will be useful for the classification of braidings for the category of bicomodules.

\begin{proposition}\label{prop: equivalent conditions canonical R-comatrix}
    Let $\RR\colon C\ot C \ot C\to \Bbbk$ be a linear map. Then the following are equivalent:
    \begin{enumerate}[label=\alph*)]
        \item $\RR$ is convolution-invertible and satisfies, for all $f\in C^*$, the following conditions:
        \begin{align}
            &(\ep\ot\ep\ot f)\ast \RR= \RR\ast(f\ot\ep\ot\ep),\label{cond: comatrix 1}
            \\&(f\ot\ep\ot\ep)\ast \RR = \RR\ast(\ep\ot f\ot\ep),\label{cond: comatrix 2}
            \\&(\ep\ot f\ot \ep) \ast \RR = \RR\ast(\ep\ot\ep\ot f),\label{cond: comatrix 3}
            \\&(\RR\ot\ep)(\id\ot \id\ot\tau) = (\RR\ot \ep)\ast (\ep\ot\RR)(\tau\ot\id\ot\id),\label{cond: comatrix 4}
            \\&(\ep\ot\RR)(\tau\ot\id\ot\id) = (\RR \ot\ep)(\id\ot\id\ot\tau)\ast (\ep\ot\RR),\label{cond: comatrix 5}
        \end{align}
        where $\tau$ denotes the usual flip isomorphism of vector spaces;
        \item $\RR$ satisfies condition (\ref{cond: comatrix 3}) and the \emph{normalizing conditions}
        \begin{equation}\label{cond: normalizing condition}
        \RR(x_1\ot x_2\ot y) = \ep(x)\ep(y) = \RR(y_2\ot x\ot y_1) \text{ for all }x,y\in C.
        \end{equation}
    \end{enumerate}
    In the latter equivalent cases, $\RR$ satisfies moreover the following respectively cyclic, normalizing, and invertibility conditions
    \begin{align}
        &\RR(p\ot q \ot r) = \RR(q\ot r\ot p) = \RR(r\ot p\ot q)\label{cond: cyclic conditions} \text{ for all }p,q,r\in C,
        \\&\RR(x\ot y_1\ot y_2) = \ep(x)\ep(y)\label{cond: third normalizing condition} \text{ for all }x,y\in C,
        \\&\RR^{-1} = \RR(\tau\ot\id)\label{cond: symmetry}.
    \end{align}
\end{proposition}
\begin{invisible}
\begin{proof}
Notice that conditions (\ref{cond: comatrix 1})--(\ref{cond: comatrix 5}) are one-by-one equivalent to stating that, for all $p,q,r,u\in C$, the following equalities hold: \ls{[Change numbering at the end to match!]}
\begin{align}
        &\RR(p\ot q\ot r_2)r_1 = \RR(p_1\ot q \ot r)p_2,\tag{30'}\label{cond: comatrix 1 prime}
        \\&\RR(p_2\ot q \ot r)p_1 =\RR(p\ot q_1\ot r)q_2,\tag{31'}\label{cond: comatrix 2 prime}
        \\&\RR(p\ot q_2\ot r)q_1 = \RR(p\ot q \ot r_1)r_2.\tag{32'}\label{cond: comatrix 3 prime}
        \\&\RR(p\ot q \ot u)\ep(r) = \RR(p_1\ot q\ot r_1)\RR(p_2\ot r_2\ot u),\tag{33'}\label{cond: comatrix 4 prime}
        \\&\RR(p\ot r\ot u)\ep(q) = \RR(p\ot q_1\ot u_1)\RR(q_2\ot r\ot u_2)\tag{34'}\label{cond: comatrix 5 prime}.
    \end{align}
    Assume $\RR$ to be convolution-invertible and satisfying conditions (\ref{cond: comatrix 1})--(\ref{cond: comatrix 5}). It suffices to verify the normalizing conditions (\ref{cond: normalizing condition}). By condition (\ref{cond: comatrix 5}) it holds that
    \begin{align*}
        \RR &= (\ep\ot \RR)(\tau\ot\id\ot\id)(\id\ot\Delta\ot\id) 
        \\&= ((\RR\ot \ep)(\id\ot\id\ot\tau)\ast (\ep\ot\RR))(\id\ot\Delta\ot\id)
        \\\iff\ep^{\ot 3} &= \RR^{-1}\ast ((\RR\ot \ep)(\id\ot\id\ot\tau)\ast (\ep\ot\RR))(\id\ot\Delta\ot\id).
    \end{align*}
    Evaluating in an element $p\ot q\ot r\in C\ot C\ot C$ we obtain that
    \begin{align*}
        \ep(p)\ep(q)\ep(r) &= \RR^{-1}(p_1\ot q_1\ot r_1)\RR(p_2\ot q_2\ot r_2)\RR(q_3\ot q_4\ot r_3)
        \\&= \ep(p)\ep(q_1)\ep(r_1)\RR(q_2\ot q_3\ot r_2)
        \\&= \RR(q_1\ot q_2\ot r)\ep(p).
    \end{align*}
    For $p\ot q \ot r= x_2\ot y \ot x_1$ this implies that $\RR(y_1\ot y_2\ot x) = \ep(x)\ep(y)$.\newline
    Similarly, by condition (\ref{cond: comatrix 5})
    \begin{align*}
        \RR &= (\ep\ot \RR)(\tau\ot\id\ot\id)(\Delta\ot\id\ot\id) 
        \\&= ((\RR\ot \ep)(\id\ot\id\ot\tau)\ast (\ep\ot\RR))(\Delta\ot\id\ot\id)
        \\\iff\ep^{\ot 3} &= \RR^{-1}\ast ((\RR\ot \ep)(\id\ot\id\ot\tau)\ast (\ep\ot\RR))(\Delta\ot\id\ot\id).
    \end{align*}
    Evaluating on an element $p\ot q\ot r\in C\ot C\ot C$ we obtain that
    \begin{align*}
        \ep(p)\ep(q)\ep(r)&= \RR^{-1}(p_1\ot q_1\ot r_1)\RR(p_2\ot p_3\ot r_2)\RR(p_4\ot q_2\ot r_3)
        \\&= \RR^{-1}(p_1\ot q_1\ot r_1)\RR(p_4\ot q_2\ot \RR(p_2\ot p_3\ot r_{21})r_{22})
        \\&\stackrel{(\ref{cond: comatrix 3 prime})}{=} \RR^{-1}(p_1\ot q_1\ot r_1)\RR(p_4\ot q_2\ot \RR(p_2\ot p_{32}\ot r_2)p_{31})
        \\&=\RR^{-1}(p_1\ot q_1\ot r_1)\RR(p_2\ot \RR(p_5\ot q_2\ot p_{3})p_{4}\ot r_2)
        \\&\stackrel{(\ref{cond: comatrix 3 prime})}{=} \RR^{-1}(p_1\ot q_1\ot r_1)\RR(p_2\ot q_{21}\ot r_2)\RR(p_4\ot q_{22}\ot p_3)
        \\&= \RR^{-1}(p_1\ot q_1\ot r_1)\RR(p_2\ot q_2\ot r_2)\RR(p_4\ot q_3\ot p_3)
        \\&= \ep(p_1)\ep(q_1)\ep(r)\RR(p_3\ot q_2\ot p_2)
        \\&=\varepsilon(r) \RR(p_2\ot q \ot p_1).
    \end{align*}
    Hence for $p\ot q \ot r = y\ot x_1\ot x_2$ it follows that $\RR(y_2\ot x\ot y_1) = \ep(x)\ep(y)$.
    Assume conversely that $\RR$ satisfies condition (\ref{cond: comatrix 3}) and the normalizing conditions (\ref{cond: normalizing condition}). We first prove that $\RR$ satisfies the cyclic conditions (\ref{cond: cyclic conditions}). For any $p,q,r\in C$ it holds that
    \begin{align*}
        \RR(r\ot p\ot q) &\stackrel{(\ref{cond: normalizing condition})}{=} \RR(p_1\ot p_2\ot r_1)\RR(r_2\ot p_3\ot q)
        \\&= \RR(p_1\ot p_{21}\ot r_1)\RR(r_2\ot p_{22}\ot q)
        \\&\stackrel{(\ref{cond: comatrix 3 prime})}{=} \RR(p_1\ot q_2\ot r_1)\RR(r_2\ot p_2\ot q_1)
        \\&\stackrel{(\ref{cond: comatrix 3 prime})}{=} \RR(p_1\ot q_{22}\ot r)\RR(q_{21}\ot p_2\ot q_1)
        \\&= \RR(p_1\ot q_3\ot r)\RR(q_2\ot p_2\ot q_1)
        \\&\stackrel{(\ref{cond: normalizing condition})}{=} \RR(p_1\ot q_2\ot r)\ep(q_1)\ep(p_2)
        = \RR(p\ot q\ot r).
    \end{align*}
    Hence $\RR$ satisfies the cyclic condition (\ref{cond: cyclic conditions}). From this the additional normalizing condition (\ref{cond: third normalizing condition}) immediately follows by combining \eqref{cond: normalizing condition} and \eqref{cond: cyclic conditions}. We now show that $\RR$ is a convolution-invertible satisfying conditions (\ref{cond: comatrix 1})--(\ref{cond: comatrix 5}), proving that a) and b) are equivalent. Conditions (\ref{cond: comatrix 1}) and (\ref{cond: comatrix 2}) follow immediately by applying \eqref{cond: cyclic conditions} to (\ref{cond: comatrix 3}). For condition (\ref{cond: comatrix 4}) we compute that
    \begin{align*}
        \RR(p_1\ot q\ot r_1)\RR(p_2\ot r_2\ot u)&\stackrel{(\ref{cond: comatrix 1 prime})}{=} \RR(p\ot q\ot r_{2})\RR(r_{1}\ot r_3\ot u)
        \\&\stackrel{(\ref{cond: comatrix 3 prime})}{=} \RR(p\ot q\ot u_2)\RR(r_1\ot r_2\ot u_1)
        \\&\stackrel{(\ref{cond: normalizing condition})}{=} \RR(p\ot q\ot u_2)\ep(r)\ep(u_{1})
        = \RR(p\ot q\ot u)\ep(r).
    \end{align*}
    Similarly, condition (\ref{cond: comatrix 5}) follows since
    \begin{align*}
        \RR(p\ot q_1\ot u_1)\RR(q_2\ot r\ot u_2) &\stackrel{(\ref{cond: comatrix 3 prime})}{=} \RR(p\ot q_{12}\ot u)\RR(q_2\ot r\ot q_{11})
        \\&= \RR(p\ot q_{21}\ot u)\RR(q_{22}\ot r\ot q_1)
        \\&\stackrel{(\ref{cond: comatrix 2 prime})}{=} \RR(p\ot r_2\ot u)\RR(q_2\ot r_1\ot q_1)
        \\&\stackrel{(\ref{cond: normalizing condition})}{=} \RR(p\ot r_2\ot u)\ep(q)\ep(r_1)
        = \RR(p\ot r\ot u)\ep(q).
    \end{align*}
    It remains to prove the invertibility of $\RR$ with respect to the convolution. We claim $\RR(\tau\ot\id)$ to be the inverse of $\RR$. Let $p,q,r\in C$, then
    \begin{align*}
        (\RR\ast \RR(\tau\ot\id))(p\ot q \ot r) &= \RR(p_1\ot q_1\ot r_1)\RR(q_2\ot p_2\ot r_2)
        \\&\stackrel{(\ref{cond: cyclic conditions})}{=} \RR(p_1\ot q_1\ot r_1)\RR(p_2\ot r_2\ot q_2)
        \\&\stackrel{(\ref{cond: comatrix 4 prime})}{=}
        \RR(p\ot q_1\ot q_2)\ep(r)
        \stackrel{(\ref{cond: third normalizing condition})}{=} \ep(p)\ep(q)\ep(r),
        \\(\RR(\tau\ot\id)\ast \RR)(p\ot q\ot r)&= \RR(q_1\ot p_1\ot r_1)\RR(p_2\ot q_2\ot r_2)
        \\&\stackrel{(\ref{cond: comatrix 5 prime})}{=} \RR(q_1\ot q_2\ot r)\ep(p)
        \stackrel{(\ref{cond: normalizing condition})}{=} \ep(p)\ep(q)\ep(r).\qedhere
    \end{align*}
\end{proof}
\end{invisible}
\begin{remark}
Notice that conditions (\ref{cond: comatrix 1})--(\ref{cond: comatrix 5}) are one-by-one equivalent to stating that, for all $p,q,r,u\in C$, the following equalities hold: 
\begin{align}
        &\RR(p\ot q\ot r_2)r_1 = \RR(p_1\ot q \ot r)p_2,\tag{20'}\label{cond: comatrix 1 prime}
        \\&\RR(p_2\ot q \ot r)p_1 =\RR(p\ot q_1\ot r)q_2,\tag{21'}\label{cond: comatrix 2 prime}
        \\&\RR(p\ot q_2\ot r)q_1 = \RR(p\ot q \ot r_1)r_2.\tag{22'}\label{cond: comatrix 3 prime}
        \\&\RR(p\ot q \ot u)\ep(r) = \RR(p_1\ot q\ot r_1)\RR(p_2\ot r_2\ot u),\tag{23'}\label{cond: comatrix 4 prime}
        \\&\RR(p\ot r\ot u)\ep(q) = \RR(p\ot q_1\ot u_1)\RR(q_2\ot r\ot u_2)\tag{24'}\label{cond: comatrix 5 prime}.
    \end{align}
\end{remark}

\begin{definition}
    Let $C$ be a coalgebra. A linear map $\RR\colon C\ot C\ot C \to \Bbbk$ satisfying the equivalent conditions of Proposition \ref{prop: equivalent conditions canonical R-comatrix} is a \emph{canonical $R$-form} of $C$.
\end{definition}

We will study the uses of canonical $R$-forms, by studying the following two closely related morphisms for any $C$-bicomodules $M,N$. 
\begin{align*}
    &\Upsilon_{M,N}:M\ot N\longrightarrow N\ot M, \quad m\ot n\mapsto \RR(m_{(-1)}\ot m_{(1)}\ot n_{(-1)})n_{(0)}\ot m_{(0)},
    \\&\tilde{\sigma}_{M,N}:M\ot N\longrightarrow N\ot M, \quad m\ot n \mapsto \RR(m_{(-1)}\ot m_{(1)}\ot n_{(1)})n_{(0)}\ot m_{(0)}.
\end{align*}
\begin{invisible}
$\Upsilon_{M,N}\vert_{M\square N}$ has the property that
    \begin{align*}
        \Upsilon_{M,N}(m\square n) &= \RR(m_{(-1)}\ot m_{(1)}\ot n_{(-1)})n_{(0)}\ot m_{(0)}
        \\&= \RR(m_{(0)(-1)}\ot m_{(0)(1)}\ot m_{(1)})n\ot m_{(0)(0)}
        \\&= \RR(m_{(-1)}\ot m_{(1)1}\ot m_{(1)2})n\ot m_{(0)}
        \\&\overset{(\ref{cond: third normalizing condition})}{=} \ep(m_{(-1)})\ep(m_{(1)})n\ot m_{(0)}
        \\&= n\ot m.
    \end{align*}
    Hence $\Upsilon_{M,N}$ induces a morphism $M\square N\to N\square M$ if and only if $\tau$ is an isomorphism from $M\square N$ to $N\square M$.
\end{invisible}
A first use of canonical $R$-forms is that they induce solutions to the quantum Yang--Baxter equation as well as the braid equation. The following result is dual to \cite[Theorem 3.13]{agore2014braidings}, 
where $\Upsilon=\Upsilon_{M,M}$ for a given $C$-bicomodule $M$. In the next result, we set $\Upsilon^{12}=\Upsilon\ot\id$, $\Upsilon^{23}=\id\ot\Upsilon$ and $\Upsilon^{13}=(\id\ot\tau)(\Upsilon\ot\id)(\id\ot\tau)$.

\begin{proposition}\label{prop:solutionYangBaxter}
    Let $C$ be a coalgebra and $\RR$ be a canonical $R$-form of $C$. For any $C$-bicomodule $M$ the map
    \begin{equation}
        \Upsilon \colon M\ot M\to M\ot M,\ m\ot n\mapsto \RR(m_{(-1)}\ot m_{(1)}\ot n_{(-1)})n_{(0)}\ot m_{(0)}
    \end{equation}
    is a solution to the quantum Yang--Baxter equation, this is $\Upsilon^{12}\Upsilon^{13}\Upsilon^{23} = \Upsilon^{23}\Upsilon^{13}\Upsilon^{12}$, as well as the braid equation $\Upsilon^{12}\Upsilon^{23}\Upsilon^{12} = \Upsilon^{23}\Upsilon^{12}\Upsilon^{23}$. Moreover, we have 
    $\Upsilon\Upsilon \Upsilon = \Upsilon$.
\end{proposition}
\begin{invisible}
\begin{proof}
    Let $M$ be a fixed $C$-bicomodule, and $p,q,r\in M$. We compute that
    \begin{align*}
        \Upsilon^{12}\Upsilon^{13}\Upsilon^{23}(p\ot q\ot r) 
        &= \RR(q_{(-1)1}\ot q_{(1)2}\ot r_{(-1)1})\RR(p_{(-1)}\ot p_{(1)}\ot q_{(-1)2})
        \\&\hspace{12em}\RR(q_{(-1)3}\ot q_{(1)1}\ot r_{(-1)2})r_{(0)}\ot q_{(0)}\ot p_{(0)}
        \\&= \RR(q_{(-1)11}\ot q_{(1)2}\ot r_{(-1)1})\RR(p_{(-1)}\ot p_{(1)}\ot q_{(-1)12})
        \\&\hspace{12em}\RR(q_{(-1)2}\ot q_{(1)1}\ot r_{(-1)2})r_{(0)}\ot q_{(0)}\ot p_{(0)}
        \\&\overset{(\ref{cond: comatrix 1 prime})}{=} \RR(p_{(-1)2}\ot q_{(1)2}\ot r_{(-1)1})\RR(p_{(-1)1}\ot p_{(1)}\ot q_{(-1)1})
        \\&\hspace{12em}\RR(q_{(-1)2}\ot q_{(1)1}\ot r_{(-1)2})r_{(0)}\ot q_{(0)}\ot p_{(0)}
        \\&\overset{(\ref{cond: comatrix 1 prime})}{=} \RR(p_{(-1)2}\ot q_{(1)2}\ot q_{(-1)22})\RR(p_{(-1)1}\ot p_{(1)}\ot q_{(-1)1})
        \\&\hspace{12em}\RR(q_{(-1)21}\ot q_{(1)1}\ot r_{(-1)})r_{(0)}\ot q_{(0)}\ot p_{(0)}
        \\&\overset{(\ref{cond: comatrix 2 prime})}{=}\RR(p_{(-1)2}\ot q_{(-1)21}\ot q_{(-1)3})\RR(p_{(-1)1}\ot p_{(1)}\ot q_{(-1)1})
        \\&\hspace{12em}\RR(q_{(-1)22}\ot q_{(1)}\ot r_{(-1)})r_{(0)}\ot q_{(0)}\ot p_{(0)}
        \\&=\RR(p_{(-1)2}\ot q_{(-1)12}\ot q_{(-1)3})\RR(p_{(-1)1}\ot p_{(1)}\ot q_{(-1)11})
        \\&\hspace{12em}\RR(q_{(-1)2}\ot q_{(1)}\ot r_{(-1)})r_{(0)}\ot q_{(0)}\ot p_{(0)}
        \\&\overset{(\ref{cond: comatrix 3 prime})}{=} \RR(p_{(-1)2}\ot p_{(1)1}\ot q_{(-1)3})\RR(p_{(-1)1}\ot p_{(1)2}\ot q_{(-1)1})
        \\&\hspace{12em}\RR(q_{(-1)2}\ot q_{(1)}\ot r_{(-1)})r_{(0)}\ot q_{(0)}\ot p_{(0)}
        \\&= \Upsilon^{23}\Upsilon^{13}\Upsilon^{12}(p\ot q\ot r).
    \end{align*}
    Hence the quantum Yang--Baxter equation holds. Analogously, we prove that $\Upsilon$ satisfies the braid equation. Indeed, on the one hand, we get
    \begin{align*}
    \Upsilon^{12}&\Upsilon^{23}\Upsilon^{12}(p\ot q\ot r)
    \\&= \RR(p_{(-1)1}\ot p_{(1)2}\ot q_{(-1)1})\RR(p_{(-1)2}\ot p_{(1)1}\ot r_{(-1)1})
    \\&\hspace{12em}\RR(q_{(-1)2}\ot q_{(1)}\ot r_{(-1)2})r_{(0)}\ot q_{(0)}\ot p_{(0)}
    \\&\overset{(\ref{cond: comatrix 2 prime})}{=} \RR(p_{(-1)1}\ot p_{(-1)21}\ot q_{(-1)1})\RR(p_{(-1)22}\ot p_{(1)}\ot r_{(-1)1})
    \\&\hspace{12em}\RR(q_{(-1)2}\ot q_{(1)}\ot r_{(-1)2})r_{(0)}\ot q_{(0)}\ot p_{(0)}
    \\&= \RR(p_{(-1)11}\ot p_{(-1)12}\ot q_{(-1)1})\RR(p_{(-1)2}\ot p_{(1)}\ot r_{(-1)1})
    \\&\hspace{12em}\RR(q_{(-1)2}\ot q_{(1)}\ot r_{(-1)2})r_{(0)}\ot q_{(0)}\ot p_{(0)}
    \\&\overset{(\ref{cond: normalizing condition})}{=} \ep(p_{(-1)1})\ep(q_{(-1)1})\RR(p_{(-1)2}\ot p_{(1)}\ot r_{(-1)1})\RR(q_{(-1)2}\ot q_{(1)}\ot r_{(-1)2})r_{(0)}\ot q_{(0)}\ot p_{(0)}
    \\&= \RR(p_{(-1)}\ot p_{(1)}\ot r_{(-1)1})\RR(q_{(-1)}\ot q_{(1)}\ot r_{(-1)2})r_{(0)}\ot q_{(0)}\ot p_{(0)}.
    \end{align*}
    On the other hand,
    \begin{align*}
        \Upsilon^{23}\Upsilon^{12}\Upsilon^{23}(p\ot q\ot r)
        &=\RR(q_{(-1)1}\ot q_{(1)}\ot r_{(-1)1})\RR(p_{(-1)1}\ot p_{(1)2}\ot r_{(-1)2})
        \\&\hspace{12em} \RR(p_{(-1)2}\ot p_{(1)1}\ot q_{(-1)2})r_{(0)}\ot q_{(0)}\ot p_{(0)}
        \\&\overset{(\ref{cond: comatrix 2 prime})}{=} \RR(q_{(-1)1}\ot q_{(1)}\ot r_{(-1)1})\RR(p_{(-1)1}\ot p_{(-1)21}\ot r_{(-1)2})
        \\&\hspace{12em} \RR(p_{(-1)22}\ot p_{(1)}\ot q_{(-1)2})r_{(0)}\ot q_{(0)}\ot p_{(0)}
        \\&= \RR(q_{(-1)1}\ot q_{(1)}\ot r_{(-1)1})\RR(p_{(-1)11}\ot p_{(-1)12}\ot r_{(-1)2})
        \\&\hspace{12em} \RR(p_{(-1)2}\ot p_{(1)}\ot q_{(-1)2})r_{(0)}\ot q_{(0)}\ot p_{(0)}
        \\&\overset{(\ref{cond: normalizing condition})}{=} \ep(p_{(-1)1})\ep(r_{(-1)2})\RR(q_{(-1)1}\ot q_{(1)}\ot r_{(-1)1})
        \\&\hspace{12em} \RR(p_{(-1)2}\ot p_{(1)}\ot q_{(-1)2})r_{(0)}\ot q_{(0)}\ot p_{(0)}
        \\&= \RR(q_{(-1)1}\ot q_{(1)}\ot r_{(-1)})\RR(p_{(-1)}\ot p_{(1)}\ot q_{(-1)2})r_{(0)}\ot q_{(0)}\ot p_{(0)}
        \\&\overset{(\ref{cond: comatrix 1 prime})}{=} \RR(q_{(-1)}\ot q_{(1)}\ot r_{(-1)2})\RR(p_{(-1)}\ot p_{(1)}\ot r_{(-1)1})r_{(0)}\ot q_{(0)}\ot p_{(0)}.
    \end{align*}
    
    Finally, we compute
    \begin{align*}
        (\Upsilon&\circ\Upsilon\circ\Upsilon)(p\ot q)
        \\&= \RR(p_{(-1)1}\ot p_{(1)2}\ot q_{(-1)1})\RR(q_{(-1)2}\ot q_{(1)}\ot p_{(-1)2})\RR(p_{(-1)3}\ot p_{(1)1}\ot q_{(-1)3})q_{(0)}\ot p_{(0)}
        \\&= \RR(p_{(-1)11}\ot p_{(1)2}\ot q_{(-1)1})\RR(q_{(-1)2}\ot q_{(1)}\ot p_{(-1)12})\RR(p_{(-1)2}\ot p_{(1)1}\ot q_{(-1)3})q_{(0)}\ot p_{(0)}
        \\&\overset{(\ref{cond: comatrix 1 prime})}{=} \RR(p_{(-1)1}\ot p_{(1)2}\ot q_{(-1)12})\RR(q_{(-1)2}\ot q_{(1)}\ot q_{(-1)11})\RR(p_{(-1)2}\ot p_{(1)1}\ot q_{(-1)3})q_{(0)}\ot p_{(0)}
        \\&= \RR(p_{(-1)1}\ot p_{(1)2}\ot q_{(-1)21})\RR(q_{(-1)22}\ot q_{(1)}\ot q_{(-1)1})\RR(p_{(-1)2}\ot p_{(1)1}\ot q_{(-1)3})q_{(0)}\ot p_{(0)}
        \\&\overset{(\ref{cond: comatrix 2 prime})}{=} \RR(p_{(-1)1}\ot p_{(1)2}\ot q_{(1)2})\RR(q_{(-1)2}\ot q_{(1)1}\ot q_{(-1)1})\RR(p_{(-1)2}\ot p_{(1)1}\ot q_{(-1)3})q_{(0)}\ot p_{(0)}
        \\&\overset{(\ref{cond: normalizing condition})}{=} \ep(q_{(-1)1})\ep(q_{(1)1})\RR(p_{(-1)1}\ot p_{(1)2}\ot q_{(1)2})\RR(p_{(-1)2}\ot p_{(1)1}\ot q_{(-1)2})q_{(0)}\ot p_{(0)}
        \\&= \RR(p_{(-1)1}\ot p_{(1)2}\ot q_{(1)})\RR(p_{(-1)2}\ot p_{(1)1}\ot q_{(-1)})q_{(0)}\ot p_{(0)}
        \\&\overset{(\ref{cond: comatrix 1 prime})}{=} \RR(p_{(-1)}\ot p_{(1)2}\ot q_{(1)2})\RR(q_{(1)1}\ot p_{(1)1}\ot q_{(-1)})q_{(0)}\ot p_{(0)}
        \\&\overset{(\ref{cond: cyclic conditions})}{=} \RR(p_{(1)2}\ot q_{(1)2}\ot p_{(-1)})\RR(p_{(1)1}\ot q_{(-1)}\ot q_{(1)1})q_{(0)}\ot p_{(0)}
        \\&\overset{(\ref{cond: comatrix 4 prime})}{=} \RR(p_{(1)}\ot q_{(-1)}\ot p_{(-1)})\ep(q_{(1)})q_{(0)}\ot p_{(0)}
        \\&\overset{(\ref{cond: cyclic conditions})}{=} \RR(p_{(-1)}\ot p_{(1)}\ot q_{(-1)})q_{(0)}\ot p_{(0)} = \Upsilon(p\ot q),
    \end{align*}
hence $\Upsilon\circ\Upsilon\circ\Upsilon=\Upsilon$.
\end{proof}
\end{invisible}

\begin{remark}
We observe that, given $p,q\in M$, then
    \begin{align*}
        \Upsilon\Upsilon(p\ot q)
        &= \RR(p_{(-1)1}\ot p_{(1)}\ot q_{(-1)1})\RR(q_{(-1)2}\ot q_{(1)}\ot p_{(-1)2})p_{(0)}\ot q_{(0)}
        \\&\overset{(\ref{cond: cyclic conditions})}{=} \RR(p_{(-1)1}\ot p_{(1)}\ot q_{(-1)1})\RR(p_{(-1)2}\ot q_{(-1)2}\ot q_{(1)})p_{(0)}\ot q_{(0)}
        \\&\overset{(\ref{cond: comatrix 4 prime})}{=} \RR(p_{(-1)}\ot p_{(1)}\ot q_{(1)})\ep(q_{(-1)})p_{(0)}\ot q_{(0)}\\&=\tau_{M,M}(\RR(p_{(-1)}\ot p_{(1)}\ot q_{(1)})q_{(0)}\ot p_{(0)})\\&=\tau_{M,M}\tilde{\sigma}_{M,M}(p\ot q),
        \end{align*}
hence $\Upsilon\Upsilon=\tau_{M,M}\tilde{\sigma}_{M,M}$. We will see that $\tilde{\sigma}_{M,N}$ induces a braiding on the category of $C$-bicomodules.
\end{remark}

The second use of canonical $R$-forms is that they classify the braidings of bicomodules. The following theorem is dual to the classification of braidings on the category of bimodules in \cite[Theorem 3.1]{agore2014braidings}.
\begin{theorem}\label{theorem: braidings on bicomodules}
    Let $C$ be a coalgebra, then there is a bijection between braidings $\sigma$ on the monoidal category $(^C\mm^C, \square^C, C)$ and canonical $R$-forms of $C$.
    Explicitly, for any braiding $\sigma$ on $(^C\mm^C, \square^C, C)$ the corresponding canonical $R$-form $\RR$ is defined by the following composition:
\begin{equation}
\begin{tikzcd}\label{eqn: canonical r comatrix}
	{C\ot C\ot C} &&&& \Bbbk \\
	{C\ot C \square^C C\ot C} && {C\ot C\square^C C\ot C} && {C\ot C\ot C\ot C}
	\arrow["\RR", from=1-1, to=1-5]
	\arrow["{\id\ot \Delta\ot \id}"', from=1-1, to=2-1]
	\arrow["{\sigma_{C\ot C, C\ot C}}"', from=2-1, to=2-3]
	\arrow["{e_{C\ot C,C\ot C}}"', from=2-3, to=2-5]
	\arrow["{\ep^{\ot 4}}"', from=2-5, to=1-5]
\end{tikzcd}
\end{equation}
    Conversely, any canonical $R$-form $\RR$ induces a unique braiding by
    \begin{equation}\label{eqn: definition of braiding}
    \sigma_{V,W}(v\square w) = \RR(v_{(-1)}\ot v_{(1)}\ot w_{(1)})w_{(0)}\square v _{(0)} = \RR(v_{(-1)}\ot w_{(-1)}\ot w_{(1)})w_{(0)}\square v _{(0)}.
    \end{equation}
    Moreover, all braidings on $(^C\mm^C, \square^C, C)$ are symmetries.
\end{theorem}

\begin{invisible}
\begin{proof}
    Let $V$ and $W$ be $C$-bicomodules, and assume the existence of a braiding $\sigma$ on $(^C\mm^C, \square^C, C)$. Define $\RR$ as in (\ref{eqn: canonical r comatrix}). For any $f\in V^*$ and $g\in W^*$ we consider the morphisms $\hat{f}\colon V\to C\ot C$ and $\hat{g}\colon W\to C\ot C$ in $^C\mm^C$ defined as in Lemma \ref{lemma: induced bicomodule morphism}. We observe that
\begin{align*}
        (g\ot f)e_{W,V}\sigma_{V,W}&= \ep^{\ot 4}(\hat{g}\ot \hat{f})e_{W,V}\sigma_{V,W}
        = \ep^{\ot 4}e_{C^{\ot2},C^{\ot 2}}(\hat{g}\square \hat{f})\sigma_{V,W}
        = \ep^{\ot 4}e_{C^{\ot2},C^{\ot 2}}\sigma_{C^{\ot2},C^{\ot 2}}(\hat{f}\square\hat{g})
\end{align*}
For notational ease, we will refer by slight abuse of notation to $\sigma_{C^{\ot2}, C^{\ot2}}$ simply as $\sigma$. We will also avoid writing inclusions $e_{C^{\ot2},C^{\ot2}}$. Using
    the explicit isomorphism $C\cong C\square^C C$ from Remark \ref{lemma: left right unitors for vec}, it follows that
\begin{align*}
        (g\ot f)e_{W,V}\sigma_{V,W}(v\square w)&=\ep^{\ot 4}\sigma(v_{(-1)}\ot v_{(1)}\square w_{(-1)}\ot w_{(1)})f(v_{(0)})g(w_{(0)})
        \\&= \ep^{\ot 4}\sigma(v_{(-1)}\ot v_{(1)1}\square v_{(1)2}\ot w_{(1)})\ep(w_{(-1)})(g\ot f)e_{W,V}(w_{(0)}\square v_{(0)})\\&= \ep^{\ot 4}\sigma(\id\ot\Delta\ot\id)(v_{(-1)}\ot v_{(1)}\ot w_{(1)})(g\ot f)e_{W,V}(w_{(0)}\square v_{(0)})
        \\&= (g\ot f)e_{W,V}(\RR(v_{(-1)}\ot v_{(1)}\ot w_{(1)})w_{(0)}\square v_{(0)}).
    \end{align*}
    By the non-degeneracy of the evaluation $W^*\ot V^*\ot W\ot V\to \Bbbk$, the fact that $f$ and $g$ were chosen arbitrarily implies that $e_{W,V}\sigma_{V,W}(v\square w)=e_{W,V}(\RR(v_{(-1)}\ot v_{(1)}\ot w_{(1)})w_{(0)}\square v_{(0)})$. Since $e_{W,V}$ is a monomorphism, we obtain the first equality of (\ref{eqn: definition of braiding}). The second equality holds likewise, since $v_{(1)}\square w_{(-1)} = \ep(v_{(1)})w_{(-1)1}\square w_{(-1)2}$. 
    \newline
    The fact that $\sigma_{V,W}(V\square W)\subseteq W\square V$ states that for all $v\square w\in V\square^C W$
    \[
    \RR(v_{(-1)}\ot v_{(1)}\ot w_{(1)})w_{(0)(0)}\ot w_{(0)(1)}\ot v_{(0)} = \RR(v_{(-1)}\ot v_{(1)}\ot w_{(1)})w_{(0)}\ot v_{(0)(-1)}\ot v_{(0)(0)}.
    \]
    In particular, for $p,q,r\in C$, it holds that $p\ot q_1\square q_2\ot r\in C^{\ot 2}\square^C C^{\ot 2}$, and this becomes the condition that
\[
\RR(p_1\ot q_2\ot r_3)q_3\ot r_1\ot r_2\ot p_2\ot q_1 = \RR(p_1\ot q_2\ot r_2)q_3\ot r_1\ot p_2\ot p_3\ot q_1.
\]
    Applying $\ep\ot\ep\ot\id\ot \ep\ot\ep$ results in identity (\ref{cond: comatrix 1 prime}). Let $v\in V$, $w\in W$, and $u\in U$, with $V,W,U$ three $C$-bicomodules, and $p,q,r,u\in C$ from now on.
    The fact that $\sigma_{V,W}$ is a morphism in $^C\mm^C$ translates to
    \[
    \sigma_{V,W}(v\square w)_{(-1)}\ot \sigma_{V,W}(v\square w)_{(0)}\ot \sigma_{V,W}(v\square w)_{(1)}\ = v_{(-1)}\ot \sigma_{V,W}(v_{(0)}\square w_{(0)})\ot w_{(1)}.
    \]
    In the case of $V=W=C\ot C$ and $p\ot q_1\square q_2\ot r \in C^{\ot 2}\square^C C^{\ot 2}$ this becomes
    \begin{align*}
        \RR(p_1\ot q_3\ot r_2)q_4\ot q_5\ot r_1\square p_2\ot q_1\ot q_2 = \RR(p_2\ot q_2\ot r_2)p_1\ot q_3\ot r_1\square p_3\ot q_1\ot r_3.
    \end{align*}
    Applying $\id\ot\ep^{\ot 5}$ respectively $\ep^{\ot 5}\ot \id$ results in condition (\ref{cond: comatrix 2 prime}) respectively (\ref{cond: comatrix 3 prime}).
    The braid relation $\sigma_{V, W\square^C U} = (\id\square \sigma_{V,U})(\sigma_{V,W}\square\id)$ translates, considered for $V=W=U=C\ot C$ and $p\ot q_1\square q_2\ot r_1\square r_2\ot u\in C^{\ot 2}\square^C C^{\ot 2}\square^C C^{\ot 2}$, to
    \begin{align*}
        \RR(p_1\ot q_2\ot u_2)&q_3\ot r_1\square r_2\ot u_1\square p_2\ot q_1
        \\&= \RR(p_1\ot q_2\ot r_2)\RR(p_2\ot r_3\ot u_2)q_3\ot r_1\square r_4\ot u_1\square p_3\ot q_1.
    \end{align*}
    Applying $\ep^{\ot 6}$ results in condition (\ref{cond: comatrix 4 prime}).
    Similarly, the braid relation $\sigma_{V\square^C W, U} = (\sigma_{V,U}\square \id)(\id\square \sigma_{W,U})$
    translates to
    \begin{align*}
        \RR(p_1\ot r_2\ot u_2)&r_3\ot u_1\square p_2\ot q_1\square q_2\ot r_1
        \\&= \RR(q_2\ot r_2\ot u_3)\RR(p_1\ot r_3\ot u_2)r_4\ot u_1\square p_2\ot q_1\square q_3\ot r_1.
    \end{align*}
    Applying $\ep^{\ot 6}$ we obtain condition (\ref{cond: comatrix 5 prime}), since
    \[
        \RR(p\ot r\ot u)\ep(q) = \RR(p\ot r_2\ot u_1)\RR(q\ot r_1\ot u_2) \overset{(\ref{cond: comatrix 2 prime})}{=} \RR(p\ot q_1\ot u_1)\RR(q_2\ot r\ot u_2).
    \]
    Notice that $\sigma_{V,W}^{-1}$ is a braiding on $(^C\mm^C, \square^C, C)$ as well and hence by the above reasoning there exists a linear functional $\Tt\colon C\ot C\ot C\to \Bbbk$
    satisfying conditions (\ref{cond: comatrix 1})--(\ref{cond: comatrix 5}). Then
    \begin{align*}
        v\square w = \sigma_{V,W}^{-1}\sigma_{V,W}(v\square w) = \RR(v_{(-1)}\ot v_{(1)}\ot w_{(1)})\Tt(w_{(0)(-1)}\ot v_{(0)(-1)}\ot v_{(0)(1)})v_{(0)(0)}\square w_{(0)(0)}.
    \end{align*}
    For $V = W = C\ot C$ and an element $p\ot q_1\square q_2\ot r \in C^{\ot 2}\square^C C^{\ot 2}$, this becomes
    \[
    p\ot q_1\square q_2\ot r= \RR(p_1\ot q_3\ot r_2)\Tt(q_4\ot p_2\ot q_2)p_3\ot q_1\square q_5\ot r_1.
    \]
    Applying $\ep^{\ot 4}$ we obtain that
    \begin{align*}
    \ep^{\ot 3}(p\ot q\ot r) &= \RR(p_1\ot q_2\ot r)\Tt(q_3\ot p_2\ot q_1) 
    \\&= \RR(p_1\ot q_{12}\ot r)\Tt(q_2\ot p_2\ot q_{11})\overset{(\ref{cond: comatrix 3 prime})}{=} \RR(p_1\ot q_1\ot r_1)\Tt(q_2\ot p_2\ot r_2).
    \end{align*}
    Hence $\Tt(\tau\ot \id)$ is the right convolution inverse to $\RR$. By symmetry, it holds that $\RR(\tau\ot\id)$ is right inverse to $\Tt$, which is equivalent to stating that $\Tt(\tau\ot \id)$ is left inverse to $\RR$, whence $\RR$ is convolution invertible. Moreover, by condition (\ref{cond: symmetry}) of Proposition \ref{prop: equivalent conditions canonical R-comatrix}, $\Tt(\tau\ot\id) = \RR^{-1} = \RR(\tau\ot\id)$, whence $\Tt = \RR$ and $\sigma_{V,W}^{-1} = \sigma_{W,V}$.

Conversely, assume $\RR\colon C\ot C\ot C \to \Bbbk$ to be a canonical $R$-form of $C$. For $C$-bicomodules $V,W$ define the linear map
    \[
    \sigma_{V,W}\colon V\square^C W \to W\ot V \colon v\square w \mapsto \RR(v_{(-1)}\ot v_{(1)}\ot w_{(1)})w_{(0)}\ot v_{(0)},
    \]
    i.e. $\sigma_{V,W}=\tilde{\sigma}_{V,W}|_{V\square^{C}W}$. Since $v\square w \in V\square^C W$, it holds that $v_{(0)}\ot v_{(1)}\ot w = v\ot w_{(-1)}\ot w_{(0)}$, and hence
    \[
    \sigma_{V,W}(v\square w) = \RR(v_{(-1)}\ot v_{(1)}\ot w_{(1)})w_{(0)}\ot v_{(0)} = \RR(v_{(-1)}\ot w_{(-1)}\ot w_{(1)})w_{(0)}\ot v_{(0)}.
    \]
    Utilizing condition \eqref{cond: comatrix 1 prime} one verifies that $\sigma_{V,W}$ has image in $W\square^C V$. Conditions \eqref{cond: comatrix 2 prime} (resp.\ \eqref{cond: comatrix 3 prime}) translate to the left (resp.\ right) $C$-colinearity of $\sigma_{V,W}$.
    For the braid relations, let $U$ be a $C$-bicomodule, and $v\square w\square u\in V\square^C W\square^C U$, then
    \begin{align*}
        (\id\square \sigma_{V,U})&(\sigma_{V,W}\square\id)(v\square w \square u)
        \\= &\RR(v_{(-1)}\ot v_{(1)}\ot w_{(1)})\RR(v_{(0)(-1)}\ot v_{(0)(1)}\ot u_{(1)})w_{(0)}\square u_{(0)}\square v_{(0)(0)}
        \\= &\RR(v_{(-1)1}\ot v_{(1)2}\ot w_{(1)})\RR(v_{(-1)2}\ot v_{(1)1}\ot u_{(1)})w_{(0)}\square u_{(0)}\square v_{(0)}
        \\\stackrel{(\ref{cond: comatrix 3 prime})}{=} &\RR(v_{(-1)1}\ot v_{(1)}\ot w_{(1)1})\RR(v_{(-1)2}\ot w_{(1)2}\ot u_{(1)})w_{(0)}\square u_{(0)}\square v_{(0)}
        \\\stackrel{(\ref{cond: comatrix 4 prime})}{=} &\RR(v_{(-1)}\ot v_{(1)}\ot u_{(1)})\ep(w_{(1)})w_{(0)}\square u_{(0)}\square v_{(0)}
        \\= &\RR(v_{(-1)}\ot v_{(1)}\ot u_{(1)})w\square u_{(0)}\square v_{(0)}
        \\= &\RR(v_{(-1)}\ot v_{(1)}\ot (w\square u)_{(1)}) (w\square u)_{(0)}\square v_{(0)}
        \\= &\sigma_{V, W\square^C U}(v\square w\square u).
    \end{align*}
    Similarly, one verifies that $\sigma_{V\square W, U} = (\sigma_{V,U}\square\id)(\id\square \sigma_{W,U})$.\newline
    Finally, we verify that $\sigma$ is a natural isomorphism. One verifies immediately that $\sigma$ is a natural transformation from $\square^C \to (\square^C)^{\op}$. It suffices to prove that $\sigma_{V,W}$ is a linear isomorphism for all $C$-bicomodules $V$ and $W$. Inspired by the converse case, we expect the inverse to $\sigma_{V,W}$ to be $\sigma_{W,V}$.
    \begin{align*}
        \sigma_{W,V}\sigma_{V,W}(v\square w)&= \RR(v_{(-1)}\ot v_{(1)}\ot w_{(1)})\sigma_{W,V}(w_{(0)}\square v_{(0)})
        \\&= \RR(v_{(-1)}\ot v_{(1)}\ot w_{(1)})\RR(w_{(0)(-1)}\ot w_{(0)(1)}\ot v_{(0)(1)})v_{(0)(0)}\square w_{(0)(0)}
        \\&= \RR(v_{(-1)}\ot v_{(1)2}\ot w_{(1)2})\RR(w_{(-1)}\ot w_{(1)1}\ot v_{(1)1})v_{(0)}\square w_{(0)}
        \\&\stackrel{(\ref{cond: cyclic conditions})}{=} \RR(v_{(1)2}\ot w_{(1)2}\ot v_{(-1)})\RR(v_{(1)1}\ot w_{(-1)}\ot w_{(1)1})v_{(0)}\square w_{(0)}
        \\&\stackrel{(\ref{cond: comatrix 4 prime})}{=} \RR(v_{(1)}\ot w_{(-1)}\ot v_{(-1)})\ep(w_{(1)})v_{(0)}\square w_{(0)}
        \\&=\RR(v_{(1)}\ot w_{(-1)}\ot v_{(0)(-1)})v_{(0)(0)}\square w_{(0)}
        \\&\stackrel{(\star)}{=} \RR(w_{(-1)}\ot w_{(0)(-1)}\ot v_{(-1)})v_{(0)}\square w_{(0)(0)}
        \\&= \RR(w_{(-1)1}\ot w_{(-1)2}\ot v_{(-1)})v_{(0)}\square w_{(0)}
        \\&\stackrel{(\ref{cond: normalizing condition})}{=} \ep(w_{(-1)})\ep(v_{(-1)})v_{(0)}\square w_{(0)}
        \\&= v\square w,
    \end{align*}
    and hence $\sigma_{W,V}$ is a right inverse as well by symmetry. 
\end{proof}
\end{invisible}

We have obtained a characterization of the braidings on $({}^C\mm^C, \square^C, C)$ in terms of canonical $R$-forms on $C$, and have concluded that all the braidings are symmetries. 
We point out that this classification could be generalized to more general categories in which Sweedler type notations for coproducts and coactions can be adopted, and presumably even general braided monoidal categories, by finding relations with respect to the convolution. This goes beyond the scope of our work.

We prove the following result.

\begin{proposition}\label{prop:trivialcanonicalform}
The following statements hold.
\begin{itemize}
    \item[1)] Let $C$ be a coalgebra containing a grouplike element $g$. If $C$ admits a canonical $R$-form, then $C\cong\Bbbk g$.
    \item[2)] Let $A$ be an augmented $\Bbbk$-algebra, i.e.\ $A$ is endowed with an algebra map $\varepsilon:A\to\Bbbk$. If $A$ admits a canonical $R$-matrix, then $A\cong\Bbbk$.
\end{itemize}
Consequently, for a $\Bbbk$-bialgebra $H$, either of the monoidal categories $(_{H}\mm_{H},\ot_{H},H)$ and $(^{H}\mm^{H},\square^{H},H)$ admits a braiding if and only if $H\cong\Bbbk$.
\end{proposition}

\begin{proof}
1). Denote by $\mathfrak{R}$ the canonical $R$-form. The normalization equality \eqref{cond: normalizing condition} with $y=g$ gives $\mathfrak{R}(g,x,g)=\varepsilon(x)$ for every $x\in C$. Then, given $q\in C$, we obtain
\[
q=\varepsilon(q_{2})q_{1}=\mathfrak{R}(g\ot q_{2}\ot g)q_{1}\overset{\eqref{cond: comatrix 3 prime}}{=}\mathfrak{R}(g\ot q\ot g)g=\varepsilon(q)g,
\]
hence $C\cong\Bbbk g$.\\
2). Write $\varepsilon: A\to\Bbbk$ for the augmentation and $R=R^{1}\ot R^{2}\ot R^{3}$ for a canonical $R$-matrix. By 3) of Theorem \ref{thm:braidingsbimodtensorA} and \eqref{cond2braiding}, we know
\[
R^{1}\ot aR^{2}\ot R^{3}=R^{1}\ot R^{2}\ot R^{3}a,\quad \text{for all $a\in A$},\qquad R^{2}\ot R^{3}R^{1}=1\ot 1.
\]
Define $u:=\varepsilon(R^{1})R^{2}\varepsilon(R^{3})$. Applying $\varepsilon\ot\mathrm{id}\ot\varepsilon$ to the first gives $au=\varepsilon(a)u$, for all $a\in A$. Applying $\mathrm{id}\ot\varepsilon$ to the second identity gives $u=1$. Hence $a=\varepsilon(a)1$ for every $a\in A$.
\end{proof}

Examples of braidings on $(^{C}\mm^{C},\square^{C},C)$ may be generated by the following proposition. 

\begin{proposition}\label{cor:bijectionsimmetries}
    If $A$ is a finite-dimensional algebra, then there is a bijection between the symmetries of $(_A\mm_A, \ot_A, A)$ and $(^{A^*}\mm^{A^*}, \square^{A^*}, A^*)$. 
    Likewise, if $C$ is a finite-dimensional coalgebra then there is a bijection between the symmetries of $(^C\mm^C, \square^C, C)$ and $(_{C^*}\mm_{C^*}, \ot_{C^*}, C^*)$.
\end{proposition}
\begin{proof}
    Let $\{e_i\}_{i}$ be a basis of $A$, then there is a linear isomorphism
    \begin{equation}\label{eqn: linear isomorphisms}
    \begin{array}{c c c}
        (A^*\ot A^*\ot A^*)^*&\longrightarrow &A\ot A\ot A
        \\ \RR&\longmapsto&\sum_{i,j,r} \langle \RR \mid e_i^*\ot e_j^*\ot e_r^*\rangle e_i\ot e_j\ot e_r,
        \\(e_i^*\ot e_j^*\ot e_r^*)^*&\longmapsfrom & e_i\ot e_j\ot e_r.
    \end{array}
    \end{equation}
    We claim that for any $\RR\in (A^*\ot A^*\ot A^*)^*$ this is a canonical $R$-form of $A^*$ if and only if the corresponding $R\in A\ot A\ot A$ is a canonical $R$-matrix of $A$. Notice firstly that, denoting
    \[
        e_ie_j = \sum_a \mu_{ij}^a e_a, \hspace{3em} \Delta(e_a^*) = \sum_{i,j} \zeta_{ij}^a e_i^*\ot e_j^*,
    \]
    it follows that
    \[
    \mu_{ij}^a = \langle e_a^*\mid e_ie_j\rangle = \langle e_{a_1}^*\mid e_i\rangle \langle e_{a_2}^*\mid e_j\rangle = \zeta_{ij}^a.
    \]
    Likewise, if we denote $1_A = \sum_i \alpha_i e_i$ then $\ep(e_i^*) = \langle e_i^*\mid 1_A\rangle = \alpha_i$. Denote, by slight abuse of notation as we implicitly use the linear isomorphism $(A^*\ot A^*\ot A^*)^*\cong A^{**}\ot A^{**}\ot A^{**}$,
    \[
    \RR = \sum_{i,j,r}\lambda_{ijr}\ev_{e_i}\ot \ev_{e_j}\ot \ev_{e_r}.
    \]
    Then it follows, utilizing the explicit linear isomorphism (\ref{eqn: linear isomorphisms}), that
    \[
    R = \sum_{i,j,r}\langle \RR\mid e_i^*\ot e_j^*\ot e_r^*\rangle e_i\ot e_j\ot e_r = \sum_{i,j,r}\lambda_{ijr}e_i\ot e_j\ot e_r.
    \]
    Conditions (\ref{cond1braiding}) for $R$ and (\ref{cond: comatrix 3}) for $\RR$ are both equivalent to the identity
    \[
    \sum_r \lambda_{irs}\mu_{ar}^j = \sum_r \lambda_{ijr}\mu_{ra}^s \hspace{1em} \text{for all }i,j,s,a,
    \]
    proving their equivalence. Analogously, one verifies immediately that conditions (\ref{cond2braiding}) for $R$ and (\ref{cond: normalizing condition}) for $\RR$ are both equivalent to the identities
    \[
    \sum_{i,j}\lambda_{ijb} \mu_{ij}^a = \alpha_a\alpha_b = \sum_{i,j}\lambda_{iaj}\mu_{ji}^b \hspace{1em} \text{for all }a,b.\qedhere
    \]
\end{proof}

\begin{remark}
While the finite-dimensionality of the algebra $A$ is needed in order for $A^*$ to be a coalgebra in general, this is not the case for coalgebras. One could be interested in possible generalizations of Proposition \ref{cor:bijectionsimmetries} omitting the finite-dimensionality assumptions. For instance, one may be able to completely drop the finite-dimensionality constraint on the coalgebra $C$, while for an algebra $A$ one could consider the \emph{Sweedler dual} $A^{\circ}$ of the algebra $A$ (see \cite[Chapter VI]{Sweedler}) and consider the relation between the symmetries of $({}_A\mm_A, \ot_A, A)$ and those of $({}^{A^{\circ}}\mm^{A^{\circ}},\square^{A^{\circ}}, A^{\circ})$.
\begin{invisible}
    , if $A^{\circ}$ is dense in $A^*$. The latter means that $f(a) = 0$ for all $f\in A^{\circ}$ if and only if $a=0$. The density assumption is reasonable, since there exist non-zero algebras $A$ such that $A^{\circ} = \{0\}$ \cite[page 114]{Sweedler}, \ls{whose space of bicomodules has no non-trivial braidings [not necessarily a problem? See below]}.

Explicitly, let $E/F$ be an infinite degree field extension, then $E$ is a well-defined algebra over $F$, yet its only ideals are $\{0\}$ and $E$. In particular, its only cofinite ideal is $E$, whence $E^{\circ} = \{0\}$ and thus $({}^{E^{\circ}}\mm^{E^{\circ}}, \square^{E^{\circ}}, E^{\circ})$ admits only the trivial zero braiding, since ${}^{E^{\circ}}\mm^{E^{\circ}}$ is the subcategory containing only the zero space. In particular, let $\Bbbk$ be a field of characteristic $p > 0$, and consider the infinite field extension $E/F = \Bbbk(X_i\mid i\geq 1) / \Bbbk(X_i^p\mid i\geq 1)$. \ls{[This should have the property that the inclusion is an epi of rings (since there is always only a unique extension of a field map $F\to E$ if $F$ has a unique $p$-th root of unity, which should follow from our assumptions by the Frobenius map?), and hence by Agore the category of $E$-bimodules is braided by the flip map. However, even here we find one braiding for each category, hence not a contradiction to there being a possible bijection? If works out, thank Carsten]} \ls{[By Agore (Prop 3.3), the only possible braiding is the flip, which does not work! So maybe dense is not needed.]}
\end{invisible}
\end{remark}
We now use Proposition \ref{cor:bijectionsimmetries} to obtain the following examples:

\begin{example}\label{ex:braidingbicomod}
The following coalgebras $C$ are such that the monoidal category $(^{C}\mm^{C},\square^{C},C)$ has a braiding.
\begin{itemize}
\item[1)] Consider the zero coalgebra $C$, then $({}^C\mm^C, \square^C, C)$ is the monoidal category with unique object $\{0\}$ and a unique morphism. This admits a unique braiding, corresponding to the canonical $R$-form $\ep\ot\ep\ot\ep = 0\colon C\ot C\ot C\to \Bbbk$.
\item[2)] The \textit{comatrix coalgebra} $C=M^{c}_{n}(\Bbbk)$. As a vector space this has a basis $\{f_{ij}\}_{i,j=1,\ldots, n}$, and the coproduct and counit are defined by
\[
\Delta(f_{ij})=\sum_{p=1}^n{f_{ip}\ot f_{pj}},\qquad \varepsilon(f_{ij})=\delta_{ij}.
\]
We have that $M^{c}_{n}(\Bbbk)\cong M_{n}(\Bbbk)^{*}$. In \cite[Example 3.9]{agore2014braidings} an $R$-matrix for the matrix algebra $A=M_{n}(\Bbbk)$ is constructed. This is given by $R=\sum_{i,j,k=1}^{n}{e_{ij}\ot e_{ki}\ot e_{jk}}$, where $e_{ij}$ is the elementary matrix with 1 in the $(i,j)$-position and 0 elsewhere. By Theorem \ref{thm:braidingsbimodtensorA}, $R$ corresponds to a symmetry on $(_{A}\mm_{A},\ot_{A},A)$ and then, by Proposition \ref{cor:bijectionsimmetries}, we have a symmetry on $(^{C}\mm^{C}, \square^{C},C)$. 
\item[3)] Let $\Bbbk$ be a field of characteristic different from 2, and $a,b\in\Bbbk^{\times}$. Consider the vector space $C$ with basis $\{1,i,j,l\}$. Define the coproduct and the counit in the following way
\begin{align*}
&\Delta(1)=1\ot1+ai\ot i+bj\ot j-abl\ot l,\quad \Delta(i)=1\ot i+i\ot1-bj\ot l+bl\ot j\\&
\Delta(j)=1\ot j+j\ot1+ai\ot l-al\ot i,\quad \Delta(l)=1\ot l+l\ot1+i\ot j-j\ot i
\end{align*}
and $\varepsilon(1)=1$, $\varepsilon(i)=\varepsilon(j)=\varepsilon(l)=0$. One has $C\cong A^{*}$, where $A$ is the generalized quaternion algebra, i.e.\ a vector space with basis $\{1,i,j,l\}$ and multiplication defined by $i^{2}=a$, $j^{2}=b$, $ij=-ji=l$. In \cite[Example 3.10]{agore2014braidings} an $R$-matrix for the generalized quaternion algebra $A$ is constructed. By Theorem \ref{thm:braidingsbimodtensorA}, $R$ corresponds to a symmetry on $(_{A}\mm_{A},\ot_{A},A)$ and then, by Proposition \ref{cor:bijectionsimmetries}, we have a symmetry on the category $(^{C}\mm^{C}, \square^{C},C)$.
\end{itemize}
\end{example}

\begin{invisible}
\ls{Notice that we have found very little examples of bicomodule categories which do admit braidings. This can be explained by the following insight that the only finite-dimensional coalgebras whose bicomodule categories admit a (necessarily unique) braiding are the comatrix coalgebras from Example \ref{ex:braidingbicomod}.}

\begin{proposition}[{\cite[Corollary 3.7]{agore2014braidings}}]
    Let $A$ be a finite-dimensional algebra, then $({}_A\mm_A, \ot_A, A)$ admits a braiding if and only if $A$ is a central simple algebra. Moreover, in this case the braiding is unique.
\end{proposition}
\ls{[This is not the complete statement, but it is the part relevant to us. However, by Wedderburn--Artin, a central simple algebra is isomorphic to a matrix algebra \textit{over a division ring}, but is this a matrix algebra over $\Bbbk$?]}
\end{invisible}

Moreover, one can obtain other examples on the tensor product of coalgebras in the following way, dualizing \cite[Example 3.12]{agore2014braidings}. We point out that the finite-dimensionality of $C$ or $D$ is assumed since we need $(C\ot D)^{*}\cong C^{*}\ot D^{*}$.

\begin{lemma}\label{lem:rtensorproduct}
    Let $C$ and $D$ be coalgebras with canonical $R$-forms $\RR_{C}:C\ot C\ot C\to\Bbbk$ and  $\RR_{D}:D\ot D\ot D\to \Bbbk$, respectively. If either $C$ or $D$ is finite-dimensional, then the following is a canonical $R$-form on the tensor product coalgebra $C\ot D$
\[
\RR_{\ot}:(C\ot D)\ot(C\ot D)\ot (C\ot D)\to\Bbbk,\ (c\ot d)\ot(c'\ot d')\ot(c''\ot d'')\mapsto\RR_{C}(c\ot c'\ot c'')\RR_{D}(d\ot d'\ot d'').
\]
\end{lemma}

Dual to \cite[Proposition 3.3]{agore2014braidings} we are able to characterize completely the braidings on the bicomodules over a cocommutative coalgebra. We will see that this setting is more restrictive than that of algebras.

\begin{proposition}\label{prop: unit canonical R-form iff condition}
    Let $C$ be a coalgebra, then the following statements hold:
    \begin{enumerate}
        \item if $C$ admits a decomposable canonical $R$-form $\RR = f\ot g\ot p\in C^*\ot C^*\ot C^*$, then this is $\ep\ot \ep\ot \ep$,
        \item $\ep\ot\ep\ot\ep$ is a canonical $R$-form of $C$ if and only if $\ep\ot f = f\ot \ep$ for all $f\in C^*$ if and only if $\ep^*\colon \Bbbk\to C^*$ is an epimorphism of rings.
    \end{enumerate}
\end{proposition}

Let us study when a non-zero coalgebra $C$ admits a decomposable canonical $R$-form in $C^*\ot C^*\ot C^*$.

\begin{corollary}\label{cor: non-zero coalgebra unit r-form iff field}
    If $C$ is a non-zero coalgebra, then $C$ admits a decomposable canonical $R$-form $\RR = f\ot g\ot p\in C^*\ot C^*\ot C^*$ if and only if $C \cong \Bbbk$, in which case the induced braiding is the flip map. 
\end{corollary}
\begin{proof}
     By Proposition \ref{prop: unit canonical R-form iff condition}, $C$ admits a decomposable canonical $R$-form in $C^*\ot C^*\ot C^*$ if and only if this is $\ep\ot\ep\ot\ep$ if and only if $\ep\ot f= f\ot \ep$ for all $f\in C^*$. The latter is equivalent to $x\ep(y)=\ep(x)y$ for all $x,y\in C$. As $C\neq \{0\}$ there exists an element $z\in C$ such that $\ep(z) = 1$. Hence for all $y\in C$ it holds that $y = z\ep(y)$, which defines a coalgebra isomorphism $C = \Bbbk z\cong \Bbbk$.
 \end{proof}
By Corollary \ref{cor: non-zero coalgebra unit r-form iff field}, canonical $R$-forms of non-trivial coalgebras are rather involved.

\begin{corollary}\label{cor: cocommutative braided}
    If $C$ is a non-zero cocommutative coalgebra, then $({}^{C}{\mm}^C, \square^C, C)$ admits a braiding if and only if $C\cong \Bbbk$, in which case the braiding is the flip map.
\end{corollary}
\begin{proof}
     Let $C$ be a cocommutative coalgebra and $\RR$ a canonical $R$-form of $C$. Then for any $p\ot q \ot r \in C\ot C\ot C$ it holds that
        \begin{align*}
             \RR(p\ot q\ot r) &= \RR(p\ot q_1\ot r)\ep(q_2)
             \\&\overset{(\ref{cond: comatrix 4 prime})}{=} \RR(p_1\ot q_1\ot q_{21})\RR(p_2\ot q_{22}\ot r)
             \\&\overset{(\ref{cond: comatrix 3 prime})}{=} \RR(p_1\ot q_1\ot r_2)\RR(p_2\ot q_2\ot r_1)
             \\&= \RR(p_1\ot q_1\ot r_1)\RR(p_2\ot q_2\ot r_2)
             = (\RR\ast \RR)(p\ot q \ot r).
         \end{align*}
         As $\RR$ is convolution invertible, this implies that $\RR= \ep\ot \ep\ot \ep$, and its induced braiding is the flip map. By Corollary \ref{cor: non-zero coalgebra unit r-form iff field} this implies that $C\cong \Bbbk$.
\end{proof}

Finally, dual to \cite[Theorem 3.2.11]{ASthesis}, we now study the infinitesimal braidings on the category of $C$-bicomodules, for an arbitrary coalgebra $C$. 

\begin{proposition}\label{prop: infinitesimal braidings characterization}
    Let $C$ be a coalgebra and $\sigma$ a braiding on the category $(^C\mm^C, \square^C, C)$ with corresponding canonical $R$-form $\RR$. Then the infinitesimal braidings on $(^C\mm^C, \square^C, C, \sigma)$ are in bijection to linear functionals $\chi\colon C\ot C\ot C\to \Bbbk$ satisfying, for all $f\in C^*$, the following identities:
    \begin{align}
        &\chi\ast(\ep\ot f\ot \ep) = (\ep\ot f\ot \ep)\ast \chi,\label{cond: inf comatrix 1}
        \\&\chi\ast (f\ot \ep\ot \ep) = (f\ot \ep\ot \ep)\ast \chi,\label{cond: inf comatrix 2}
        \\&\chi\ast (\ep\ot\ep\ot f) = (\ep\ot\ep\ot f)\ast \chi,\label{cond: inf comatrix 3}
        \\ &(\chi\ot \ep)(\id\ot\id\ot\tau) = \chi\ot\ep + (\RR\ot\ep)(\id\ot\tau\ot\id)\ast (\RR\ot\ep)\ast (\ep\ot\chi)(\tau\ot\id\ot\id),\label{cond: inf comatrix 4}
        \\ &(\ep\ot\chi)(\tau\ot\id\ot\id) = \ep\ot\chi + (\chi\ot\ep)(\id\ot\id\ot\tau)\ast (\ep\ot\RR)\ast (\ep\ot\RR)(\id\ot\tau\ot\id).\label{cond: inf comatrix 5}
    \end{align}
\end{proposition}
\begin{proof}
     Notice that, by utilizing conditions (\ref{cond: comatrix 2 prime}) and (\ref{cond: comatrix 3 prime}) for (\ref{cond: inf comatrix 4}) respectively (\ref{cond: inf comatrix 5}), conditions (\ref{cond: inf comatrix 1})--(\ref{cond: inf comatrix 5}) are one-by-one equivalent to stating that for all $p,q,r,t\in C$ 
   \begin{align}
         &\chi(p\ot q_1\ot r)q_2 = \chi(p\ot q_2\ot r)q_1,\tag{33'}\label{cond: inf comatrix 1 prime}
         \\&\chi(p_1\ot q\ot r)p_2 = \chi(p_2\ot q\ot r)p_1,\tag{34'}\label{cond: inf comatrix 2 prime}
         \\&\chi(p\ot q\ot r_1)r_2 = \chi(p\ot q\ot r_2)r_1,\tag{35'}\label{cond: inf comatrix 3 prime}
         \\&\chi(p\ot q \ot u)\ep(r) = \chi(p\ot q\ot r)\ep(u) + \RR(p_1\ot q_2\ot r_2)\chi(p_2\ot r_3\ot u)\RR(q_3\ot r_1\ot q_1),\tag{36'}\label{cond: inf comatrix 4 prime}
         \\&\chi(p\ot r\ot u)\ep(q) = \chi(q\ot r\ot u)\ep(p)+ \RR(q_2\ot r_2\ot u_2)\chi(p\ot q_1\ot u_1)\RR(r_3\ot q_3\ot r_1).\tag{37'}\label{cond: inf comatrix 5 prime}
     \end{align}
 Completely analogously to the proof of Theorem \ref{theorem: braidings on bicomodules} one finds that an infinitesimal braiding $t$ is of the form
 \begin{equation}\label{eqn: definition of inf braiding}
 t_{V,W}(v\square w) = \chi(v_{(-1)}\ot v_{(1)}\ot w_{(1)})v_{(0)}\square w_{(0)} = \chi(v_{(-1)}\ot w_{(-1)}\ot w_{(1)})v_{(0)}\square w_{(0)},
 \end{equation}
 and conditions (\ref{cond: inf comatrix 1 prime})--(\ref{cond: inf comatrix 5 prime}) can be shown to be equivalent to the respective properties that $t_{V,W}$ has image in $V\square^C W$, $t_{V,W}$ is left respectively right $C$-colinear, and the infinitesimal braid equations (\ref{eqn: inf braid 1}) and (\ref{eqn: inf braid 2}).
\end{proof}


As a consequence we can deduce that $(^C\mm^C, \square^C, C)$ admits no non-zero infinitesimal braidings, for any braiding $\sigma$ thereon.

\begin{theorem}\label{thm:infbraidbicomod}
    Let $C$ be a coalgebra. For any braiding $\sigma$ on $({}^C\mm^C, \square^C, C)$, there is no non-zero infinitesimal braiding on $({}^C\mm^C, \square^C, C, \sigma)$.
\end{theorem}
\begin{proof}
    Let $\RR$ be the canonical $R$-form corresponding to the braiding $\sigma$, and $t$ an infinitesimal braiding of $(^C\mm^C, \square^C, C,\sigma)$ corresponding to the linear functional $\chi$. For any $p,q,r\in C$ it holds that
    \begin{align*}
        \chi(p\ot q \ot r) &= \chi(p\ot q \ot r_2)\ep(r_1)
        \\&\stackrel{(\ref{cond: inf comatrix 4 prime})}{=} \chi(p\ot q\ot r)+ \RR(p_1\ot q_2\ot r_2)\chi(p_2\ot r_3\ot r_4)\RR(q_3\ot r_1\ot q_1).
    \end{align*}
    Hence we find that
    \begin{align*}
    0 &= \RR(p_1\ot q_2\ot r_2)\chi(p_2\ot r_3\ot r_4)\RR(q_3\ot r_1\ot q_1)
    \\&= \RR(p_1\ot q_2\ot r_{21})\chi(p_2\ot r_{22}\ot r_3)\RR(q_3\ot r_1\ot q_1)
    \\&\stackrel{(\ref{cond: comatrix 3 prime})}{=} \RR(p_1\ot q_{22}\ot r_2)\chi(p_2\ot q_{21}\ot r_3)\RR(q_3\ot r_1\ot q_1)
    \\&= \RR(p_1\ot q_{3}\ot r_2)\chi(p_2\ot q_{2}\ot r_3)\RR(q_4\ot r_1\ot q_1)
    \\&\stackrel{(\ref{cond: inf comatrix 1 prime})}{=} \RR(p_1\ot q_2\ot r_2)\chi(p_2\ot q_3\ot r_3)\RR(q_4\ot r_1\ot q_1)
    \\&\stackrel{(\ref{cond: cyclic conditions})}{=} \RR(q_4\ot r_1\ot q_1)\RR(r_2\ot p_1\ot q_2)\chi(p_2\ot q_3\ot r_3)
    \\&\stackrel{(\ref{cond: comatrix 5 prime})}{=} \RR(q_3\ot p_1\ot q_1)\ep(r_1)\chi(p_2\ot q_2\ot r_2)
    \\&= \RR(q_3\ot p_1\ot q_1)\chi(p_2\ot q_2\ot r)
    \\&\stackrel{(\ref{cond: inf comatrix 1 prime})}{=} \RR(q_2\ot p_1\ot q_1)\chi(p_2\ot q_3\ot r)
    \\&\stackrel{(\ref{cond: normalizing condition})}{=} \ep(q_1)\ep(p_1)\chi(p_2\ot q_2\ot r)
    \\&= \chi(p\ot q\ot r).
    \end{align*}
    Thus $t$, as this is completely determined by $\chi$ as in \eqref{eqn: definition of inf braiding}, is the zero infinitesimal braiding.
\end{proof}

In the next section, we investigate duoidal structures associated to the categories of bimodules and bicomodules studied before. This will be done more generally for the categories of bi(co)modules in arbitrary braided monoidal categories.

\section{Duoidal structures for bi(co)modules in general monoidal categories}\label{sec:duoidal} 

First, we recall the definition of a duoidal category.

\begin{definition}[{cf.\ \cite[Definition 6.1]{Aguiar}}]\label{defn: duoidal}
    A \emph{duoidal category} is a category $\Cc$ equipped with two monoidal structures $\Cc^\circ= (\Cc, \circ, I, a^\circ, l^\circ, r^\circ )$ and $\Cc^\bullet = (\Cc, \bullet, J, a^\bullet, l^\bullet, r^\bullet )$,  related through a natural transformation
\begin{equation}\label{def:zetaduoidal}
\zeta=(\zeta_{A,B,C,D}: (A\bullet B)\circ (C\bullet D)\to (A\circ C)\bullet (B\circ D))_{A,B,C,D\in\Cc},
\end{equation}
called the \emph{interchange law}, and morphisms
\begin{equation}\label{def:deltamepsilonduoidal}
\Delta^{\bullet}_I:I\to I\bullet I, \quad m^{\circ}_J: J\circ J\to J, \quad \varepsilon^{\bullet}_I=u^{\circ}_{J}:I\to J
\end{equation}
such that
\begin{itemize}
    \item[$i)$] $(J, m^\circ_J, u^\circ_J)$ is a monoid in $\Cc^\circ$ and $(I, \Delta^\bullet_I, \varepsilon^\bullet_I)$ is a comonoid in $\Cc^\bullet$;
    \item[$ii)$] the following associativity conditions hold, for all $A,B,C,D,E,F\in\Cc$,
    \begin{equation}\label{eq:assoc1}
    \begin{split}
\zeta_{A,B, C\circ E, D\circ F}(\id_{A\bullet B}&\circ\zeta_{C,D,E,F})a^\circ_{A\bullet B, C\bullet D, E\bullet F}\\&=(a^\circ_{A,C,E}\bullet a^\circ_{B,D,F})\zeta_{A\circ C, B\circ D,E,F}(\zeta_{A,B,C,D}\circ\id_{E\bullet F})
 \end{split}   \end{equation}
 \begin{equation}\label{eq:assoc2}
    \begin{split}
(\id_{A\circ D}\bullet \zeta_{B,C,E,F})&\zeta_{A,B\bullet C, D, E\bullet F}(a^\bullet_{A,B,C}\circ a^\bullet_{D,E,F})\\&=a^\bullet_{A\circ D, B\circ E, C\circ F}(\zeta_{A,B,D, E}\bullet \id_{C\circ F})\zeta_{A\bullet B, C, D\bullet E, F};
    \end{split}
 \end{equation}
 \item[$iii)$] the following unitality conditions hold, for all $A,B\in\Cc$
 \begin{equation}\label{eq:unit1}
         (l^\circ_{A}\bullet l^\circ_B)\zeta_{I,I, A,B}(\Delta^{\bullet}_I\circ\id_{A\bullet B})=l^\circ_{A\bullet B};\quad (r^\circ_A\bullet r^\circ_B)\zeta_{A,B,I,I}(\id_{A\bullet B}\circ\Delta^{\bullet}_I)=r^\circ_{A\bullet B};
 \end{equation}
 \begin{equation}\label{eq:unit2}
    l^\bullet_{A\circ B}(m^{\circ}_J\bullet \id_{A\circ B})\zeta_{J,A,J,B}=l^\bullet_A\circ l^\bullet_B;\quad r^\bullet_{A\circ B}(\id_{A\circ B}\bullet m^{\circ}_J)\zeta_{A,J,B,J}=r^\bullet_A\circ r^\bullet_B.
 \end{equation}
\end{itemize}
A duoidal category will be denoted by $(\Cc,\circ, I,\bullet, J)$ or $(\Cc,\circ,\bullet)$. In case the structure morphisms \eqref{def:zetaduoidal} and \eqref{def:deltamepsilonduoidal} are isomorphisms in $\Cc$, the duoidal category is called \textit{strong}, see \cite[Definition 6.3]{Aguiar}.
\end{definition}

It has been proven in \cite{Batanin-Markl} that every duoidal category is \emph{duoidally equivalent} to a strict duoidal category, i.e.\ one where both monoidal categories are strict. For this reason, we will omit the associativity constraints in computations involving general duoidal categories. 

\begin{remark}
    As observed in \cite[Remark 6.2]{Aguiar}, the notion of a strict duoidal category $(\Cc,\circ,I,\bullet,J)$ where $J=I$ (and $\Delta^{\bullet}_{I}=m^{\circ}_{J}=\varepsilon^{\bullet}_{I}=\id_{I}$) coincides with the notion of \textit{2-fold monoidal category} which has appeared in \cite{BFSV}. The strictness assumption was then removed in \cite{FSS} and the unit objects were allowed to be distinct, but the structure maps $\Delta^{\bullet}_{I}$ and $m_{J}^{\circ}$ were assumed to be isomorphisms.
\end{remark}

\begin{remark}\label{remark: strong duoidal is braided}
    Generally, the interchange law and morphisms (\ref{def:deltamepsilonduoidal}) are not isomorphisms, this is, the duoidal category is not strong. Aguiar and Mahajan proved that strong duoidal categories correspond to braided monoidal categories \cite[Propositions 6.10, 6.11]{Aguiar}. 
    For a strong duoidal category, both monoidal structures are braided and braided monoidally isomorphic.
    Moreover, the interchange law is then of the form
    \[
    \zeta_{A,B,C,D}\colon (A\bullet B)\bullet (C\bullet D)\xrightarrow{\cong} A \bullet (B\bullet C)\bullet D\xrightarrow{\id\bullet\sigma_{B,C}\bullet\id} A\bullet (C\bullet B)\bullet D\xrightarrow{\cong} (A\bullet C)\bullet(B\bullet D).
    \]
\end{remark}

\begin{invisible}
We also recall that, given a duoidal category $(\mathcal{C}, \circ, I, \bullet, J)$, there are canonical equivalences of categories
\begin{equation}\label{equivalnecesBimonC}
\Bimon (\mathcal{C},\circ,\bullet)\cong\Comon (\mathsf{Mon}(\mathcal{C}^\circ), \bullet)\cong \mathsf{Mon}(\Comon (\mathcal{C}^\bullet), \circ),
\end{equation}
see \cite[Proposition 6.36]{Aguiar}.
\end{invisible}

\begin{example}\label{ex:duoidalbimodules}
\label{es:bimod} As shown in \cite[Proposition 2.7]{BT25}, see also \cite[Remark 1.6]{Sar21}, the category ${}_H\mm_H$, where $H$ is a $\Bbbk$-bialgebra, is duoidal. In fact,  the monoidal structures are given by $({}_H\mm_H,  \circ =\otimes_H, I=H )$  and $ ({}_H\mm_H, \bullet=\otimes_\Bbbk, J=\Bbbk )$, where the base field $\Bbbk$ is regarded as an $H$-bimodule via the counit $\varepsilon_{H}$ of $H$. The interchange law $\zeta$ is the natural transformation whose component $\zeta_{A,B,C,D}:(A\bullet B)\circ (C\bullet D)\to (A\circ C)\bullet (B\circ D),$ for all $A,B,C,D\in{}_H\mm_H$, is defined by
\[
\zeta_{A,B,C,D}\big((a\otimes b)\otimes_H (c\otimes d)\big)=(a\otimes_H c)\otimes (b\otimes_H d).
\]
Moreover, the morphisms $\Delta^{\bullet}_{I}:I\to I\bullet I$, $m^{\circ}_{J}:J\circ J\to J$ and $\varepsilon^{\bullet}_{I}:I\to J$ are, respectively, given by $\Delta_{H}$, $m^{\circ}_{\Bbbk}(k\ot_{H}k')=kk'$ and $\varepsilon_{H}$. \begin{invisible}
We observe that $\Delta_{H}$ and $\varepsilon_{H}$ are $H$-bilinear since $H$ is a bialgebra. In fact, we have $\Delta_{H}(h\cdot a\cdot h')=\Delta_{H}(hah')=h_{1}a_{1}h'_{1}\ot h_{2}a_{2}h'_{2}=h\cdot\Delta(a)\cdot h'$.
\end{invisible}
This example will be generalized in Theorem \ref{theorem: bi(co)modules duoidal}.
\end{example}

\begin{remark}\label{rmk:obstructionliftingduoidal}
We observe that $\Delta_{H}$ is not a morphism of bicomodules in general, considering $H$ as a bicomodule with its coproduct. 
\begin{invisible}
Given $h\in H$, we have that
\begin{align*}
&\rho^L_{H\ot H}\Delta_{H}(h)=\rho^L_{H\ot H}(h_{1}\ot h_{2})=h_{1_{(-1)}}h_{2_{(-1)}}\ot h_{1_{(0)}}\ot h_{2_{(0)}}=h_{1_{1}}h_{2_{1}}\ot h_{1_{2}}\ot h_{2_{2}}=h_{1}h_{3}\ot h_{2}\ot h_{4}\\
&\rho^{R}_{H\ot H}\Delta_{H}(h)=\rho^{R}_{H\ot H}(h_{1}\ot h_{2})=h_{1_{(0)}}\ot h_{2_{(0)}}\ot h_{1_{(1)}}h_{2_{(1)}}=h_{1_{1}}\ot h_{2_{1}}\ot h_{1_{2}}h_{2_{2}}=h_{1}\ot h_{3}\ot h_{2}h_{4}
\end{align*}
while $(\id\ot\Delta_{H})\Delta_{H}(h)=h_{1}\ot h_{2}\ot h_{3}=(\Delta_{H}\ot\id)\Delta_{H}(h)$. 
\end{invisible}
More precisely, we have that $\Delta_{H}$ is $H$-bicolinear if and only if the following equalities hold, for all $h\in H$:
\begin{equation}\label{eq:Deltabicolinear}
    h_{1}h_{3}\ot h_{2}\ot h_{4}=h_{1}\ot h_{2}\ot h_{3}=h_{1}\ot h_{3}\ot h_{2}h_{4}
\end{equation}
By applying $\id\ot\varepsilon\ot\id$ and $\id\ot\id\ot\varepsilon$ to the first equality we get $h_{1}h_{2}\ot h_{3}=\Delta(h)=h_{1}h_{3}\ot h_{2}$, while applying $\id\ot\varepsilon\ot\id$ and $\varepsilon\ot\id\ot\id$ to the second one we get $h_{1}\ot h_{2}h_{3}=\Delta(h)=h_{2}\ot h_{1}h_{3}$. These imply that $H$ is cocommutative, so that \eqref{eq:Deltabicolinear} can be rewritten as 
\begin{equation}\label{eq:Deltabicolinear2}
h_{1}h_{2}\ot h_{3}\ot h_{4}=h_{1}\ot h_{2}\ot h_{3}=h_{1}\ot h_{2}\ot h_{3}h_{4}, 
\end{equation}
which is clearly equivalent to $h_{1}h_{2}\ot h_{3}=\Delta(h)=h_{1}\ot h_{2}h_{3}$. Moreover, the latter is equivalent to $h=h_{1}h_{2}$, for all $h\in H$. This is a quite strong condition which does not hold in general. In fact, one can observe that $P(H)=\{0\}$.
\begin{invisible}
Suppose $x\in H$ such that $\Delta(x)=x\ot 1+1\ot x$, then $x=2x$, hence $x=0$. 
\end{invisible}
Moreover, the monoid $G(H)$ has only idempotent elements. We point out that, in case $H$ is a Hopf algebra, the condition $h=h_{1}h_{2}$ for all $h\in H$ is equivalent to $h=\varepsilon(h)1$ for all $h\in H$, i.e. $H=\Bbbk1_{H}$. 
\begin{invisible}
$h=\varepsilon(h_{1})h_{2}=S(h_{1_{1}})h_{1_{2}}h_{2}=S(h_{1})h_{2_{1}}h_{2_{2}}=S(h_{1})h_{2}=\varepsilon(h)1$
\end{invisible}
However, the morphism $\varepsilon_{H}:H\to\Bbbk$ is $H$-bicolinear if and only if $h=\varepsilon(h)1_{H}$, for all $h\in H$, i.e.\ $H=\Bbbk1_{H}$.
As a consequence, although ${}^H({}_H\mm_H)^H \cong {}^H_H\mm^H_H$ are equivalent categories this duoidal structure cannot be lifted to the category $^{H}_{H}\mm^{H}_{H}$.
\end{remark}

\begin{invisible}
\begin{remark}
\label{rmk:bimonbimod}
In view of by \cite[Corollary 2.10]{BT25}, in the setting of Example \ref{es:bimod},   giving a bimonoid $(H,m_H^\circ,u_H^\circ,\Delta_H^\bullet,\varepsilon_H^\bullet)$ in $\Cc$
is equivalent to giving a bialgebra $H$ together with a bialgebra map $u_{H}^{\circ}:B\rightarrow H$ which defines the $B$-bimodule
structure of $H$ so that $u_{H}^{\circ}\left(
b\right) =b\cdot 1_{H}.$
Through this identification, a morphism of bimonoids turns out to be a bialgebra map $f:H\to H'$ such that $fu_H^\circ=u_{H'}^\circ$.
In other words, the category  $\Bimon(\Cc,\circ,\bullet)$ of bimonoids is isomorphic to the coslice category of bialgebras under $B$, i.e.\ $B/\Bialg$.
\end{remark}
\end{invisible}

We can generalize the duoidal structure of Example \ref{ex:duoidalbimodules} to an arbitrary braided monoidal category and a bimonoid in it. Dually, we prove the same result for the category of bicomodules.

\begin{theorem}\label{theorem: bi(co)modules duoidal}
    Let $(\Mm, \ot, I, \sigma)$ be a braided monoidal category and $H$ be an object in $\mathsf{Bimon}(\Mm)$.
    \begin{enumerate}[label = \alph*)]
        \item If $\Mm$ is coregular, then $(_H\Mm_H, \ot_H, H, \ot, I)$ is a duoidal category.
        \item If $\Mm$ is regular, then $(^H\Mm^H, \ot, I, \square^H, H)$ is a duoidal category.
    \end{enumerate}
\end{theorem}

\begin{proof}
    Without loss of generality, we can assume $\Mm$ to be strict monoidal.

Since $(_H\Mm_H, \ot_{H},H)$ is a monoidal category by 1) of Theorem \ref{theorem: bi(co)modules monoidal}, and $(_{H}\Mm_{H},\ot,I)$ is a monoidal category by Proposition \ref{prop:bi(co)modmonoidalot}, to have that $(_H\Mm_H, \ot_H, H, \ot, I)$ is a duoidal category we need:
\begin{itemize}
    \item a natural transformation $\zeta_{M,N,P,Q}\colon (M\ot N)\ot_H (P\ot Q)\to (M\ot_H P)\ot (N\ot_H Q)$ of $H$-bimodules,
    \item a monoid structure $(I, m_I, u_I)$ in $(_H\Mm_H, \ot_H, H)$,
    \item a comonoid structure $(H, d_H, e_H=u_I)$ in $(_H\Mm_H, \ot, I)$,
\end{itemize}
satisfying the three compatibility conditions outlined in Definition \ref{defn: duoidal}. Given $M,N,P,Q$ in $_{H}\Mm_{H}$, consider the morphism in $\Mm$ defined as follows:
\[
\Xi_{M,N,P,Q} \colon M\ot N \ot P \ot Q \xrightarrow{M\ot \sigma_{N,P}\ot Q} M\ot P\ot N \ot Q \xrightarrow{q_{M,P}\ot q_{N,Q}} (M\ot_H P)\ot (N\ot_H Q).
\]
It is straightforward to prove that $\Xi_{M,N,P,Q}$ coequalizes the diagram
\[\begin{tikzcd}
	{M\ot N\ot H \ot P\ot Q} && {M\ot N \ot P\ot Q.}
	\arrow["{\alpha_{M\ot N}^R\ot P\ot Q}", shift left=1.5, from=1-1, to=1-3]
	\arrow["{M\ot N \ot \alpha^L_{P\ot Q}}"', shift right=1, from=1-1, to=1-3]
\end{tikzcd}\]
Hence there is a unique $\Mm$-morphism $\zeta_{M,N,P,Q}\colon (M\ot N)\ot_H (P\ot Q)\to (M\ot_H P)\ot (N\ot_H Q)$ satisfying that
\begin{equation}\label{eqn: definition of zeta}
    \zeta_{M,N,P,Q}q_{M\ot N, P\ot Q} = \Xi_{M,N,P,Q}.
\end{equation}
We claim that the induced $\zeta$ is an interchanging law for the monoidal structures $\ot_H$ and $\ot$ on $_H\Mm_H$. Explicitly, we need verify that $\zeta$ is a natural transformation of $H$-bimodules satisfying identities (\ref{eq:assoc1}) and (\ref{eq:assoc2}). Firstly, notice that\begin{invisible}, by the canonical morphisms of coequalizers being epimorphisms,\end{invisible} $\zeta$ is a morphism in $_{H}\Mm_{H}$ if and only if
\begin{align*}
    &\begin{split}\zeta_{M,N,P,Q}\alpha^L_{(M\ot N)\ot_H (P\ot Q)}&(H\ot q_{M\ot N, P\ot Q}) \\&= \alpha^L_{(M\ot_H P)\ot (N\ot_H Q)}(H\ot \zeta_{M,N,P,Q})(H\ot q_{M\ot N, P\ot Q}),
    \end{split}
    \\&\begin{split}\zeta_{M,N,P,Q}\alpha^R_{(M\ot N)\ot_H (P\ot Q)}&(q_{M\ot N, P\ot Q}\ot H) 
    \\&= \alpha^R_{(M\ot_H P)\ot (N\ot_H Q)}(\zeta_{M,N,P,Q}\ot H)(q_{M\ot N, P\ot Q}\ot H).\end{split}
\end{align*}
Both identities are easily verified to hold by construction of the $H$-bimodule structure on the $H$-linear tensor product of $H$-bimodules, rendering $\zeta_{M,N,P,Q}$ a morphism of $H$-bimodules. Notice that the naturality in every component is equivalent to stating that for all $H$-bimodule morphisms $f\colon M\to X, g\colon N\to Y, h\colon P\to Z,$ and $t\colon Q\to U$ the following diagram commutes;
\[\begin{tikzcd}
	{(M\ot N)\ot_H(P\ot Q)} && {(M\ot_H P)\ot (N\ot_H Q)} \\
	\\
	{(X\ot Y)\ot_H(Z\ot U)} && {(X\ot_H Z)\ot (Y\ot_H U).}
	\arrow["{\zeta_{M,N,P,Q}}", from=1-1, to=1-3]
	\arrow["{(f\ot g)\ot_H (h\ot t)}"', from=1-1, to=3-1]
	\arrow["{(f\ot_H h)\ot (g\ot_H t)}", from=1-3, to=3-3]
	\arrow["{\zeta_{X,Y,Z,U}}"', from=3-1, to=3-3]
\end{tikzcd}\]
Since $q_{M\ot N, P\ot Q}$ is an epimorphism, this holds immediately as
\begin{align*}
    \zeta_{X,Y,Z,U}((f\ot g)\ot_H (h\ot t))q_{M\ot N, P\ot Q} &= \zeta_{X,Y,Z,U}q_{X\ot Y, Z\ot U}(f\ot g\ot h\ot t)
    \\&= \Xi_{X,Y,Z,U}(f\ot g\ot h\ot t)
    \\&= (q_{X,Z}\ot q_{Y,U})(X\ot \sigma_{Y,Z}\ot U)(f\ot g\ot h\ot t)
    \\&= (q_{X,Z}\ot q_{Y,U})(f\ot h\ot g\ot t)(M\ot \sigma_{N,P}\ot Q)
    \\&= ((f\ot_H h)\ot (g\ot_H t))(q_{M,P}\ot q_{N,Q})(M\ot \sigma_{N,P}\ot Q)
    \\&= ((f\ot_H h)\ot (g\ot_H t))\Xi_{M,N,P,Q}
    \\&= ((f\ot_H h)\ot (g\ot_H t))\zeta_{M,N,P,Q}q_{M\ot N, P\ot Q}.
\end{align*}
Because the monoidal categories are strict, the compatibility conditions (\ref{eq:assoc1}) and (\ref{eq:assoc2}) translate to the commutativity of the below two diagrams for all $H$-bimodules $M,N,P,Q,R,$ and $T$
\[\begin{tikzcd}
	{(M\ot N)\ot_H (P\ot Q)\ot_H (R\ot T)} &&& {(M\ot N)\ot_H (P\ot_H R)\ot (Q\ot_H T)} \\
	\\
	{(M\ot_H P)\ot (N\ot_H Q)\ot_H (R\ot T)} &&& {(M\ot_H P\ot_H R) \ot (N\ot_H Q\ot_H T),} \\
	{(M\ot N \ot P) \ot_H (Q\ot R \ot T)} &&& {(M\ot_H Q)\ot (N\ot P)\ot_H (R\ot T)} \\
	\\
	{(M\ot N)\ot_H (Q\ot R)\ot (P\ot_H T)} &&& {(M\ot_H Q)\ot (N\ot_H R) \ot (P\ot_H T).}
	\arrow["{M\ot N \ot_H \zeta_{P, Q, R, T}}", from=1-1, to=1-4]
	\arrow["{\zeta_{M, N, P, Q}\ot_H R\ot T}"', from=1-1, to=3-1]
	\arrow["{\zeta_{M, N, P\ot_H R, Q\ot_H T}}", from=1-4, to=3-4]
	\arrow["{\zeta_{M\ot_H P, N\ot_H Q, R, T}}"', from=3-1, to=3-4]
	\arrow["{\zeta_{M, N\ot P, Q, R\ot T}}", from=4-1, to=4-4]
	\arrow["{\zeta_{M\ot N, P, Q\ot R, T}}"', from=4-1, to=6-1]
	\arrow["{M\ot_H Q \ot \zeta_{N,P,R,T}}", from=4-4, to=6-4]
	\arrow["{\zeta_{M, N, Q, R}\ot P\ot_H T}"', from=6-1, to=6-4]
\end{tikzcd}\]
Precomposing with the appropriate epimorphisms $q$ proves this immediately.\newline
Let us endow $I$ with a monoid structure in $(_H\Mm_H, \ot_H, H)$. Notice that $I$ is an $H$-bimodule for the counit of $H$. By definition of this $H$-bimodule structure on $I$, the unitors $l_I=r_I\colon I\ot I \to I$ coequalize $\alpha^R_I\ot I$ and $I\ot\alpha^L_I$. Hence there is a unique morphism $m_I\colon I\ot_H I \to I$ such that $m_Iq_{I,I} = l_I$. An elementary verification confirms that $(I, m_I, \ep_{H})$ is a monoid in $(_H\Mm_H, \ot_H, H)$. Moreover, we know that $(H,\Delta, \ep)$ is a comonoid in $(\Mm, \ot, I, \sigma)$ by assumption, whose counit coincides with the unit of $(I,m_I, \ep)$. 
Its morphisms are moreover $H$-bilinear, since
\begin{alignat*}{2}
    &\Delta\alpha^L_H &&= \Delta\mu
    \\& &&=(\mu\ot\mu)(H\ot\sigma_{H,H}\ot H)(\Delta\ot\Delta)
    \\& &&= (\alpha^L_H\ot \alpha^L_H)(H\ot\sigma_{H,H}\ot H)(\Delta\ot H\ot H)(H\ot \Delta)
    =\alpha^L_{H\ot H}(H\ot \Delta),
    \\&\ep\alpha^L_H &&= \ep\mu
    = \ep\ot\ep
   = (\ep\ot I)(H\ot \ep)
    =\alpha^L_I(H\ot \ep).
\end{alignat*}
Similarly, one proves both to be right $H$-linear. Hence it remains to verify conditions (\ref{eq:unit1}) and (\ref{eq:unit2}). Suppressing the canonical associativity and unit constraints, these identities become
\begin{align}
    &\label{eqn: unit1}\zeta_{H,H,M,N}(\Delta\ot_H M\ot N) 
    =\id_{M\ot N},
    \\&\label{eqn: unit2}\zeta_{M,N,H,H}(M\ot N \ot_H \Delta) 
    =\id_{M\ot N},
    \\&\label{eqn: unit3}(m_I\ot M\ot_H N)\zeta_{I,M,I,N} 
    =\id_{M\ot_H N},
    \\&\label{eqn: unit4}(M\ot_H N\ot m_I)\zeta_{M,I,N,I} 
    =\id_{M\ot_H N}.
\end{align}
These are easily verified, again by utilizing the canonical epimorphisms of the concerning coequalizers. For example, identity (\ref{eqn: unit1}) holds as $q_{H,M\ot N} = \alpha^L_{M\ot N}$ and hence
\begin{align*}
    \zeta_{H,H,M,N}(\Delta\ot_H M\ot N)q_{H,M\ot N} &= \zeta_{H,H,M,N}q_{H\ot H, M\ot N}(\Delta\ot M \ot N)
    \\&= (q_{H,M}\ot q_{H,N})(H\ot \sigma_{H,M}\ot N)(\Delta\ot M \ot N)
    \\&= (\alpha^L_M\ot \alpha^L_N)(H\ot\sigma_{H,M}\ot N)(\Delta\ot M\ot N)
    \\&= \alpha^L_{M\ot N}.
\end{align*}
Similarly, one verifies condition (\ref{eqn: unit2}), while conditions (\ref{eqn: unit3}) and (\ref{eqn: unit4}) follow more easily,
proving that $(_H\Mm_H, \ot_H, H, \ot, I)$ is a duoidal category.
\newline
The case of bicomodules is completely dual to a). We recall that $(^{H}\Mm^{H},\square^{H},H)$ is a monoidal category by 2) of Theorem \ref{theorem: bi(co)modules monoidal}, and $(^{H}\Mm^{H},\ot,I)$ is a monoidal category by Proposition \ref{prop:bi(co)modmonoidalot}. The interchange law for the duoidal structure on $(^H\Mm^H, \ot, I, \square^H, H)$ is the unique $H$-bicolinear morphism $\zeta_{U,V,W,T}$ induced by the $\Mm$-morphism
\[
\Xi_{U,V,W,T}\colon (U\square^H V)\ot (W\square^H T)\xrightarrow{e_{U,V}\ot e_{W,T}}U\ot V\ot W\ot T\xrightarrow{U\ot \sigma_{V,W}\ot T}U\ot W\ot V\ot T. \qedhere
\]
\begin{invisible}
which equalizes the diagram
\[\begin{tikzcd}
	{U\ot W\ot V\ot T} && {U\ot W\ot H\ot V\ot T.}
	\arrow["{\rho^R_{U\ot W}\ot V\ot T}", shift left=1.5, from=1-1, to=1-3]
	\arrow["{U\ot W\ot \rho^L_{V\ot T}}"', shift right=2, from=1-1, to=1-3]
\end{tikzcd}
\]
The monoid structure of $H$ in $(^{H}\Mm^{H},\ot,I)$ is given by its monoid structure $(H,m_{H},u_{H})$ in $\Mm$. The structure morphisms are of $H$-bicomodules since $H$ is in $\mathsf{Bimon}(\Mm)$. By definition of the $H$-bicomodule structure on $I$, the unitors $l_I^{-1}=r_I^{-1}\colon I \to I\ot I$ equalize the pair $(\rho^{R}_{I}\ot\id_{I},\id_{I}\ot\rho^{L}_{I})$ in $\Mm$. Hence there is a unique morphism $\Delta_I\colon I\to I\square^H I $ in $\Mm$ such that $e_{I,I}\Delta_I = l_I^{-1}$. An elementary computation shows that $(I, \Delta_I, \ep_{I})$ is a comonoid in $(^H\Mm^H, \square^H, H)$.
\end{invisible}
\end{proof}

\begin{remark}\label{remark: obstruction duoidal to tetramodules general}
    From the previous result one may be tempted to think that for any $H$ in $\Bimon(\Mm)$, $({}^{H}_{H}{\Mm}^H_H, \ot_H, H, \square^H, H)$ is immediately duoidal, since $^H_H\Mm^H_H\cong{} ^H(_H\Mm_H)^H\cong{} _H(^H\Mm^H)_H$. 
    However, in order to apply for instance b), we need $H$ to be a bimonoid in $_H\Mm_H$, which alters the constructed cotensor product since the bimonoid structure is altered. Hence, alike in Remark \ref{rmk:obstructionliftingduoidal}, the duoidal structure of bi(co)modules in a general setting does not immediately generalize to the category of tetramodules.
    
    The obtained duoidal structure is discussed in-depth in Examples \ref{remark: bimodules with k-tensor} and \ref{remark: bicomodules with k-tensor} below, and tetramodules in general categories are proven to be indeed duoidal in Theorem \ref{thm:duoidaltetramodules}.
\end{remark}

In the next section, we obtain some interesting examples of duoidal categories by applying Theorem \ref{theorem: bi(co)modules duoidal}.

\subsection{Examples of duoidal categories}\label{sec:Examplesduoidal} 

First, we recall the following facts that will be useful in the following and whose proof is straightforward. 

\begin{lemma}\label{lem:bimonIandbi(co)modonI}
    Let $(\Mm, \ot, I, a, l, r, \sigma)$ be a braided monoidal category. The following statements hold:
\begin{itemize}
    \item[1)] There are canonical monoid, comonoid, and bimonoid structures on $I$, namely (those underlying) $(I, l_I, \id_I, l_I^{-1}, \id_I)$.
    \item[2)] If $\Mm$ is coregular then $(_I\Mm_I, \ot_I, (I,l_I,r_I)) \cong (\Mm, \ot, I)$ are monoidally isomorphic. Dually, if $\Mm$ is regular then $(^I\Mm^I, \square^I,(I,l_I^{-1},r_I^{-1})) \cong (\Mm, \ot, I)$.
\end{itemize}
\end{lemma}
\begin{invisible}
\begin{proof}
    1). Notice that the proposed bimonoid structure is a well-defined monoid, comonoid, and bimonoid structure on $I$. Let us verify that these are unique. We will prove this unicity for the monoid structure on $I$, the proof for comonoid is completely analogous and these imply the case for the bimonoid structure.

    Let $(I,m,u)$ be a monoid structure on $I$. As the constraints of $\Mon(\Mm)$ are those of $\Mm$, it holds that $l_Iu_{I\ot I} = u$. By definition of the monoidal structure on $\Mon(\Mm)$, the left hand side equals $l_I(u\ot u)l_I^{-1}$, whence by naturality,
    \[
        l_I(u\ot u) = ul_I = l_I(I\ot u).
    \]
    Thus $u\ot u = I\ot u$, and since $I\ot u$ is a monomorphism this implies that $u\ot I = \id_{I\ot I}$, and hence $ur_I = r_I$. In conclusion, $u = \id_I$. Hence $m(u\ot I) = m = r_I = l_I$, concluding the proof.

2). Assume $\Mm$ to be regular, and recall that $I$ is uniquely a comonoid in $\Mm$, namely as $(I, l_I^{-1}, \id_I)$. Hence an object in ${}^I\Mm^I$ is a triple $(V, \rho^L, \rho^R)$ with $V$ in $\Mm$, and $\rho^L,\rho^R$ satisfying that
    \begin{gather*}
    (\id_I\ot I)\rho^L = l_V^{-1}, \hspace{3em}  (I\ot \id_I)\rho^R= r_V^{-1}, \hspace{3em} (I\ot \rho^L)\rho^L = (l_I^{-1}\ot V)\rho^L, \\(\rho^R\ot I)\rho^R = (V\ot l_I^{-1})\rho^R, \hspace{3em} (\rho^L\ot I)\rho^R = (I\ot \rho^R)\rho^L.
    \end{gather*}
    This implies immediately that $(V,\rho^L, \rho^R) = (V,l_V^{-1}, r_V^{-1})$, and the relations hold by coherence. Hence we define for any $V$ in $\Mm$ $F(V) = (V,l_V^{-1}, r_V^{-1})$ and conversely, for any object in ${}^I\Mm^I$, $G(W,\rho^L, \rho^R) = W$. By the above, $FG = \id$ and $GF = \id$. Notice that morphisms in $\Mm$ are immediately morphisms in ${}^I\Mm^I$, and letting $F$ and $G$ be trivial on morphisms results in an isomorphism of categories. For the monoidal structure, notice that for any $V, W$ in ${}^I\Mm^I$, $(V\ot W, \id_{V\ot W}) = (V\square^I W, e_{V,W})$. Hence the associators and constraints of both categories coincide, by their constructions in the Proposition \ref{prop: (co)regularity of bi(co)modules} and the preliminaries. By construction,
    \begin{align*}
    F(V) \square^I F(W) &= (V\ot W, \rho^L_{V\square W}, \rho^R_{V\square W}) = (V\ot W, \rho^L_V\ot W, V\ot \rho^R_W) 
    \\&= (V\ot W, l_{V}^{-1}\ot W, V\ot r_W^{-1}) = (V\ot W, l_{V\ot W}^{-1}, r_{V\ot W}^{-1}) = F(V\ot W)
    \end{align*}
    Hence $(F,\id, \id_I)$ and $(G,\id, \id_I)$ constitute a monoidal isomorphism.
    Likewise for the $I$-bimodules in case $\Mm$ is coregular.
\end{proof}
\end{invisible}
We now compute some examples. First, let $(\Mm,\ot,I,\sigma)=(\mathrm{Vec}_{\Bbbk},\ot,\Bbbk,\tau)$, which is known to be a regular and coregular symmetric monoidal category, and $H$ be a $\Bbbk$-bialgebra. By a) of Theorem \ref{theorem: bi(co)modules duoidal} one recovers Example \ref{ex:duoidalbimodules}, while by b) of Theorem \ref{theorem: bi(co)modules duoidal} one obtains the following dual example.

\begin{example}
    $(^H\mm^H,\ot, \Bbbk, \square^H, H)$ is a duoidal category, with interchange law defined by
    \begin{alignat*}{2}
        \zeta_{V,W,U,T}((v\square^H w)\ot (u\square^H t)) = (v\ot u)\square^H (w\ot t),
    \end{alignat*}
for all $V,W,U,T\in {}^{H}{\mm}^H$, and morphisms $\Delta_{I}:\Bbbk\to \Bbbk\square^{H}\Bbbk$, $m_{J}:H\ot H\to H$ and $\varepsilon_{I}:\Bbbk\to H$ given, respectively, by $\Delta_{I}(k)=1_{\Bbbk}\square^{H}k=k\square^{H}1_{\Bbbk}$, $m_{H}$, and $u_{H}$.
\end{example}

\begin{remark}\label{rmk:obstruction}
We observe that $u_{H}$ is a morphism of $H$-bimodules if and only if $H=\Bbbk1_{H}$
\begin{invisible}
    $u_{H}(h\cdot1_{\Bbbk}\cdot h')=u_{H}(\varepsilon(hh'))=\varepsilon(hh')1_{H}$, while $h\cdot u_{H}(1_{\Bbbk})\cdot h'=hh'$.
\end{invisible}
This shows that we cannot immediately lift the duoidal structure of ${}^H\mm^H$ to the category $^{H}_{H}\mm^{H}_{H}$, like we were unable to do so for bimodules (Remark \ref{rmk:obstructionliftingduoidal}). 
\end{remark}

Now we apply Theorem \ref{theorem: bi(co)modules duoidal} considering other braided monoidal categories. In order to do this, we recall how one can obtain braidings for the categories $(_{H}\mm_{H},\ot,\Bbbk)$ and $(^{H}\mm^{H},\ot,\Bbbk)$. These results are considered part of the folklore in the literature, and we have not found a specific reference. We include this remark to explain how they can be derived.

\begin{remark}\label{rmk:braidingbimodquasi}
We recall that the monoidal structure on $_{H}\mm_{H}$ is obtained by identifying $_{H}\mm_{H}$ with $_{H\ot H^{\mathrm{op}}}\mm$. Hence, braidings on $(_{H}\mm_{H},\ot,\Bbbk)$ correspond to braidings on $(_{H\ot H^{\mathrm{op}}}\mm,\ot,\Bbbk)$, and thus to $\Rr$-matrices for $H\ot H^{\mathrm{op}}$. Moreover, by e.g.\ \cite[Theorem 2.9]{MR1624475}, the latter can be classified in terms of $\Rr$-matrices of $H$, $\Rr$-matrices of $H^{\mathrm{op}}$, and central ``weak $\Rr$-matrices'' of $(H,H^{\mathrm{op}})$ (see \cite[Definition 1.1]{MR1624475}). We point out that we are looking at \cite[Theorem 2.9]{MR1624475}, with $A=H^{\mathrm{op}}$ and $R=1\ot1$, so that we recover the standard tensor product of bialgebras $H\ot H^{\mathrm{op}}$. Choosing the central weak $\Rr$-matrices $u$ and $v$ equal to $1_{H}\ot1_{H}$, the general quasitriangular structure on $H\ot H^{\mathrm{op}}$ given in \cite[Theorem 2.9]{MR1624475}  is $\Rr_{\ot}=(\id\ot\tau\ot\id)(\Rr\ot\Rr')$ with $\Rr$ and $\Rr'$ quasitriangular structures of $H$ and $H^{\mathrm{op}}$, respectively. 

Let $(H,\Rr)$ be a quasitriangular bialgebra. Then $\Rr^\op$ is a quasitriangular structure for $H^\op$, see e.g.\ \cite[Exercise 2.1.3]{Majid-book}. Therefore, 
$(\id\ot\tau\ot\id)(\Rr\ot\Rr^\op)$ is a quasitriangular structure for $H\ot H^\op$. 
By identifying left $H\ot H^\op$-modules with $H$-bimodules, we get that the monoidal category $({}_H\mathfrak{M}_H,\ot,\Bbbk)$ is braided, where, for all $X,Y\in{}_H\mathfrak{M}_H$, the braiding $\sigma_{X,Y}:X\ot Y\to Y\ot X$ is given by $\sigma_{X,Y}(x\ot y)=\Rr^{\mathrm{op}}(y \ot x)\Rr$. 
\begin{invisible}
 Set $\Ss:=(\id\ot\tau\ot\id)(\Rr\ot\Rr^\op)
=\Rr^i\ot\Rr_j\ot\Rr_i\ot\Rr^j.$ 
Then $\Ss^\op(y\ot x)
=(\Rr_i\ot\Rr^j\ot\Rr^i\ot\Rr_j)(y\ot x)
=(\Rr_i\ot\Rr^j)y\ot ( \Rr^i\ot\Rr_j)x
=\Rr_iy\Rr^j\ot\Rr^ix\Rr_j
=\Rr^{\mathrm{op}}(y\ot x)\Rr.$
\end{invisible}
Moreover, we observe that one can also consider $(\Rr^{-1})^{\op}$ as a quasitriangular structure for $H$, obtaining the quasitriangular structure $(\id\ot\tau\ot\id)((\Rr^{-1})^{\mathrm{op}}\ot\Rr^\op)$ for $H\ot H^{\mathrm{op}}$. The induced braiding $\sigma_{X,Y}:X\ot Y\to Y\ot X$ on $(_{H}\mm_{H},\ot,\Bbbk)$ is given by $\sigma_{X,Y}(x\ot y)=\Rr^{-1}(y \ot x)\Rr$. 
\begin{invisible}
 Set $\Ss:=(\id\ot\tau\ot\id)((\Rr^{-1})^{\op}\ot\Rr^\op)
=\overline{\Rr}_i\ot\Rr_i\ot\overline{\Rr}^i\ot\Rr^i.$ 
Then $\Ss^\op(y\ot x)
=(\overline{\Rr}^i\ot\Rr^i\ot \overline{\Rr}_i\ot\Rr_i)(y\ot x)
=(\overline{\Rr}^i\ot\Rr^i)y\ot ( \overline{\Rr}_i\ot\Rr_i)x
=\overline{\Rr}^iy\Rr^i\ot  \overline{\Rr}_ix\Rr_i
=\Rr^{-1}(y\ot x)\Rr.$
\end{invisible}
This braiding will have an advantage with respect to the previous one once we will consider the duoidal setting (see Proposition \ref{prop:braidedduoidalbimod}). 

Dually, given a coquasitriangular bialgebra $(H,\Rr:H\ot H\to\Bbbk)$, $(H^{\cop},\Rr^{\op})$ and $(H^{\cop}, \Rr^{-1})$ are coquasitriangular bialgebras, with $\Rr^{\op} = \Rr\tau_{H,H}$ and $\Rr^{-1}$ the convolution inverse of $\Rr$. In particular, $(H, (\Rr^{-1})^{\op})$ is a co(quasi)triangular bialgebra.
\begin{invisible}
        The fact that $(H^{\cop}, \Rr^{\op})$ is a coquasitriangular bialgebra is an immediate verification, while the claim on $(H^{\cop}, \Rr^{-1})$ is noticed right after \cite[Definition 5.1]{Ferri2025}. The last claim now follows immediately by combining the former two, noticing that $(\Rr^{-1})^{\op} = (\Rr^{\op})^{-1}$.
\end{invisible}
The category $(^{H}\mm^{H},\ot,\Bbbk)$ is braided with braiding defined, for all $X,Y\in{}^H\mathfrak{M}^H$, by $\sigma_{X,Y}:X\ot Y\to Y\ot X$, $x\ot y\mapsto \Rr^{-1}(x_{(-1)}\ot y_{(-1)})y_{(0)}\ot x_{(0)}\Rr(x_{(1)}\ot y_{(1)})$. This may be derived from the isomorphism with the category of right $H\ot H^{\cop}$-comodules \cite[Lemma 4.5.11 and Exercise 9.2.9]{Majid-book}, whose braidings are classified \cite[Theorem 2.10]{chen2} (with $\tau = \nu = \nu' = \ep\ot\ep$). Historically, braidings on the category of left comodules over a bialgebra were first studied in \cite[Theorem 2.7]{Larson1991}.

\end{remark}

\begin{example}\label{example: bimodules over smash products duoidal}
Consider a $\Bbbk$-bialgebra $H$ and the monoidal category $(_{H}\mm_{H},\ot,\Bbbk)$. Suppose that the latter category is braided, which happens when $H$ is a quasitriangular bialgebra as explained in Remark \ref{rmk:braidingbimodquasi}. By Proposition \ref{prop:bi(co)modmonoidalot}, $(_{H}\mm_{H},\ot,\Bbbk)$ is regular and coregular since $(\mathsf{Vec}_{\Bbbk},\ot,\Bbbk)$ is.

Let $K$ be an object in $\mathsf{Bimon}(_{H}\mm_{H},\ot,\Bbbk)$. In particular, it is an object in $\mathsf{Mon}(_{H}\mm,\ot,\Bbbk)$, $\mathsf{Mon}(\mm_{H},\ot,\Bbbk)$, and in $\mathsf{Comon}(_{H}\mm,\ot,\Bbbk)$, $\mathsf{Comon}(\mm_{H},\ot,\Bbbk)$. Then, by Theorem \ref{theorem: bi(co)modules duoidal}, we obtain the following duoidal categories: 
\[
        \left({}_{K}{({}_{H}{\mm}_H)_K}, \ot_K, K, \ot, \Bbbk\right), \hspace{5em} \left({}^{K}{({}_{H}{\mm}_H)}^K
        , \ot, \Bbbk, \square^K, K\right).
    \]
We point out that we have functors $X\ot_{K}(-):{}_{K}(_{H}\mm_{H})\to{}_{H}\mm_{H}$, $(-)\ot_{K}X:({}_{H}\mm_{H})_{K}\to{}_{H}\mm_{H}$ and $X\square^{K}(-):{}^{K}(_{H}\mm_{H})\to{}_{H}\mm_{H}$, $(-)\square^{K}X:{}(_{H}\mm_{H})^{K}\to{}_{H}\mm_{H}$, for suitable objects $X\in{}_{H}\mm_{H}$. 
\end{example}

\begin{remark}\label{remark: bimodules with k-tensor}
We observe $H$ is in $\Bimon({}_{H}{\mm}_H, \ot, \Bbbk, \sigma)$ for its $\Bbbk$-linear bialgebra structure if and only if $H\cong \Bbbk$. Indeed, it follows easily that $H$ is an object in $\Comon({}_{H}{\mm}_H, \ot, \Bbbk, \sigma)$ but it is in $\Mon({}_{H}{\mm}_H, \ot, \Bbbk, \sigma)$ if and only if $h = \ep(h)1_H$ for all $h\in H$ (see Remark \ref{rmk:obstruction}), which defines a bialgebra isomorphism from $H$ to $\Bbbk$. In this case, by utilizing Lemma \ref{lem:bimonIandbi(co)modonI}, Example \ref{example: bimodules over smash products duoidal} instantializes to the following duoidal categories;
    \begin{align*}
        &\left({}_{H}{({}_{H}{\mm}_H)}_H, \ot_H, H, \ot, \Bbbk\right) \cong ({}_{\Bbbk}{\mm}_{\Bbbk}, \ot_{\Bbbk}, \Bbbk, \ot, \Bbbk) \cong (\Vec_{\Bbbk}, \ot, \Bbbk, \ot, \Bbbk),
        \\&\left({}^{H}{({}_{H}{\mm}_H)^H}, \ot, \Bbbk, \square^H, H\right) \cong ({}^{\Bbbk}{\mm}^{\Bbbk}, \ot, \Bbbk, \square^{\Bbbk}, \Bbbk) \cong (\Vec_{\Bbbk}, \ot, \Bbbk, \ot, \Bbbk).
    \end{align*}
    Hence both correspond to the datum of $(\Vec_{\Bbbk},\ot,\Bbbk)$ endowed with the symmetry $\tau$. 
\end{remark}

\begin{example}\label{example: bicomodules over smash coproduct duoidal}
Consider a $\Bbbk$-bialgebra $H$ and the monoidal category $(^{H}\mm^{H},\ot,\Bbbk)$. Suppose that the latter category is braided. This happens when $H$ is a coquasitriangular bialgebra as explained in Remark \ref{rmk:braidingbimodquasi}. By Proposition \ref{prop:bi(co)modmonoidalot}, $(^{H}\mm^{H},\ot,\Bbbk)$ is regular and coregular since $(\mathsf{Vec}_{\Bbbk},\ot,\Bbbk)$ is. 

Let $K$ be an object in $\mathsf{Bimon}(^{H}\mm^{H})$. In particular, it is an object in $\mathsf{Mon}(^{H}\mm)$, $\mathsf{Mon}(\mm^{H})$, and in $\mathsf{Comon}(^{H}\mm)$, $\mathsf{Comon}(\mm^{H})$. Then, by Theorem \ref{theorem: bi(co)modules duoidal}, we obtain the following duoidal categories: 
\begin{align*}
        &\left({}_{K}{(^H{\mm}^H)_K}
        , \ot_K, K, \ot, \Bbbk\right), \hspace{5em} \left({}^{K}{(^{H}{\mm}^H)^K}, \ot, \Bbbk, \square^K, K\right).
    \end{align*}
\begin{invisible}
    If $H$ is a $\Bbbk$-bialgebra, then $(\Mm^H, \ot, \Bbbk)$ is monoidal for the diagonal coaction. Let $(C, \rho^H_C)$ be in $\Comon(\Mm^H, \ot, \Bbbk)$, this is a right $H$-comodule $C$ with a coalgebra structure for which its structure morphisms are $H$-comodule morphisms. Analogous to \cite{HecSch}, $H\# C$ is a coalgebra for the structure morphisms
    \begin{align*}
        &\Delta_{H\# C} = (\rho^H_{H\ot C}\ot C)(H\ot \Delta_C) = (H\ot C\ot \mu_H\ot C)(H\ot \sigma_{H,C}\ot H\ot C)(\Delta_H\ot \rho^H_C\ot C)(H\ot \Delta_C),
        \\&\ep_{H\# C} = \ep_H\ot \ep_C.
    \end{align*}
    The below two functors, which are trivial on morphisms, form an isomorphism of categories $(\Mm^H)^C \cong \Mm^{H\# C}$, utilizing that the coalgebra structure of $C$ is $H$-colinear.
    \begin{align*}
        &F\colon (\Mm^H)^C\to \Mm^{H\# C} \colon ((V,\rho_V^H), \rho^C_{V})\mapsto (V,\rho^{H\# C}_V),
        \hspace{2em} \rho^{H\# C}_V = (\rho^H_V\ot C)\rho_V^C
        \\&G\colon \Mm^{H\# C}\to (\Mm^H)^C\colon (W, \rho^{H\#C}_W)\mapsto ((W,\rho^H_W),\rho^C_W),
        \\&\hspace{3em} \rho_W^H = (W\ot H\ot \ep_C)\rho_W^{H\# C}, \hspace{3em} \rho^C_W = (W\ot \ep_H\ot C)\rho^{H\# C}_W.
    \end{align*}
    Notice that the right $C$-coaction depends on the right $H$-coaction $\rho^H_V$, and so more rigorously we should denote $\rho^C_{(V,\rho^H_V)}$, yet this is quite cumbersome notation.
\end{invisible}
\end{example}

\begin{remark}\label{remark: bicomodules with k-tensor}
     We observe that $H$ is in $\Bimon({}^{H}{}{\mm}^H, \ot, \Bbbk, \sigma)$ if and only if $H\cong \Bbbk$. Indeed, it follows easily that $H$ is in $\Mon({}^{H}{\mm}^H, \ot, \Bbbk)$ but it is in $\Comon({}^{H}{\mm}^H, \ot, \Bbbk)$ if and only if $h = \ep(h)1_H$ for all $h\in H$ (see Remark \ref{rmk:obstructionliftingduoidal}), which defines a bialgebra isomorphism from $H$ to $\Bbbk$. In this case, by utilizing Lemma \ref{lem:bimonIandbi(co)modonI}, Example \ref{example: bicomodules over smash coproduct duoidal} instantializes to the following duoidal categories;
    \begin{align*}
        &\left( {}_{H}{({}^{H}{\mm^H}})_H, \ot_H, H, \ot, \Bbbk \right) \cong ({}^{\Bbbk}{\mm}^\Bbbk, \ot_\Bbbk, \Bbbk, \ot, \Bbbk) \cong (\Vec_{\Bbbk}, \ot, \Bbbk, \ot, \Bbbk),
        \\&\left( {}^{H}{({}^{H}{\mm^H}})^H, \ot, \Bbbk, \square^H, H \right) \cong ({}^{\Bbbk}{\mm}^\Bbbk, \ot, \Bbbk, \square^\Bbbk, \Bbbk) \cong (\Vec_{\Bbbk}, \ot, \Bbbk, \ot, \Bbbk).
    \end{align*}
    Hence both correspond to the datum of $(\Vec_{\Bbbk},\ot,\Bbbk)$ endowed with the symmetry $\tau$.
\end{remark}

Finally, we consider the category $(_{A}\mm_{A},\ot_{A},A)$ with $A$ an algebra and $(^{C}\mm^{C},\square^{C},C)$, with $C$ a coalgebra.

\begin{example}
   Given an algebra $A$, then $({}_{A}{\mm}_A, \ot_A,A)$ is a coregular monoidal category by Theorem \ref{theorem: bi(co)modules monoidal} and Proposition \ref{prop: (co)regularity of bi(co)modules}, since $(\mathsf{Vec}_{\Bbbk},\ot,\Bbbk)$ is. In general, as observed in Remark \ref{rmk:regularityotA}, $({}_{A}{\mm}_A, \ot_A, A)$ is not regular. Suppose that $({}_{A}{\mm}_A, \ot_A, A)$ is braided, this is there exists a canonical $R$-matrix $R$ for $A$ (see Theorem \ref{thm:braidingsbimodtensorA}). Denote its corresponding braiding by $\sigma^R$, and consider $K\in\mathsf{Bimon}(_{A}\mm_{A},\ot_{A},A,\sigma^R)$, then by Theorem \ref{theorem: bi(co)modules duoidal} we obtain the following duoidal category
\[
(_{K}(_{A}\mm_{A})_{K},(\ot_{A})_{K},K,\ot_{A},A).
\]
Given $M,N$ in ${}_{K}(_{A}\mm_{A})_{K}$, denoting by $\triangleright$ and $\triangleleft$ the left and right $K$-actions on $M,N$ \begin{invisible}(which are $A$-bilinear maps),\end{invisible} one has $M(\ot_{A})_{K}N=M\ot_A N / \langle (m\tl k) \ot_A n - m\ot_A (k\tr n)\rangle.
$
\end{example}

\begin{remark}
    One can consider $A$ as a bimonoid in $(_{A}\mm_{A},\ot_{A},A, \sigma^R)$ for the identity morphisms by Lemma \ref{lem:bimonIandbi(co)modonI}. The previous example instantializes to
    \begin{align*}
        &\left({}_{A}{({}_{A}{\mm_A})_A}, (\ot_A)_A, A, \ot_A, A\right) \cong ({}_{A}{\mm}_A, \ot_A, A, \ot_A, A),
    \end{align*}
    The interchange law is induced by the braiding as discussed in Remark \ref{remark: strong duoidal is braided}. Hence this corresponds to the datum of the braided monoidal category $({}_{A}{\mm}_A, \ot_A, A, \sigma^R)$.  
\end{remark}

\begin{example}
    Consider a coalgebra $C$, then $({}^{C}{\mm}^C, \square^C,C)$ is a regular monoidal category by Theorem \ref{theorem: bi(co)modules monoidal} and Proposition \ref{prop: (co)regularity of bi(co)modules},  since $(\mathsf{Vec}_{\Bbbk},\ot,\Bbbk)$ is. In general, as observed in Remark \ref{rmk:regularityotA}, $({}^{C}{\mm}^C, \square^C, C)$ is not coregular. Suppose that $({}^{C}{\mm}^C, \square^C, C)$ is braided, this is there exists a canonical $R$-form $\RR$ for $C$ (see Theorem \ref{theorem: braidings on bicomodules}). Denote its corresponding braiding by $\sigma^{\RR}$, and consider $K\in\mathsf{Bimon}(^{C}\mm^{C},\square^{C},C,\sigma^{\RR})$. By Theorem \ref{theorem: bi(co)modules duoidal} we obtain the following duoidal category
\[
(^{K}(^{C}\mm^{C})^{K},\square^{C},C,(\square^{C})^{K},K)
\]
Given $M,N$ in $^{K}(^{C}\mm^{C})^{K}$, differentiating between $C$-coactions and $K$-coactions by a superscript, one has that $M(\square^{C})^{K}N=\{m\square^{C}n\ |\ m_{(0)}^{K}\square^Cm^{K}_{(1)}\square^{C}n=m\square^{C} n_{(-1)}^{K}\square^C n_{(0)}^{K}\}$. 
\end{example}

\begin{example}
    One can consider $C$ as a bimonoid in $(^{C}\mm^{C},\square^C,C,\sigma^{\RR})$ for the identity morphisms by Lemma \ref{lem:bimonIandbi(co)modonI}. 
    The previous example instantializes to
    \begin{align*}
        \left( {}^{C}{({}^{C}{\mm^C}})^C, \square^C, C, (\square^C)^C, C \right)\cong ({}^{C}{\mm}^C, \square^C, C, \square^C, C).
    \end{align*}
    The interchange law is induced by the braiding $\sigma^{\RR}$ as discussed in Remark \ref{remark: strong duoidal is braided}. Hence this corresponds to the datum of the braided monoidal category $({}^{C}{\mm}^C, \square^C, C, \sigma^{\RR})$. 
\end{example}

\section{On the duoidal category of tetramodules and Yetter--Drinfeld modules}\label{sec:tetramodulesduoidal}

We recall that tetramodules over a $\Bbbk$-bialgebra $H$ can be structured as a duoidal category $(^{H}_{H}\mm^{H}_{H},\ot_{H},H,\square^{H},H)$ (in fact, a 2-fold monoidal category); this was proven in \cite[Theorem 3.6]{Shoikhet}. We prove a generalization of this construction, and view said result as an example thereof in Example \ref{example: tetramodules duoidal}.

\begin{theorem}\label{thm:duoidaltetramodules}
    Let $(\Mm,\ot,I,\sigma)$ be a braided (strict) monoidal category, and $H$ in $\Bimon(\Mm)$. If $\Mm$ is regular and coregular, then $({}^{H}_{H}{\Mm}^H_H, \ot_H, H, \square^H, H)$ is duoidal. 
\end{theorem}
\begin{proof}
As recalled in Proposition \ref{prop:monoidaltetramodules}, $(_{H}^{H}\Mm^{H}_{H},\ot_{H},H)$ and $(_{H}^{H}\Mm^{H}_{H},\square^{H},H)$ are monoidal categories. Let us construct a suitable interchange law. For any tetramodules $M,N,P,Q$ we define $\Xi_{M,N,P,Q}$ as the composition making the below diagram commute.
\[\begin{tikzcd}
	{(M\square^H N) \ot (P\square^H Q)} && {(M\ot_H P)\ot(N\ot_H Q).} \\
	{M\ot N\ot P\ot Q} && {M\ot P\ot N\ot Q}
	\arrow["{\Xi_{M,N,P,Q}}", dashed, from=1-1, to=1-3]
	\arrow["{e_{M,N}\ot e_{P,Q}}"', from=1-1, to=2-1]
	\arrow["{M\ot\sigma_{N,P}\ot Q}"', from=2-1, to=2-3]
	\arrow["{q_{M,P}\ot q_{N,Q}}"', from=2-3, to=1-3]
\end{tikzcd}\]
One verifies that $\Xi_{M,N,P,Q}$ equalizes $\rho^R_{M\ot_H P}\ot (N\ot_H Q)$ and $(M\ot_H P)\ot \rho^L_{N\ot_H Q}$.
\begin{invisible} For notational ease, we will often write $\Xi$ instead of $\Xi_{M,N,P,Q}$. Using that coequalizers in ${}^{H}_{H}{\Mm}^H_H$ are equipped with the unique $H$-bicomodule structures rendering their canonical morphism $H$-bicolinear by Remark \ref{rmk:forgetfulbimod}, we compute that 
\begin{align*}
    (&\rho^R_{M\ot_H P}\ot (N\ot_H Q))\Xi
    \\&= (\rho^R_{M\ot_H P}\ot (N\ot_H Q))(q_{M, P} \ot q_{N,Q})(M\ot \sigma_{N,P}\ot Q)(e_{M,N}\ot e_{P,Q})
    \\&=(q_{M,P}\ot H\ot q_{N,Q})(\rho^R_{M\ot P}\ot N\ot Q)(M\ot \sigma_{N,P}\ot Q)(e_{M,N}\ot e_{P,Q})
    \\&= (q_{M,P}\ot H\ot q_{N,Q})(M\ot P\ot \mu\ot N\ot Q)(M\ot \sigma_{H,P}\ot H\ot N\ot Q)
    \\&\hspace{17em}(\rho^R_M \ot \rho^R_P\ot N\ot Q)(M\ot \sigma_{N,P}\ot Q)(e_{M,N}\ot e_{P,Q})
    \\&= (q_{M,P}\ot H\ot q_{N,Q})(M\ot P\ot \mu\ot N\ot Q)(M\ot \sigma_{H,P}\ot H\ot N\ot Q)
    \\&\hspace{17em}(M\ot H\ot \sigma_{N,P\ot H}\ot Q)(\rho^R_M \ot N\ot \rho^R_P\ot Q)(e_{M,N}\ot e_{P,Q})
    \\&= (q_{M,P}\ot H\ot q_{N,Q})(M\ot P\ot \mu\ot N\ot Q)(M\ot \sigma_{H,P}\ot H\ot N\ot Q)
    \\&\hspace{17em}(M\ot H\ot \sigma_{N,P\ot H}\ot Q)(M \ot \rho^L_N\ot P\ot \rho^L_Q)(e_{M,N}\ot e_{P,Q})
    \\&= (q_{M,P}\ot H\ot q_{N,Q})(M\ot P\ot \mu\ot N\ot Q)(M\ot P \ot H\ot  \sigma_{N,H}\ot Q)
    \\&\hspace{17em}(M\ot \sigma_{H\ot N,P}\ot H\ot Q)(M \ot \rho^L_N\ot P\ot \rho^L_Q)(e_{M,N}\ot e_{P,Q})
    \\&= (q_{M,P}\ot H\ot q_{N,Q})(M\ot P\ot \mu\ot N\ot Q)(M\ot P \ot H\ot  \sigma_{N,H}\ot Q)
    \\&\hspace{17em}(M \ot P\ot \rho^L_N\ot \rho^L_Q)(M\ot \sigma_{N,P}\ot H\ot Q)(e_{M,N}\ot e_{P,Q})
    \\&= (q_{M,P}\ot H\ot q_{N,Q})(M\ot P \ot \rho^L_{N\ot Q})(M\ot \sigma_{N,P}\ot H\ot Q)(e_{M,N}\ot e_{P,Q})
    \\&= ((M\ot_H P) \ot \rho^L_{N\ot_H Q})(q_{M,P}\ot q_{N,Q})(M\ot \sigma_{N,P}\ot H\ot Q)(e_{M,N}\ot e_{P,Q})
    \\&= ((M\ot_H P) \ot \rho^L_{N\ot_H Q})\Xi.
\end{align*}
\end{invisible}
Hence, there exists a unique morphism $\gamma\colon (M\square^H N)\ot(P\square^H Q)\to (M\ot_H P)\square^H (N\ot_H Q)$ in $\Mm$ such that $\Xi_{M,N,P,Q} = e_{M\ot_H P, N\ot_H Q}\gamma$.
We claim $\gamma$ to coequalize $\alpha^R_{M\square^H N}\ot (P\square^H Q)$ and $(M\square^H N) \ot \alpha^L_{P\square^H Q}$. Since $e_{M\ot_H P, N\ot_H Q}$ is a monomorphism, this is equivalent to proving that $\Xi_{M,N,P,Q}$ does so, which follows analogously to the coequalizing \begin{invisible}proven\end{invisible} above. Hence there exists a unique morphism $\zeta_{M,N,P,Q}\colon (M\square^H N)\ot_H (P\square^H Q)\to (M\ot_H P)\square^H (N\ot_H Q)$ in $\Mm$ satisfying $\zeta_{M,N,P,Q} q_{M\square^H N, P\square^H Q} = \gamma$.
\begin{invisible}Again, as the canonical morphism of an equalizer is always a monomorphism,\end{invisible} $\zeta_{M,N,P,Q}$ can be characterized as the unique $\Mm$-morphism satisfying that
\begin{equation}\label{eqn: UP of zeta}
    e_{M\ot_H P, N\ot_H Q}\zeta_{M,N,P,Q} q_{M\square^H N, P\square^H Q} = \Xi_{M,N,P,Q}.
\end{equation}
In order to prove that $\zeta_{M,N,P,Q}$ is a morphism of tetramodules, one writes down the desired identities, pre- and post-composes these with the correct epi- respectively monomorphisms stemming from the corresponding (co)equalizer to obtain an identity on $\Xi_{M,N,P,Q}$. The latter are easily verified by the tetramodule structure of the considered (co)equalizer. Similarly, one proves that $\zeta_{M,N,P,Q}$ defines a natural transformation by reverting back to the corresponding identity on $\Xi_{M,N,P,Q}$ and verifying this. Let us prove at last that $\zeta$ is an interchange law for the considered monoidal structures. Conditions \eqref{eq:assoc1} and \eqref{eq:assoc2} become
\begin{align*}
    &(a_{A,C,E}^{\ot_H}\square^H a_{B,D,F}^{\ot_H})\zeta_{A\ot_H C, B\ot_H D, E, F}(\zeta_{A,B,C,D}\ot_H (E\square^H F))
    \\&\hspace{5em}= \zeta_{A,B,C\ot_H E, D\ot_H F}((A\square^H B)\ot_H \zeta_{C,D,E,F})a^{\ot_H}_{A\square B, C\square D, E\square F},
    \\&a^{\square}_{A\ot_H D, B\ot_H E, C\ot_H F}(\zeta_{A,B,D,E}\square^H (C\ot_H F))\zeta_{A\square^H B, C, D\square^H E, F}
    \\&\hspace{5em} =((A\ot_H D)\square^H \zeta_{B,C,E,F})\zeta_{A,B\square C, D, E\square F}(a^{\square}_{A,B,C}\ot_H a^{\square}_{D,E,F}).
\end{align*}
These can be verified by pre- and postcomposing with the appropriate canonical morphisms of the coequalizer respectively equalizer, utilizing the defining properties of the associators $a^{\ot_H}$ and $a^{\square}$ as discussed in Remark \ref{remark: constraints of tetramodules}. Explicit computations are omitted due to their tedious nature.

Lastly, we need morphisms
\[
\Delta_H \colon H\to H\square^H H, \hspace{2em} m_H \colon H\ot_H H\to H, \hspace{2em} \ep_H = u_H\colon H\to H
\]
in ${}^{H}_{H}{\Mm}^H_H$ such that $(H, m_H, u_H)$ is in $\Mon({}^{H}_{H}{\Mm}_H^H, \ot_H, H)$, $(H, \Delta_H, \ep_H)$ is in $\Comon({}^{H}_{H}{\Mm}^H_H, \square^H, H)$. 
By using the notation of Remark \ref{remark: constraints of tetramodules}, $(H, m_H, u_H) = (H, \Upsilon_H = \Upsilon'_H, \id_H)$ and $(H, \Delta_H, \ep_H) = (H, \Lambda_H = \Lambda'_H, \id_H)$. The identities of \eqref{eq:unit1}, \eqref{eq:unit2} translate to, for any $H$-tetramodules $M$ and $N$,
\begin{align*}
    &(\Upsilon'_M\square \Upsilon'_N)\zeta_{H,H,M,N}(\Delta_H\ot_H (M\square N)) = \Upsilon'_{M\square N},
    \\&(\Upsilon_M\square \Upsilon_N)\zeta_{M,N,H,H}((M\square N)\ot_H \Delta_H) = \Upsilon_{M\square N},
    \\&(m_H\square (M\ot_H N))\zeta_{H,M,H,N}(\Lambda'_M\ot_H \Lambda'_N) = \Lambda'_{M\ot_H N},
    \\&((M\ot_H N)\square m_H)\zeta_{M,H,N,H}(\Lambda_M\ot_H \Lambda_N) = \Lambda_{M\ot_H N}.
\end{align*}
These are easily verified by pre- and postcomposing with the appropriate canonical morphisms of the concerning coequalizer respectively equalizer.
\end{proof}

\begin{example}[{\cite[Theorem 3.6]{Shoikhet}}]\label{example: tetramodules duoidal}
    Let us consider the case of vector spaces. Notice that both $({}^H_H\mm^H_H, \ot_H, H)$ and $({}^H_H\mm^H_H, \square^H, H)$ are monoidal categories by Proposition \ref{prop:monoidaltetramodules}.
Given $M,N,P,Q$ in $^{H}_{H}\mm^{H}_{H}$, one defines the interchange law as in Theorem \ref{thm:duoidaltetramodules} to obtain
    \begin{equation}\label{eq:interchangetetramodules}
    \begin{split}
    \zeta_{M,N,P,Q}:\ &(M\square^{H}N)\ot_{H}(P\square^{H}Q)\longrightarrow(M\ot_{H}P)\square^{H}(N\ot_{H}Q),
    \\&(m\square n)\ot_H (p\square q) \longmapsto (m\ot_H p)\square (n\ot_H q).
    \end{split}
    \end{equation}
    Explicitly, consider the linear map
\[\begin{tikzcd}[cramped]
	{(M\square^H N)\ot (P\square^H Q)} && {M\ot P\ot N\ot Q} & \\
	& {M\ot N \ot P\ot Q} && {\hspace{-2em}(M\ot_H P)\ot (N\ot_H Q).}
	\arrow["{e_{M,N}\ot e_{P,Q}}"', from=1-1, to=2-2]
	\arrow["{q_{M,P}\ot q_{N,Q}}", from=1-3, to=2-4]
	\arrow["{M\ot \tau_{N,P}\ot P}"', from=2-2, to=1-3]
\end{tikzcd}\]
As in Theorem \ref{thm:duoidaltetramodules}, this map takes values in
$(M\otimes_H P)\square^H(N\otimes_H Q)$ and is $H$-balanced.
The universal properties therefore yield the interchange map \eqref{eq:interchangetetramodules}. Moreover, one defines the maps $\Delta^{\bullet}_{H}:H\to H\square^{H}H$, $h\mapsto h_1\square^{H}h_2$, $m^{\circ}_{H}:H\ot_{H}H\to H$, $h\ot_{H}h'\mapsto hh'$, $\varepsilon^{\bullet}_{H}=\id_{H}$.
\end{example}

One is able to explicitly construct the duoidal structure of tetramodules in a general setting.
\begin{proposition}[{\cite[Proposition 3.2.1]{Drabant}}]\label{prop:splittingidempotent}
Let $(\Mm, \ot, I,\sigma)$ be a braided monoidal category and $H$ be an object in $\Hopfmon(\Mm, \ot, I, \sigma)$. For any $(M,\alpha^L, \rho^L)$ in ${}^H_H\Mm$, the morphism
\begin{equation}\label{eqn: idempotent}
    {}_M\Pi \colon M\xrightarrow{\rho^L}H\ot M\xrightarrow{S\ot M} H\ot M\xrightarrow{\alpha^L}M
\end{equation}
is an idempotent, this is ${}_M\Pi\circ {}_M\Pi = {}_M\Pi$.\newline
If ${}_M\Pi$ is a split idempotent, i.e.\ there exists a triple $(_HM, i_M, p_M)$ with $_HM$ an object in $\Mm$ and $i_M$, $p_M$ morphisms in $\Mm$ such that $i_Mp_M = {}_M\Pi$ and $p_Mi_M = \id_{_HM}$, then
\begin{equation}\label{eqn: (co)invariants}
    (_HM, i_M) = \Eq\left( 
\begin{tikzcd}
	M & {H\ot M}
	\arrow["{\rho^L}", shift left=1.5, from=1-1, to=1-2]
	\arrow["{\eta\ot M}"', shift right=1.5, from=1-1, to=1-2]
\end{tikzcd}\right), \hspace{3em} (_HM, p_M) = \Coeq\left( 
\begin{tikzcd}
	{H\ot M} & M
	\arrow["{\alpha^L}", shift left=1.5, from=1-1, to=1-2]
	\arrow["{\ep\ot M}"', shift right=1.5, from=1-1, to=1-2]
\end{tikzcd}\right).
\end{equation}
In particular, $_HM$ realizes both the coinvariants of the left coaction and the invariants of the left action of $M$.
\end{proposition}
We say that a category $\Mm$ has \emph{split idempotents} if any idempotent in $\Mm$ is split. It is an elementary exercise to prove that if $\Mm$ has equalizers or coequalizers, then it has split idempotents, hence the above results apply to our setting wherein we assume the existence of equalizers and coequalizers.
\begin{corollary}
Let $H$ be a $\Bbbk$-Hopf algebra and let $M$ be in ${}^H_H\mm$. The canonical quotient map induces an isomorphism
\[
{}^{\mathrm{co}H}M \longrightarrow
\Bbbk\otimes_H M \cong M/H^+M,
\ m\longmapsto [m],
\]
where $H^+=\ker\varepsilon$.
Its inverse is induced by $m\mapsto S(m_{(-1)})m_{(0)}$.
\end{corollary}
\begin{lemma}[{\cite[Lemma 3.3.1 and Lemma 3.3.3]{Drabant}}]\label{lemma: left Hopf module isomorphisms}
    Let $(\Mm,\ot, I,\sigma)$ be a braided monoidal category, $H$ be an object in $\Hopfmon(\Mm,\ot, I, \sigma)$, and $M$ be an object in ${}^H_H\Mm$. If $\Mm$ has split idempotents, then $H\ot (_HM)$ is in ${}^H_H\Mm$ for the structure morphisms $\alpha^L = \mu\ot (_HM)$ and $\rho^L = \Delta\ot (_HM)$. Moreover,
    \begin{align}\label{eqn: inverse tetramodule morphisms}
        \phi_M& \colon H\ot (_HM)\xrightarrow{H\ot i_M} H\ot M\xrightarrow{\alpha^L_M} M,
        \\\psi_M&\colon M\xrightarrow{\rho^L_M} H\ot M\xrightarrow{H\ot p_M}H\ot (_HM)
    \end{align}
    are inverse morphisms in ${}^H_H\Mm$.
\end{lemma}
The following result will be used in the explicit construction of the $H$-linear tensor and $H$-colinear cotensor products of $H$-tetramodules (Theorem \ref{theorem: monoidal structures tetramodules}).
\begin{proposition}[{\cite[Proposition 3.4.1 and Theorem 4.3.1]{Drabant}}]\label{prop: tensor and cotensor by invariants}
    Let $(\Mm, \ot, I,\sigma)$ be a braided monoidal category with split idempotents and $H$ be an object in $\Hopfmon(\Mm,\ot, I, \sigma)$. If $M\in \prescript{H}{H}{\Mm}$, $N\in \Mm_H$, and $P\in \Mm^H$, then $N\ot_H M$ and $P\square^H M$ are well-defined objects of $\Mm$. Explicitly, they are given (up to isomorphism) by
    \begin{align}
        &N\ot_H M := N\ot (_HM) \text{ with }q_{N,M} \colon N\ot M \xrightarrow{N\ot \psi_M} N\ot H\ot (_HM) \xrightarrow{\alpha_N^R\ot (_HM)} N\ot (_HM),
        \\&P\square^H M := P\ot ({}_{H}{M}) \text{ with } e_{P,M}\colon P\ot (_HM)\xrightarrow{\rho^R_P\ot (_HM)} P\ot H \ot (_HM) \xrightarrow{P\ot \phi_M} P\ot M.
    \end{align}
    Moreover, if $N\in \Mm^H_H$ then $N\ot_H M \cong N\square^H M$.
\end{proposition}
\begin{remark}
    Since on $\Mm^H_H\times \prescript{H}{H}{\Mm}$ the $H$-linear tensor product and the $H$-colinear cotensor product coincide, it suffices to assume existence of equalizers and/or coequalizers in $\Mm$ to discuss both constructions.
\end{remark}


\begin{invisible}
\begin{proof}
    By (\ref{eqn: (co)invariants}) it holds that $_HM$ is the equalizer of $H$-bimodule morphisms, and since $({}_{H}{\Mm}{_H}, \ot, I)$ has equalizers by Proposition \ref{prop:bi(co)modmonoidalot}. Thus $(_HM, \alpha^L_{_HM}, \alpha^R_{_HM})$ is an $H$-bimodule for the unique actions such that $i_M$ becomes an $H$-bimodule morphism. Likewise, as coequalizer of $H$-bicomodule morphisms $(_HM, \rho^L_{_HM}, \rho^R_{_HM})$ is an $H$-bicomodule for the unique coactions rendering $p_M$ an $H$-bicomodule morphism. It remains to be verified that this constitutes a tetramodule structure on $_HM$. Notice that $\rho^L_{_HM}$ is an $H$-bimodule morphism, since
    \begin{align*}
        \rho^L_{_HM}\alpha^L_{_HM}(H\ot p_M) &= (H\ot p_M)\rho^L_M\alpha^L_M
        \\&= (H\ot p_M)(\mu\ot \alpha^L_M)(H\ot \sigma_{H,H}\ot M)(\Delta\ot \rho^L_M)
        \\&= (\mu\ot \alpha^L_{_HM})(H\ot\sigma_{H,H}\ot _HM)(\Delta\ot\rho^L_{_HM})(H\ot p_M)
        \\&= \alpha^L_{H\ot (_HM)}(H\ot \rho^L_{_HM})(H\ot p_M).
    \end{align*}
    The conclusion follows as $H\ot p_M$ is epimorphic. Likewise one proves that $\rho^L_{_HM}$ is a right $H$-module morphism, and $\rho^R_{_HM}$ is an $H$-bimodule morphism.
\end{proof}
\end{invisible}

This allows us to explicitly construct the $H$-linear tensor and $H$-colinear cotensor product of $H$-tetramodules in a general setting. Indeed, the constructions in Proposition \ref{prop: tensor and cotensor by invariants} induce monoidal structures on the category of $H$-tetramodules.

\begin{theorem}\label{theorem: monoidal structures tetramodules}
Let $(\mathcal M,\otimes,I,\sigma)$ be a braided monoidal category
with split idempotents, and let $H$ be an object in $\mathsf{Hopf}(\mathcal M)$.
For each $H$-tetramodule $N$, choose a splitting
$(K_N,i_N,p_N)$ of the idempotent of Proposition \ref{prop:splittingidempotent}, where
$K_N={}_HN$, and let $\phi_N$ and $\psi_N$ be the inverse
isomorphisms in $^{H}_{H}\Mm$ provided by Lemma \ref{lemma: left Hopf module isomorphisms}.

For any $H$-tetramodules $M$ and $N$, the relative tensor and
cotensor products can both be realized on $M\otimes K_N$, with
canonical maps
\begin{align}
q_{M,N}
&=
(\alpha_M^R\otimes K_N)(M\otimes\psi_N)
\colon M\otimes N\longrightarrow M\otimes K_N,
\\
e_{M,N}
&=
(M\otimes\varphi_N)(\rho_M^R\otimes K_N)
\colon M\otimes K_N\longrightarrow M\otimes N.
\end{align}
The two induced tetramodule structures on $M\otimes K_N$
coincide. With these realizations, the identity functor
\[
(\mathrm{id},\mathrm{id},\mathrm{id}_H)\colon
({}_H^H\mathcal M_H^H,\otimes_H,H)
\longrightarrow
({}_H^H\mathcal M_H^H,\square^H,H)
\]
is a strict monoidal isomorphism. Moreover, the following equality holds:
\begin{equation}\label{eqn: monoidal structure}
e_{M,N}q_{M,N}
=
(\alpha_M^R\otimes\alpha_N^L)
(M\otimes\sigma_{H,H}\otimes N)
(\rho_M^R\otimes\rho_N^L).
\end{equation}
\end{theorem}

\begin{proof}
We suppress the associativity and unit constraints of the ambient
category $\mathcal M$. By Lemma \ref{lemma: left Hopf module isomorphisms}, the morphisms
\[
\phi_N
=
\alpha_N^L(H\otimes i_N)
\colon H\otimes K_N\longrightarrow N,
\qquad
\psi_N
=
(H\otimes p_N)\rho_N^L
\colon N\longrightarrow H\otimes K_N
\]
are mutually inverse morphisms in $^{H}_{H}\Mm$.
Set $T_{M,N}:=M\otimes K_N$. By Proposition \ref{prop: tensor and cotensor by invariants}, $q_{M,N}$ is the coequalizer of
$\alpha_M^R\otimes N$ and $M\otimes\alpha_N^L$, whereas
$e_{M,N}$ is the equalizer of $\rho_M^R\otimes N$ and
$M\otimes\rho_N^L$.
We first explain why these diagrams are preserved by tensoring. Under the identification $N\cong H\otimes K_N$, they become,
respectively, the diagrams
\[\begin{tikzcd}
	M\otimes H\otimes H & M\ot H & M
	\arrow[shift left, from=1-1, to=1-2, "\alpha^{R}_{M}\ot H"]
	\arrow[shift right, from=1-1, to=1-2,"M\ot \mu"']
	\arrow[from=1-2, to=1-3,"\alpha^{R}_{M}"]
\end{tikzcd}\]
and
\[\begin{tikzcd}
	M & M\ot H & M\ot H\ot H
	\arrow[from=1-1, to=1-2,"\rho^{R}_{M}"]
	\arrow[shift left, from=1-2, to=1-3,"\rho^{R}_{M}\ot H"]
	\arrow[shift right, from=1-2, to=1-3, "M\ot\Delta_{H}"']
\end{tikzcd}\]
both tensored on the right with $K_N$. The first diagram is split by
\[
s=M\otimes\eta,
\qquad
t=M\otimes H\otimes\eta,
\]
since
\[
\alpha_M^Rs=\mathrm{id}_M,
\qquad
(M\otimes \mu)t=\mathrm{id}_{M\otimes H},
\qquad
(\alpha_M^R\otimes H)t=s\alpha_M^R.
\]
Dually, the second diagram is split by
\[
r=M\otimes\varepsilon,
\qquad
v=M\otimes H\otimes\varepsilon,
\]
since
\[
r\rho_M^R=\mathrm{id}_M,
\qquad
v(M\otimes\Delta)=\mathrm{id}_{M\otimes H},
\qquad
v(\rho_M^R\otimes H)=\rho_M^Rr.
\]
Tensoring these splittings with $K_N$ proves that the diagrams
defining $q_{M,N}$ and $e_{M,N}$ are split.
Since split equalizers and coequalizers are preserved by every
functor, these diagrams are preserved by tensoring on either
side with any object of $\mathcal M$. Consequently, the constructions in the proof of Proposition \ref{prop:monoidaltetramodules}
apply: the outer actions and diagonal coactions descend to
$M\otimes_HN$, while the diagonal actions and outer coactions
restrict to $M\square^HN$.
Their universal properties give the corresponding monoidal
structures with unit $H$ and the canonical constraints described
in Remark \ref{remark: constraints of tetramodules}.
The construction in the proof of Theorem \ref{thm:duoidaltetramodules} also applies,
since it only requires these relative products and their
preservation by tensoring. Let $\zeta$ denote the resulting interchange.
Using the canonical unit constraints, define the natural
morphism of tetramodules
\[\begin{tikzcd}
	{M\ot_H N} & {(M\square^H H)\ot_H (H\square^H N)} & {(M\ot_H H)\square^H (H\ot_H N)} & {M\square^H N.}
	\arrow["\cong", from=1-1, to=1-2]
	\arrow["{\zeta_{M,H,H,N}}", from=1-2, to=1-3]
	\arrow["\cong", from=1-3, to=1-4]
\end{tikzcd}\]
Write
\[
\Gamma_{M,N}
=
(\alpha_M^R\otimes\alpha_N^L)
(M\otimes\sigma_{H,H}\otimes N)
(\rho_M^R\otimes\rho_N^L).
\]
The defining formula for the interchange, together with the
descriptions of the unit constraints in Remark \ref{remark: constraints of tetramodules}, gives
\[
e_{M,N}\Xi_{M,N}q_{M,N}=\Gamma_{M,N}.
\]
We now prove that the underlying morphism of $\Xi_{M,N}$
is $\mathrm{id}_{T_{M,N}}$.
By Proposition \ref{prop:splittingidempotent}, we have
\[
\rho_N^Li_N=\eta\otimes i_N,
\qquad
p_Ni_N=\mathrm{id}_{K_N}.
\]
Hence $\psi_Ni_N=\eta\otimes K_N$, and therefore
\[
q_{M,N}(M\otimes i_N)
=
(\alpha_M^R\otimes K_N)
(M\otimes\eta\otimes K_N)
=
\mathrm{id}_{T_{M,N}}.
\]
Moreover, naturality of the braiding with respect to $\eta$
and unitality of $\alpha_M^R$ imply
\[
\begin{aligned}
\Gamma_{M,N}(M\otimes i_N)=
(M\otimes\alpha_N^L)(\rho_M^R\otimes i_N)=
(M\otimes\varphi_N)(\rho_M^R\otimes K_N)=e_{M,N}.
\end{aligned}
\]
It follows that
\[
\begin{aligned}
e_{M,N}\Xi_{M,N}=
e_{M,N}\Xi_{M,N}q_{M,N}(M\otimes i_N)=
\Gamma_{M,N}(M\otimes i_N)=e_{M,N}.
\end{aligned}
\]
Since $e_{M,N}$ is a monomorphism, we obtain $\Xi_{M,N}=\mathrm{id}_{T_{M,N}}$.
As $\Xi_{M,N}$ is a morphism of tetramodules, the two induced
tetramodule structures on $T_{M,N}$ coincide.
Naturality of $\Xi$ then shows that the two product bifunctors
also coincide on morphisms.
Equation \eqref{eqn: monoidal structure} follows immediately. Finally, the interchange and unit axioms of Definition \ref{defn: duoidal}
give the coherence identities for the above canonical comparison.
Writing $a^{\otimes_H}$ and $a^{\square^H}$ for the two
associators, these include
\[
\begin{aligned}
&a^{\square^H}_{M,N,P}
(\Xi_{M,N}\square^H\mathrm{id}_P)
\Xi_{M\otimes_HN,P}=
(\mathrm{id}_M\square^H\Xi_{N,P})
\Xi_{M,N\otimes_HP}
a^{\otimes_H}_{M,N,P}.
\end{aligned}
\]
Likewise, writing $\ell$ and $r$ for the corresponding unitors,
one has
\[
\ell^{\square^H}_M\Xi_{H,M}
=
\ell^{\otimes_H}_M,
\qquad
r^{\square^H}_M\Xi_{M,H}
=
r^{\otimes_H}_M.
\]
Since every $\Xi_{M,N}$ is the identity in the chosen
realizations, the associators and unitors of the two monoidal
structures coincide.
Thus the identity functor, with identity structure maps,
is a strict monoidal isomorphism.
\end{proof}

\begin{remark}
In the realizations of Theorem \ref{theorem: monoidal structures tetramodules}, write
\[
T_{M,N}:=M\otimes{}_HN=M\otimes_HN=M\square^HN.
\]
The common tetramodule structure on $T_{M,N}$ is characterized
by the identities
\[
\begin{aligned}
q_{M,N}(\alpha_M^L\otimes N)
&=
\alpha_{T_{M,N}}^L(H\otimes q_{M,N}),
\\
q_{M,N}(M\otimes\alpha_N^R)
&=
\alpha_{T_{M,N}}^R(q_{M,N}\otimes H),
\\
(\rho_M^L\otimes N)e_{M,N}
&=
(H\otimes e_{M,N})\rho_{T_{M,N}}^L,
\\
(M\otimes\rho_N^R)e_{M,N}
&=
(e_{M,N}\otimes H)\rho_{T_{M,N}}^R.
\end{aligned}
\]
Thus the actions are induced through the coequalizer $q_{M,N}$,
while the coactions are induced through the equalizer $e_{M,N}$.
Their existence and uniqueness follow from the constructions
and universal properties established in Theorem \ref{theorem: monoidal structures tetramodules}.
\end{remark}

\begin{corollary}\label{cor: monoidal structures of tetramodules}
Let $(\mathcal M,\otimes,I,\sigma)$ be a braided monoidal
category with split idempotents, and let $H$ be an object in $\mathsf{Hopf}(\mathcal M)$ whose antipode is an isomorphism.
Then the canonical duoidal structure $({}_H^H\mathcal M_H^H,\otimes_H,H,\square^H,H)$ is strong.

More precisely, in the common realizations of Theorem \ref{theorem: monoidal structures tetramodules},
the two monoidal products coincide and the interchange is,
with the canonical associativity constraints understood,
\[
\zeta_{A,B,C,D}
=
\mathrm{id}_A\otimes_H c_{B,C}\otimes_H\mathrm{id}_D,
\]
where $c$ is the canonical braiding on
$({}_H^H\mathcal M_H^H,\otimes_H,H)$.
Thus this duoidal structure is the one induced by that braiding.
\end{corollary}

\begin{proof}
By Theorem \ref{theorem: monoidal structures tetramodules}, the equalizer and coequalizer diagrams defining
the relative products are split.
Since split equalizers and coequalizers are preserved by every
functor, these diagrams are preserved by tensoring on either
side with any object of $\mathcal M$.
Hence the construction of Theorem \ref{thm:duoidaltetramodules} applies under the present
assumptions. Let $c$ denote the canonical braiding on
$({}_H^H\mathcal M_H^H,\otimes_H,H)$, which exists since the
antipode of $H$ is an isomorphism \cite[Theorem 4.3.1]{Drabant}.
Use the common realizations of Theorem \ref{theorem: monoidal structures tetramodules}, so that the two
monoidal products and their constraints coincide.
Substituting these realizations into the defining formulas
for the canonical interchange and the braiding gives
\[
\zeta_{A,B,C,D}
=
\mathrm{id}_A\otimes_H c_{B,C}\otimes_H\mathrm{id}_D,
\]
where associativity constraints are suppressed.
Since $c_{B,C}$ is invertible, every component of $\zeta$
is invertible. The unit structure maps are the canonical isomorphisms
\[
H\longrightarrow H\square^HH,
\qquad
H\otimes_HH\longrightarrow H,
\qquad
\mathrm{id}_H\colon H\longrightarrow H.
\]
Therefore the duoidal structure is strong, and its interchange
is precisely the one induced by the canonical braiding.
\end{proof}

\begin{example}
Let $H$ be a $\Bbbk$-Hopf algebra, and consider the symmetric
monoidal category $(\Vec_{\Bbbk}, \ot_{\Bbbk}, \Bbbk, \tau)$.
For an $H$-tetramodule $N$, we may choose
\[
{}_H N:={}^{\operatorname{co}H}N
=\{n\in N\mid \rho_N^L(n)=1_H\otimes n\}.
\]
The splitting of the idempotent of Proposition \ref{prop:splittingidempotent} is then given by
the inclusion $i_N:{}_H N\hookrightarrow N$ and the projection
\[
p_N:N\longrightarrow {}_H N,\ n\longmapsto S(n_{(-1)})\triangleright n_{(0)}.
\]
Writing $H^+=\ker\varepsilon$, the canonical quotient map induces
an isomorphism
\[
{}^{\operatorname{co}H}N
\xrightarrow{\ \cong\ }
N/H^+N
\cong \Bbbk\otimes_H N,
\ n\longmapsto[n],
\]
whose inverse is $[n]\mapsto p_N(n)$.
Here $\Bbbk$ is regarded as a right $H$-module via $\varepsilon$. By Proposition \ref{prop: tensor and cotensor by invariants} and Theorem \ref{theorem: monoidal structures tetramodules}, for any $H$-tetramodules
$M$ and $N$, both relative products can be realized on the vector
space $M\otimes({}_H N)$, with the induced tetramodule structure:
\[
M\otimes_H N
\cong M\otimes({}_H N)
\cong M\square^H N.
\]
Consequently, the canonical duoidal category $({}_H^H\mm_H^H,\otimes_H,H,\square^H,H)$
can be represented using the same bifunctor
$M\odot N:=M\otimes({}_H N)$ for both monoidal products. If, in addition, the antipode of $H$ is bijective, the canonical
interchange corresponds, under these identifications, to the
interchange induced by the canonical braiding on
$({}_H^H\mm_H^H,\otimes_H,H)$, as made explicit through
the equivalence with Yetter--Drinfeld modules in
Proposition \ref{prop:duoidalequivalenceYetterDrinfeld}.
\end{example}

We end this section by considering the category of Yetter--Drinfeld modules over $H$. We will work over the symmetric monoidal category $(\Vec_{\Bbbk}, \ot_{\Bbbk}, \Bbbk, \tau)$. \medskip

\noindent\textbf{Transporting duoidal structures to Yetter--Drinfeld modules.} Let $H$ be a Hopf algebra with bijective antipode. We recall the following definition:

\begin{definition}
    A (right-right) \textit{Yetter--Drinfeld module} over $H$ is a tuple $(V,\rho,\triangleleft)$, where $(V,\rho)$ is a right $H$-comodule, $(V,\triangleleft)$ is a right $H$-module and the following compatibility condition holds:
\begin{equation}\label{eq:compcondYDmod}
\rho(m\triangleleft b)=
(m_{0}\triangleleft b_{2})\ot S(b_{1})m_{1}b_{3}, \qquad m\in V, b\in H.
\end{equation}
A morphism of Yetter--Drinfeld modules is a right $H$-linear and right $H$-colinear map. We denote the category of Yetter--Drinfeld modules over $H$ and their morphisms by $\mathcal{YD}^{H}_{H}$. 
\end{definition}

The category $(\mathcal{YD}^{H}_{H},\ot,\Bbbk)$ is monoidal as follows; given $M,N$ in $\mathcal{YD}^{H}_{H}$, the object $M\ot N$ is in $\mathcal{YD}^{H}_{H}$ with the diagonal action and the diagonal coaction and $\Bbbk$ is in $\mathcal{YD}^{H}_{H}$ via the unit and the counit of $H$. The constraints are those of vector spaces regarded as morphisms of Yetter--Drinfeld modules.
Moreover, the monoidal category $(\mathcal{YD}^{H}_{H},\ot,\Bbbk)$ is braided with braiding defined, for all $M,N$ in $\mathcal{YD}^{H}_{H}$, in the following way:
\begin{equation}\label{def:bradingYD}
\sigma^{\mathcal{YD}}_{M,N}:M\ot N\to N\ot M,\ m\ot n\mapsto n_{0}\ot(m\triangleleft n_{1})
\end{equation}
and its inverse $(\sigma^{\mathcal{YD}}_{M,N})^{-1}$ is defined by $(\sigma^{\mathcal{YD}}_{M,N})^{-1}(n\ot m)=(m\triangleleft S^{-1}(n_{1}))\ot n_{0}$.

\begin{theorem}[{\cite[Theorem 5.7]{Schauenburg}}]\label{thm:equivalencetetraYD}
There is an equivalence of categories ${}^H_H\mm^H_H\overset{F}{\underset{G}{\rightleftharpoons}} \Yd^H_H$. Explicitly, this is constructed as follows.
\begin{itemize}
    \item $F(M)= {}_HM = {}^{\mathrm{co}H}M$, which is in $\mathcal{YD}^{H}_{H}$ with restricted right $H$-coaction and right $H$-action $m\triangleleft h:=S(h_{1})\cdot m\cdot h_{2}$, for $m\in{}^{\mathrm{co}H}M$ and $h\in H$. Moreover, for any tetramodule morphism $f\colon M\to X$, $F(f) = p_Xfi_M = p_Xf\vert_{{}_HM}$.
    \item $G(N)=H\ot N$, where $\rho^{L}_{G(N)}=\Delta\ot\id$, $\rho^{R}_{G(N)}(h\ot n)=h_{1}\ot n_{(0)}\ot h_{2}n_{(1)}$ (diagonal coaction), $h'\triangleright(h\ot n)=h'h\ot n$, and $(h\ot n)\triangleleft h'=hh'_{1}\ot(n\triangleleft h'_{2})$ (diagonal action). Moreover, for any Yetter--Drinfeld module morphism $g$, $G(g) = \id_H\ot g$.
\end{itemize}
This equivalence is moreover monoidal, for the monoidal structures $(\Yd^H_H, \ot, \Bbbk)$ and $({}^H_H\mm^H_H, \ot_H, H) \cong ({}^H_H\mm^H_H, \square^H, H)$ as described before. 
\end{theorem}
\begin{remark}\label{rmk:braidedequivalence}
    If $H$ is a Hopf algebra with bijective antipode, using the explicit monoidal structure, one can transport the braided structure of $(\mathcal{YD}^{H}_{H},\ot,\Bbbk,\sigma^{\mathcal{YD}})$ to obtain a braided monoidal equivalence with $({}^{H}_{H}{\mm}^{H}_{H},\ot_{H},H,\sigma^{\ot_{H}}) 
    $, where $\sigma^{\ot_H}$ is defined, for all $M,N$ in ${}^{H}_{H}{\mm}^{H}_{H}$ by
\begin{equation}\label{Worbraiding}
\sigma^{\ot_H}_{M,N}:M\ot_{H}N\to N\ot_{H}M,\ m\ot_{H}n\mapsto (m_{-2}\tr n_{0}\tl S(n_{1}))\ot_{H}(S(m_{-1})\tr m_{0}\tl n_{2}).
\end{equation}
Similarly, one can obtain a braiding $\sigma^{\square^H}$ on $(^{H}_{H}\mm^{H}_{H},\square^{H},H)$ such that this becomes braided equivalent to $(\mathcal{YD}^{H}_{H},\ot,\Bbbk,\sigma^{\mathcal{YD}})$.
The braiding for $({}^{H}_{H}{\mm}^H_H, \ot_H, H)$ was already discussed in \cite[Theorem 6.3]{Schauenburg}, where it was noticed that this braiding is the unique one such that $\sigma^{\ot_H}_{M,N}(m\ot n) = n\ot m$ for all $m\in {}_{H}{M}$ and $n\in N_H$.
\end{remark}

In general, the braided monoidal category of Yetter--Drinfeld modules may admit
non-trivial infinitesimal braidings.

\begin{example}\label{ex:infbraidYD}
For instance, assume that $\operatorname{char}(\Bbbk)=0$
and let $H:=\Bbbk\mathbb Z=\Bbbk[g,g^{-1}]$.
A right $H$-comodule is equivalently a $\mathbb Z$-graded vector space $M=\bigoplus_{p\in\mathbb Z}M_p$, where $\rho(m)=m\otimes g^p$ for $m\in M_p$. Since $\mathbb Z$ is abelian, the
Yetter--Drinfeld compatibility condition implies that the $H$-action preserves
the homogeneous components, namely $M_p\triangleleft g^r\subseteq M_p$ for all $p,r\in\mathbb Z$. For $M,N\in\mathcal{YD}^H_H$, define a natural transformation $t_{M,N}\colon M\otimes N\longrightarrow M\otimes N$ on homogeneous elements by
\[
t_{M,N}(m\otimes n):=pq\,m\otimes n,
\qquad m\in M_p,\quad n\in N_q.
\]
Since morphisms of Yetter--Drinfeld modules preserve the grading, the family
$t=(t_{M,N})_{M,N}$ is natural, and each $t_{M,N}$ is both $H$-linear and
$H$-colinear. Moreover, for homogeneous elements
$m\in M_p$, $n\in N_q$, and $u\in P_r$, the two infinitesimal braid identities
reduce, respectively, to
\[
p(q+r)=pq+pr
\qquad\text{and}\qquad
(p+q)r=pr+qr.
\]
Indeed, the Yetter--Drinfeld braiding and its inverse are given by
\[
\sigma^{\mathcal{YD}}_{M,N}(m\otimes n)
 =n\otimes(m\triangleleft g^q),
\qquad
(\sigma^{\mathcal{YD}}_{M,N})^{-1}(n\otimes m)
 =(m\triangleleft g^{-q})\otimes n.
\]
Therefore, $t$ is an infinitesimal braiding on
$(\mathcal{YD}^H_H,\otimes,\Bbbk,\sigma^{\mathcal{YD}})$, and it is non-zero; for
example, on one-dimensional Yetter--Drinfeld modules concentrated in degree
$1$ it is the identity.
Consequently, non-trivial infinitesimal braidings on Yetter--Drinfeld modules
can be transported, through the braided monoidal equivalence between
$\mathcal{YD}^H_H$ and ${}_H^H\mathcal{M}_H^H$, to non-trivial infinitesimal
braidings on the corresponding braided monoidal category of tetramodules.
\end{example}

We recall from \cite[Definition 6.54]{Aguiar} that, given two duoidal categories $(\Cc,\circ,I,\bullet,J),(\Cc',\circ',I',\bullet',J')$, a \textit{double lax monoidal functor} $(F,\phi^{\circ}_{X,Y},\phi^{\circ}_{0},\phi^{\bullet}_{X,Y},\phi^{\bullet}_{0}):(\Cc,\circ,I,\bullet,J)\to(\Cc',\circ',I',\bullet',J')$ is the datum of:
\begin{itemize}
    \item[1)] $F:\Cc\to\Cc'$ is a functor;
    \item[2)] $(F,\phi^{\circ}_{X,Y},\phi^{\circ}_{0}):(\Cc,\circ,I)\to(\Cc',\circ',I')$ is lax monoidal;
    \item[3)] $(F,\phi^{\bullet}_{X,Y},\phi^{\bullet}_{0}):(\Cc,\bullet,J)\to(\Cc',\bullet',J')$ is lax monoidal.
\end{itemize}
    Moreover, the following diagrams are required to be commutative (where we simply write $\circ$ and $\bullet$ instead of $\circ'$ and $\bullet'$):
\[
\begin{tikzcd}[column sep=huge, row sep=huge]
(F(A) \bullet F(B)) \circ (F(C) \bullet F(D)) \arrow[r, "\zeta^{\mathcal{\Cc'}}"] \arrow[d, "\phi^{\bullet}_{A,B} \circ \phi^{\bullet}_{C,D}"'] & (F(A) \circ F(C)) \bullet (F(B) \circ F(D)) \arrow[d, "\phi^{\circ}_{A,C} \bullet \phi^{\circ}_{B,D}"] \\
F(A \bullet B) \circ F(C \bullet D) \arrow[d, "\phi^{\circ}_{A \bullet B, C \bullet D}"'] & F(A \circ C) \bullet F(B \circ D) \arrow[d, "\phi^{\bullet}_{A \circ C, B \circ D}"] \\
F((A \bullet B) \circ (C \bullet D)) \arrow[r, "F(\zeta^{\mathcal{C}})"'] & F((A \circ C) \bullet (B \circ D)),
\end{tikzcd}
\]
\[\begin{tikzcd}
	I' & F(I) & F(I\bullet I) & J' & F(J) & F(J\circ J) \\
	I'\bullet I' && F(I)\bullet F(I), & J'\circ J' && F(J)\circ F(J),
	\arrow[from=1-1, to=1-2,"\phi^{\circ}_{0}"]
	\arrow[from=1-1, to=2-1,"\Delta^{\bullet}_{I}"']
	\arrow[from=1-2, to=1-3,"F(\Delta^{\bullet}_{I})"]
	\arrow[from=1-4, to=1-5,"\phi^{\bullet}_{0}"]
	\arrow[from=1-6, to=1-5,"F(m^{\circ}_{J})"']
	\arrow[from=2-1, to=2-3,"\phi^{\circ}_{0}\bullet\phi^{\circ}_{0}"']
	\arrow[from=2-3, to=1-3,"\phi^{\bullet}_{I,I}"']
	\arrow[from=2-4, to=1-4,"m^{\circ}_{J}"]
	\arrow[from=2-4, to=2-6,"\phi^{\bullet}_{0}\circ\phi^{\bullet}_{0}"']
	\arrow[from=2-6, to=1-6, "\phi^{\circ}_{J,J}"']
\end{tikzcd}\]
and $\phi^{\bullet}_{0}\varepsilon^{\bullet}_{I}=F(\varepsilon^{\bullet}_{I})\phi^{\circ}_{0}$. A duoidal lax monoidal functor is called \textit{strong} if $\phi^{\circ}_{X,Y},\phi^{\circ}_{0},\phi^{\bullet}_{X,Y},\phi^{\bullet}_{0}$ are isomorphisms in $\Cc'$. This is a \textit{duoidal equivalence} if it is a monoidal equivalence for both the underlying categories.

Dually, one can define \textit{(strong) double colax monoidal functors}, see \cite[Definition 6.55]{Aguiar}.\medskip

One can easily prove the following result. 
\begin{lemma}\label{lem:transportequiv}
    Let $(\Cc,\circ,I,\bullet,J)$ be a duoidal category. Let $(\Dd,\circ',I')$ be a monoidal category for which there exists a monoidal equivalence by a lax monoidal functor $(F,\phi^{\circ},\phi^{\circ}_{0}):(\Cc,\circ,I)\to(\Dd,\circ',I')$ with inverse colax monoidal functor $(G, \psi^{\circ'}, \psi^{\circ'}_0)$, then $(\Cc,\circ,I,\bullet,J)$ is duoidally equivalent to $(\Dd,\circ',I',\bullet',J')$ where $X\bullet' Y:=F(G(X)\bullet G(Y))$, $J'=F(J)$.

    Dually, let $(\Dd, \bullet', J')$ be a monoidal category for which there exists a monoidal equivalence by a colax monoidal functor $(F,\phi^{\bullet},\phi^{\bullet}_{0}):(\Cc,\bullet,J)\to(\Dd,\bullet',J')$ with inverse lax monoidal functor $(G,\psi^{\bullet'}, \psi^{\bullet'}_0)$, then one obtains that $(\Cc,\circ,I,\bullet,J)$ is duoidally equivalent to $(\Dd,\circ',I',\bullet',J')$ where $X\circ' Y:=F(G(X)\circ G(Y))$ and $I'=F(I)$.
\end{lemma}
\begin{invisible}
\begin{proof}
    Since monoidal equivalences are by strong monoidal functors, we have 
\begin{align*}
(A\bullet'B)\circ'(C\bullet'D)&=F(G(A)\bullet G(B))\circ' F(G(C)\bullet G(D))\cong F((G(A)\bullet G(B))\circ (G(C)\bullet G(D))),\\
(A\circ' C)\bullet'(B\circ'D)&=F(G(A\circ' C)\bullet G(B\circ'D))\cong F((G(A)\circ G(C))\bullet(G(B)\circ G(D))).
\end{align*}
Hence one defines 
\[
\zeta^{\Dd}_{A,B,C,D}:=F(\psi^{\circ}_{A,C}\bullet\psi^{\circ}_{B,D})F(\zeta^{\Cc}_{G(A),G(B),G(C),G(D)})\phi^{\circ}_{F(G(A)\bullet G(B)),F(G(C)\bullet G(D))}.
\]
Moreover, $I'\bullet'I'=F(G(I')\bullet G(I'))\cong F(I\bullet I)$, $I'\cong F(I)$, $J'\circ'J'=F(J)\circ'F(J)\cong F(J\circ J)$, and $J'=F(J)$, so one defines 
\[
\Delta^{\bullet'}_{I'}:=F(\phi^{\circ}_{0}\bullet\phi^{\circ}_{0})F(\Delta^{\bullet}_{I})\phi^{\circ}_{0},\qquad m^{\circ'}_{J'}:=F(m^{\circ}_{J})\phi^{\circ}_{J,J},\qquad \varepsilon^{\bullet'}_{I'}:=F(\varepsilon^{\bullet}_{I})\phi^{\circ}_{0}.
\]
One can then check that $(\Dd,\circ',I',\bullet',J')$ becomes a duoidal category such that $(F,\phi^{\circ}_{X,Y},\phi^{\circ}_{0},\phi^{\bullet}_{X,Y},\phi^{\bullet}_{0}):(\Cc,\circ,I,\bullet,J)\to(\Dd,\circ',I',\bullet',J')$ becomes a duoidal equivalence, where $\phi^{\bullet}_{X,Y}:=F(\epsilon_X\bullet\epsilon_Y)$ and $\phi^{\bullet}_{0}:=\id_{F(J)}$ and $\epsilon_X:GF(X)\to X$, $\epsilon_Y:GF(Y)\to Y$ are the isomorphisms in $\Cc$. 

Likewise, in the dual setting
\begin{align*}
    &(A\bullet' B)\circ' (C\bullet' D) = F(G(A\bullet 'B)\circ G(C\bullet' D)) \cong F((G(A)\bullet G(B))\circ (G(C)\bullet G(D))),
    \\&(A\circ' C)\bullet' (B\circ' D) = F(G(A)\circ G(C))\bullet' F(G(B)\circ G(D)) \cong F((G(A)\circ G(C))\bullet (G(B)\circ G(D))).
\end{align*}
Hence one defines
\[
\zeta^D_{A,B,C,D} = (\phi^{\bullet}_{G(A)\circ G(C), G(B)\circ G(D)})^{-1}F(\zeta^{\Cc}_{G(A), G(B), G(C), G(D)})F((\psi^{\bullet}_{A,B})^{-1}\circ (\psi^{\bullet}_{C,D})^{-1}).
\]
\end{proof}
\end{invisible}

By Theorem \ref{thm:equivalencetetraYD} and Lemma \ref{lem:transportequiv}, we obtain the following result.

\begin{proposition}\label{prop:duoidalequivalenceYetterDrinfeld}
    Let $H$ be a Hopf algebra with bijective antipode, then $(^{H}_{H}\mm^{H}_{H},\ot_{H},H,\square^{H},H)$ is duoidally equivalent to $(\mathcal{YD}^{H}_{H},\ot,\Bbbk,\ot,\Bbbk)$.
\end{proposition}

\begin{proof}
    We consider the monoidal equivalence $F:(^{H}_{H}\mm^{H}_{H},\ot_{H},H)\to(\mathcal{YD}^{H}_{H},\ot,\Bbbk)$ given in 1) of Theorem \ref{thm:equivalencetetraYD}, with inverse $G:(\mathcal{YD}^{H}_{H},\ot,\Bbbk)\to(^{H}_{H}\mm^{H}_{H},\ot_{H},H)$. By Lemma \ref{lem:transportequiv}, we know that $(^{H}_{H}\mm^{H}_{H},\ot_{H},H,\square^{H},H)$ is duoidally equivalent to $(\mathcal{YD}^{H}_{H},\ot,\Bbbk,\bullet,J)$, where $\bullet$ and $J$ are defined as follows. Given $A$ and $B$ in $\mathcal{YD}^{H}_{H}$, recalling that $G$ is also a monoidal equivalence, namely $G:(\mathcal{YD}^{H}_{H},\ot,\Bbbk)\to(^{H}_{H}\mm^{H}_{H},\square^{H},H)$, we get
\[
A\bullet B=F(G(A)\square^{H}G(B))\cong FG(A\ot B)\cong A\ot B,\qquad J=F(H)={}^{\mathrm{co}H}H\cong\Bbbk.
\]
The transported interchange can be identified explicitly. Recall that $G(A) = H\ot A$ for $A$ in $\mathcal{YD}^{H}_{H}$. Given $A,B$ in $\mathcal{YD}^H_H$, the monoidal structures are given by \cite[Theorem 5.7]{Schauenburg}
\[
G(A)\ot_{H}G(B)\to G(A\ot B),\ (h\ot a)\ot_{H}(k\ot b)\mapsto hk_{1}\ot(a\triangleleft k_{2})\ot b,
\]
and 
\[
G(A\ot B)\to G(A)\square^{H}G(B), \ h\ot a\ot b\mapsto(h_{1}\ot a_{(0)})\square^H(h_{2}a_{(1)}\ot b).
\]
The inverse of the latter map is the restriction of $h\ot a\ot k\ot b\mapsto h\varepsilon(k)\ot a\ot b$. Using these identifications and then taking left coinvariants, the interchange of tetramodules becomes
\[
(a\ot b)\ot(c\ot d)\longmapsto(a\ot c_{(0)})\ot((b\triangleleft c_{(1)})\ot d).
\]
Thus, the transported interchange is $\mathrm{id}_{A}\ot\sigma^{\mathcal{YD}}_{B,C}\ot\mathrm{id}_{D}$, with $\sigma^{\mathcal{YD}}$ as in \eqref{def:bradingYD}. The unit maps
become the canonical identifications involving $\Bbbk$. Hence $(^{H}_{H}\mm^{H}_{H},\ot_{H},H,\square^{H},H)$ is duoidally equivalent to $(\mathcal{YD}^{H}_{H},\ot,\Bbbk,\ot,\Bbbk)$. 
\end{proof}

\begin{remark}
We observe that the category $\mathcal{YD}^{H}_{H}$ can be equipped with other duoidal structures. We give two examples.
\begin{itemize}
    \item[1)] If $H$ is a finite dimensional Hopf algebra, the category $(\mathcal{YD}^{H}_{H},\ot,\Bbbk)$ is monoidally equivalent to $(\mm_{D(H)},\ot,\Bbbk)$, where $D(H)$ is the so-called \textit{Drinfeld double} of $H$ introduced in \cite{Drinfeld}. We know that braidings on the latter monoidal category correspond bijectively to quasitriangular structures on $D(H)$. Any braiding on $(\mm_{D(H)},\ot,\Bbbk)$ can be lifted to $(\mathcal{YD}^{H}_{H},\ot,\Bbbk)$. Choosing a braiding different from the canonical one we get a duoidal structure on $(\mathcal{YD}^{H}_{H},\ot,\Bbbk,\ot,\Bbbk)$ which is different from the one obtained above. In fact, the interchange law obtained from the braiding changes.
    \item[2)] It is known that the forgetful functor $\mathcal{YD}^{H}_{H}\to\mm^{H}$ creates small limits and colimits, see e.g.\ \cite[Lemma 5.3]{AgGoVe}. Therefore, the category $\mathcal{YD}^{H}_{H}$ has binary products $\times$ and terminal object $\mathds{1}$, which are constructed as in $\mm^{H}$. More precisely, one has $\times=\oplus$ and $\mathds{1}=\{0\}$. Therefore, we have a duoidal category $(\mathcal{YD}^{H}_{H},\ot,\Bbbk,\oplus,\{0\})$, see \cite[Example 6.19 and Remark 6.21]{Aguiar}.
\end{itemize}

\end{remark}

In the next section, we introduce pre-Cartier duoidal categories and we study this notion for bimodule and bicomodule categories.

\section{Pre-Cartier duoidal categories and applications to bi(co)modules}\label{sec:preCartier}

In this section, we introduce the notion of a \textit{pre-Cartier duoidal category}. In order to do this, we first recall the notion of a \textit{braided duoidal category}. For notational ease, we will often write actions by concatenation.

\begin{definition}[{\cite[Definition 6.5]{Aguiar}}]
A duoidal category $(\Cc,\circ,I,\bullet,J)$ is $\circ$-braided if the monoidal category $(\Cc,\circ,I)$ has a braiding $\sigma^{\circ}$ and the diagrams in \eqref{diagramscircbraided} commute.
\begin{equation}\label{diagramscircbraided}
\begin{tikzcd}
	(A\bullet B)\circ(C\bullet D) & (A\circ C)\bullet(B\circ D) && J\circ J \\
	(C\bullet D)\circ(A\bullet B) & (C\circ A)\bullet(D\circ B) & J & J\circ J
	\arrow[from=1-1, to=1-2,"\zeta_{A,B,C,D}"]
	\arrow[from=1-1, to=2-1,"\sigma^{\circ}_{A\bullet B, C\bullet D}"']
	\arrow[from=1-2, to=2-2,"\sigma^{\circ}_{A,C}\bullet\sigma^{\circ}_{B,D}"]
	\arrow[from=1-4, to=2-3,"m^{\circ}_{J}"']
	\arrow[from=1-4, to=2-4,"\sigma^{\circ}_{J,J}"]
	\arrow[from=2-1, to=2-2,"\zeta_{C,D,A,B}"']
	\arrow[from=2-4, to=2-3,"m^{\circ}_{J}"]
\end{tikzcd}
\end{equation}
Similarly, a duoidal category $(\Cc,\circ,I,\bullet,J)$ is $\bullet$-braided if the monoidal category $(\Cc,\bullet,J)$ has braiding $\sigma^{\bullet}$ and the diagrams in \eqref{diagramsbulletbraided} commute.
\begin{equation}\label{diagramsbulletbraided}
\begin{tikzcd}
	(A\bullet B)\circ(C\bullet D) & (A\circ C)\bullet(B\circ D) && I\bullet I \\
	(B\bullet A)\circ(D\bullet C) & (B\circ D)\bullet(A\circ C) & I & I\bullet I
	\arrow[from=1-1, to=1-2,"\zeta_{A,B,C,D}"]
	\arrow[from=1-1, to=2-1,"\sigma^{\bullet}_{A,B}\circ\sigma^{\bullet}_{C,D}"']
	\arrow[from=1-2, to=2-2,"\sigma^{\bullet}_{A\circ C, B\circ D}"]
	\arrow[from=2-3, to=1-4,"\Delta^{\bullet}_{I}"]
	\arrow[from=2-4, to=1-4,"\sigma^{\bullet}_{I,I}"']
	\arrow[from=2-1, to=2-2,"\zeta_{B,A,D,C}"']
	\arrow[from=2-3, to=2-4,"\Delta^{\bullet}_{I}"']
\end{tikzcd}
\end{equation}
A duoidal category $(\Cc,\circ,I,\bullet,J)$ is \textit{braided} if it is $\circ$-braided and $\bullet$-braided. We denote such a duoidal category by $(\Cc,\circ,I,\sigma^{\circ},\bullet,J,\sigma^{\bullet})$.
\end{definition}

First, we have the following result.

\begin{proposition}\label{prop:braidedduoidalbimod}
The following statements hold:
\begin{itemize}
    \item[1)] Let $(H,\Rr)$ be a quasitriangular bialgebra with $\Rr^{-1} = \overline{\Rr}^i\ot \overline{\Rr}_i$. 
Then the duoidal category $(_{H}\mm_{H},\ot_{H},H,\ot,\Bbbk)$ is $\bullet$-braided, 
where $\bullet=\ot$ and 
\begin{equation}\label{eq:braidingbimod}
\sigma^{\bullet}_{M,N}(m\ot n)=\overline{\Rr}^{i}n\Rr^{j}\ot\overline{\Rr}_{i}m\Rr_{j}.
\end{equation}
\item[2)] If $(H, \Rr)$ is a coquasitriangular bialgebra 
then $({}^{H}{\mm}^H, \ot, \Bbbk,\square^H, H)$ is a $\circ$-braided duoidal category, where 
where $\circ=\ot$ and
    \begin{align}
&\sigma_{M,N}^{\circ}(m\ot n) = \Rr^{-1}(m_{(-1)}\ot n_{(-1)})n_{(0)}\ot m_{(0)}\Rr(m_{(1)}\ot n_{(1)}).\label{braidingbicom}
    \end{align}
\end{itemize}
\end{proposition}

\begin{proof}
1). In \cite[Proposition 4.14]{ABCS} it is proven that, given a quasitriangular Hopf algebra $(H,\Rr)$, the duoidal category $(_{H}\mm_{H},\ot_{H},H,\ot,\Bbbk)$ is $\ot$-braided with braiding defined as in \eqref{eq:braidingbimod}. We simply point out that the antipode is not needed. 

2). 
The category $({}^{H}{\mm}^H, \ot,\Bbbk, \sigma^{\circ})$ is a braided monoidal category as observed in Remark \ref{rmk:braidingbimodquasi}. Dual to the computations in \cite[Proposition 4.14]{ABCS}, one verifies immediately that (\ref{diagramscircbraided}) holds.
\end{proof}

\begin{lemma}\label{lem:braidedduoidal}
Let $(\mathcal{C},\ot,I,\sigma)$ be a braided monoidal category and equip $(\mathcal{C},\ot,I,\ot,I)$ with interchange $\zeta_{A,B,C,D}=\mathrm{id}_{A}\ot\sigma_{B,C}\ot\mathrm{id}_{D}$. With $\sigma$ as the braiding for both products, this duoidal category is $\circ$-braided, or $\bullet$-braided, if and only if $\sigma$ is a symmetry. In that case it is braided duoidal.
\end{lemma}

\begin{proof}
In the $\circ$-braiding compatibility, put $B=C=I$. Suppressing unit constraints gives $\sigma_{D,A}\sigma_{A,D}=\mathrm{id}_{A\ot D}$, for all $A,D\in\Cc$. The same condition follows from the $\bullet$-braiding compatibility by putting $A=D=I$. Conversely, symmetry makes both compatibility diagrams commute by symmetric monoidal coherence; the unit conditions follow as well. This is also the criterion given in \cite[Proposition 6.13]{Aguiar}.
\end{proof}

As a consequence, we have the following result.

\begin{proposition}\label{prop:braidedduoidalequivalencetetraYetter}
Let $H$ be a Hopf algebra with bijective antipode. With the two braidings transported from $\sigma^{\mathcal{YD}}$, the duoidal category $(^{H}_{H}\mm^{H}_{H},\ot_{H},H,\square^{H},H)$ is braided if and only if $\sigma^{\mathcal{YD}}$ is a symmetry. This forces $H\cong\Bbbk$.
\end{proposition}

\begin{proof}
By Proposition \ref{prop:duoidalequivalenceYetterDrinfeld}, $(^{H}_{H}\mm^{H}_{H},\ot_{H},H,\square^{H},H)$ is duoidally equivalent to $(\mathcal{YD}^{H}_{H},\ot,\Bbbk,\ot,\Bbbk)$. Therefore, by Lemma \ref{lem:braidedduoidal}, $(^{H}_{H}\mm^{H}_{H},\ot_{H},H,\square^{H},H)$ is braided if and only if $\sigma^{\mathcal{YD}}$ is a symmetry. This forces $H\cong\Bbbk$ as proven in \cite{PareigisbrYD}.
\end{proof}

The previous result can be extended to an arbitrary bialgebra $H$.

\begin{theorem}\label{prop:tetramodule-braided-duoidal-obstruction}
Let $H$ be a bialgebra and consider the duoidal category $({}_H^H\mm_H^H,\otimes_H,H,\square^H,H)$. Then the following conditions are equivalent:
\begin{enumerate}
    \item $H\cong \Bbbk$;
    \item this duoidal category is $\circ$-braided,
          where $\circ=\otimes_H$;
    \item this duoidal category is $\bullet$-braided,
          where $\bullet=\square^H$;
    \item this duoidal category is braided.
\end{enumerate}
\end{theorem}

\begin{proof}
For $H=\Bbbk$, both monoidal products are the ordinary tensor product
of vector spaces, and the usual flip gives a braided duoidal
structure.\\
Conversely, suppose that $\sigma^{\bullet}$ is a braiding on
$({}_H^H\mm_H^H,\square^H,H)$ compatible with the canonical
interchange. Suppressing the canonical unit constraints, define
\begin{align*}
\alpha_{M,N}
&=\zeta_{M,H,H,N}:
M\otimes_H N\longrightarrow M\square^H N,\\
\beta_{M,N}
&=\zeta_{H,M,N,H}:
M\otimes_H N\longrightarrow N\square^H M.
\end{align*}
Their explicit formulas are
\begin{align}
\alpha_{M,N}(m\otimes_H n)
&=m_0n_{-1}\square^H m_1n_0,
\label{eq:tetramodule-alpha}\\
\beta_{M,N}(m\otimes_H n)
&=m_{-1}n_0\square^H m_0n_1.
\label{eq:tetramodule-beta}
\end{align}
Indeed, the inverses of the cotensor unit constraints are the
coactions, whereas the tensor unit constraints are the actions. The $\bullet$-braiding compatibility reads
\[
\sigma^{\bullet}_{A\otimes_H C,B\otimes_H D}\zeta_{A,B,C,D}
=
\zeta_{B,A,D,C}(\sigma^{\bullet}_{A,B}\otimes_H \sigma^{\bullet}_{C,D}).
\]
Taking $(A,B,C,D)=(M,H,H,N)$ and $(H,N,M,H)$, respectively,
and using the unit identities for a braiding, we obtain
\begin{equation}\label{eq:tetramodule-comparison-identities}
\sigma^{\bullet}_{M,N}\alpha_{M,N}=\beta_{M,N},
\qquad
\sigma^{\bullet}_{M,N}\beta_{N,M}=\alpha_{N,M}.
\end{equation}
Consider two tetramodules $F$ and $G$, both with underlying
vector space $H\otimes H$. On $F$ take outer actions and diagonal
coactions:
\begin{align*}
r(a\otimes b)s
&=ra\otimes bs,\\
\lambda_F(a\otimes b)
&=a_1b_1\otimes(a_2\otimes b_2),\\
\rho_F(a\otimes b)
&=(a_1\otimes b_1)\otimes a_2b_2.
\end{align*}
On $G$ take diagonal actions and outer coactions:
\begin{align*}
r(a\otimes b)s
&=r_1as_1\otimes r_2bs_2,\\
\lambda_G(a\otimes b)
&=a_1\otimes(a_2\otimes b),\\
\rho_G(a\otimes b)
&=(a\otimes b_1)\otimes b_2.
\end{align*}
These are the external products $H\boxtimes_1H$ and
$H\boxtimes_2H$ of \cite[Section~2.2.1]{Shoikhet}.
Their tetramodule axioms follow from the bialgebra identities
and do not require an antipode.\\
Set $f=1\otimes1\in F$ and, for $h\in H$, set
$g_h=1\otimes h\in G$. We have
\begin{align*}
\lambda_F(f)&=1\otimes f,
&
\rho_F(f)&=f\otimes1,\\
\lambda_G(g_h)&=1\otimes g_h,
&
\rho_G(g_h)&=g_{h_1}\otimes h_2,
\end{align*}
as well as $fh=1\otimes h$ and $hf=h\otimes1$ in $F$.
In particular, $f\otimes g_h$ belongs to $F\square^HG$, and
\eqref{eq:tetramodule-alpha}--\eqref{eq:tetramodule-beta} give
\begin{equation}\label{eq:tetramodule-common-image}
\alpha_{F,G}(f\otimes_Hg_h)
=
f\square^Hg_h
=
\beta_{G,F}(g_h\otimes_Hf).
\end{equation}
Applying $\sigma^{\bullet}_{F,G}$ to \eqref{eq:tetramodule-common-image} and using
\eqref{eq:tetramodule-comparison-identities}, we obtain
\[
\beta_{F,G}(f\otimes_Hg_h)
=
\alpha_{G,F}(g_h\otimes_Hf).
\]
Expanding both sides yields
\[
g_{h_1}\square^H(fh_2)
=
g_{h_1}\square^H(h_2f).
\]
Under the inclusion
$G\square^HF\hookrightarrow G\otimes F=H^{\otimes4}$,
this becomes
\[
1\otimes h_1\otimes1\otimes h_2
=
1\otimes h_1\otimes h_2\otimes1.
\]
Applying
$\varepsilon\otimes\varepsilon\otimes\mathrm{id}\otimes\mathrm{id}$ to the latter gives $1\otimes h=h\otimes1$.
Applying $\varepsilon\otimes\mathrm{id}$ once more gives
$h=\varepsilon(h)1$ for every $h\in H$.
Thus the unit $\Bbbk\to H$ and the counit $H\to\Bbbk$ are inverse
bialgebra isomorphisms.\\
Suppose now that $\sigma^\circ$ is a braiding on
$({}_H^H\mm_H^H,\otimes_H,H)$ compatible with the canonical
interchange. Its compatibility condition reads
\[
\zeta_{C,D,A,B}\sigma^\circ_{A\square_H B,C\square_H D}
=
(\sigma^\circ_{A,C}\square_H \sigma^\circ_{B,D})\zeta_{A,B,C,D}.
\]
Specializing to $(A,B,C,D)=(H,M,N,H)$ and $(M,H,H,N)$,
respectively, and suppressing the canonical unit constraints,
we obtain
\[
\alpha_{N,M}\sigma^\circ_{M,N}=\beta_{M,N},
\qquad
\beta_{N,M}\sigma^\circ_{M,N}=\alpha_{M,N}.
\]
Fix $h\in H$. Since $F=H\otimes H$ has the outer left action,
we may write
\[
\sigma^\circ_{F,G}(f\otimes_H g_h)
=
\sum_i(a_i\otimes b_i)\otimes_H(1\otimes d_i)
\]
for suitable $a_i,b_i,d_i\in H$. Indeed, balancing gives
\[
(a\otimes b)\otimes_H(u\otimes v)
=
(a u_1\otimes b u_2)\otimes_H(1\otimes v).
\]
The preceding comparison identities give
\begin{align*}
\alpha_{G,F}\sigma^\circ_{F,G}(f\otimes_H g_h)
&=\beta_{F,G}(f\otimes_H g_h),\\
\beta_{G,F}\sigma^\circ_{F,G}(f\otimes_H g_h)
&=\alpha_{F,G}(f\otimes_H g_h).
\end{align*}
Expand the first equality in $G\otimes F=H^{\otimes4}$ and apply
$\mathrm{id}\otimes\varepsilon\otimes\mathrm{id}\otimes\mathrm{id}$.
Expand the second equality in $F\otimes G=H^{\otimes4}$ and apply
$\mathrm{id}\otimes\mathrm{id}\otimes\varepsilon\otimes\mathrm{id}$.
Using the explicit formulas for $\alpha$ and $\beta$, we obtain
\begin{align*}
&\sum_i a_i d_{i1}\otimes b_i\otimes d_{i2}
=1\otimes1\otimes h,\\
&\sum_i a_i\otimes d_{i1}\otimes b_i d_{i2}
=1\otimes1\otimes h.
\end{align*}
Applying $\mathrm{id}\otimes\varepsilon\otimes\mathrm{id}$ to
the first equality and
$\mathrm{id}\otimes\mathrm{id}\otimes\varepsilon$ to the second
gives
\begin{align*}
&\sum_i \varepsilon(b_i)a_i d_{i1}\otimes d_{i2}
=1\otimes h,\\
&\sum_i \varepsilon(b_i)a_i\otimes d_i
=\varepsilon(h)1\otimes1.
\end{align*}
Apply the linear map $H\otimes H\rightarrow H\otimes H,\ a\otimes d\mapsto a d_1\otimes d_2$
to the latter equality. Comparing with the former yields $1\otimes h=\varepsilon(h)1\otimes1$. Hence $h=\varepsilon(h)1$ for every $h\in H$, and therefore
$H\cong \Bbbk$.\\
Together with the $\bullet$-braided case and the braided
duoidal structure for $H=\Bbbk$, this proves all the asserted
equivalences.
\end{proof}

\begin{remark}
Compatibility with the canonical interchange is essential.
Theorem \ref{prop:tetramodule-braided-duoidal-obstruction}
concerns braidings making this fixed duoidal structure
$\bullet$-braided; it does not rule out braidings on either
underlying monoidal category considered separately.
\end{remark}

We introduce the following definition.

\begin{definition}\label{def:preCartier}
    Let $(\Cc,\circ,I,\bullet,J)$ be a duoidal category such that $\Cc$ is pre-additive and both monoidal products are additive in each variable. 
    
    If it it $\circ$-braided, we call it $\circ$-pre-Cartier if there exists an infinitesimal braiding $t^{\circ}$ such that $(\Cc,\circ,I,\sigma^{\circ},t^{\circ})$ is pre-Cartier and the diagrams in \eqref{diagcircprecartier} commute.
\begin{equation}\label{diagcircprecartier}
\begin{tikzcd}
	(A\bullet B)\circ(C\bullet D) & (A\circ C)\bullet(B\circ D) &&& J\circ J \\
	(A\bullet B)\circ(C\bullet D) & (A\circ C)\bullet(B\circ D) && J & J\circ J
	\arrow[from=1-1, to=1-2,"\zeta_{A,B,C,D}"]
	\arrow[from=1-1, to=2-1,"t^{\circ}_{A\bullet B, C\bullet D}"']
	\arrow[from=1-2, to=2-2,"(A\circ C)\bullet t^{\circ}_{B,D}+t^{\circ}_{A,C}\bullet(B\circ D)"]
	\arrow[from=1-5, to=2-4,"0"']
	\arrow[from=1-5, to=2-5,"t^{\circ}_{J,J}"]
	\arrow[from=2-1, to=2-2,"\zeta_{A,B,C,D}"']
	\arrow[from=2-5, to=2-4,"m^{\circ}_{J}"]
\end{tikzcd}
\end{equation}
If it is $\bullet$-braided, we call it $\bullet$-pre-Cartier if there exists an infinitesimal braiding $t^{\bullet}$ such that $(\Cc,\bullet,J,\sigma^{\bullet},t^{\bullet})$ is pre-Cartier and the diagrams in \eqref{diagbulletprecartier} commute.
\begin{equation}\label{diagbulletprecartier}
\begin{tikzcd}
	(A\bullet B)\circ(C\bullet D) & (A\circ C)\bullet(B\circ D) && I\bullet I \\
	(A\bullet B)\circ(C\bullet D) & (A\circ C)\bullet(B\circ D) & I & I\bullet I
	\arrow[from=1-1, to=1-2,"\zeta_{A,B,C,D}"]
	\arrow[from=1-1, to=2-1,"(A\bullet B)\circ t^{\bullet}_{C,D}+t^{\bullet}_{A,B}\circ(C\bullet D)"']
	\arrow[from=1-2, to=2-2,"t^{\bullet}_{A\circ C,B\circ D}"]
	\arrow[from=2-3, to=1-4,"0"]
	\arrow[from=2-4, to=1-4,"t^{\bullet}_{I,I}"']
	\arrow[from=2-1, to=2-2,"\zeta_{A,B,C,D}"']
	\arrow[from=2-3, to=2-4,"\Delta^{\bullet}_{I}"']
\end{tikzcd}
\end{equation}
A braided duoidal category $(\Cc,\circ,I,\bullet,J)$ is \textit{pre-Cartier} if it is $\circ$-pre-Cartier and $\bullet$-pre-Cartier. 
\end{definition}

\begin{remark}
We recall that the axioms of a pre-Cartier category are obtained by looking at the first degree in $\hslash$ and asking the following morphism  
\[
c_{X,Y}:=\sigma_{X,Y}(\id_{X\ot Y}+\hslash t_{X,Y}+\mathcal{O}(\hslash^2))
\]
to be a braiding, given a pre-additive braided monoidal category $(\Cc,\ot,I,\sigma)$. The definition of $\circ$-pre-Cartier duoidal category is obtained asking that the latter morphism satisfies \eqref{diagramscircbraided}. Indeed, at first order in $\hslash$, we have
\[
\zeta_{C,D,A,B}\sigma^{\circ}_{A\bullet B,C\bullet D}t^{\circ}_{A\bullet B,C\bullet D}=(\sigma^{\circ}_{A,C}\bullet\sigma^{\circ}_{B,D}t^{\circ}_{B,D})\zeta_{A,B,C,D}+(\sigma^{\circ}_{A,C}t^{\circ}_{A,C}\bullet\sigma^{\circ}_{B,D})\zeta_{A,B,C,D},
\]
hence
\[
(\sigma^{\circ}_{A,C}\bullet\sigma^{\circ}_{B,D})\zeta_{A,B,C,D}t^{\circ}_{A\bullet B,C\bullet D}=(\sigma^{\circ}_{A,C}\bullet\sigma^{\circ}_{B,D}t^{\circ}_{B,D})\zeta_{A,B,C,D}+(\sigma^{\circ}_{A,C}t^{\circ}_{A,C}\bullet\sigma^{\circ}_{B,D})\zeta_{A,B,C,D},
\]
which is equivalent to
\[
\zeta_{A,B,C,D}t^{\circ}_{A\bullet B,C\bullet D}=((A\circ C)\bullet t^{\circ}_{B,D}+t^{\circ}_{A,C}\bullet(B\circ D))\zeta_{A,B,C,D}.
\]
Moreover, we have $0=m^{\circ}_{J}\sigma^{\circ}_{J,J}t^{\circ}_{J,J}=m^{\circ}_{J}t^{\circ}_{J,J}$. Similarly, the definition of $\bullet$-pre-Cartier duoidal category is obtained.
\end{remark}

\begin{proposition}\label{prop:trivialpreCartier}
Let $(\Cc,\ot,I,\sigma)$ be a pre-additive symmetric monoidal category whose tensor product is additive in each variable. For the associated braided duoidal category $(\Cc,\ot,I,\sigma,\ot,I,\sigma)$, a $\circ$-pre-Cartier structure necessarily has $t^{\circ}=0$, and a $\bullet$-pre-Cartier structure necessarily has $t^{\bullet}=0$.
\end{proposition}

\begin{proof}
The infinitesimal braid identities imply $t_{X,I} =0=t_{I,X}$ for every $X\in\Cc$, see \cite[Remark 1.2 (ii)]{ABSW}. In \eqref{diagcircprecartier}, put $B=C=I$. Since $\zeta_{A,I,I,D}=\mathrm{id}_{A\ot D}$, the right-hand side is zero and $t^{\circ}_{A,D}=0$. The same substitution in \eqref{diagbulletprecartier} gives $t^{\bullet}_{A,D}=0$.
\end{proof}




We now study the one-sided pre-Cartier structures of the duoidal categories of bimodules and bicomodules. Proposition \ref{prop:trivialcanonicalform} prevents these categories from carrying braidings on both products over a nontrivial bialgebra. The braidings of Proposition \ref{prop:braidedduoidalbimod}, however, admit nonzero infinitesimal braidings compatible with the interchange.

\begin{lemma}\label{lem:infbraidbimod}
    Let $(H,\Rr)$ be a quasitriangular bialgebra such that $(H,(\Rr^{-1})^{\mathrm{op}},\chi)$ is pre-Cartier. Then, the braided monoidal category $(_{H}\mm_{H},\ot,\Bbbk,\sigma^{\bullet})$, with $\sigma^{\bullet}$ defined as in 1) of Proposition \ref{prop:braidedduoidalbimod}, is pre-Cartier with infinitesimal braiding defined by
\begin{equation}\label{infbraidingbimod}
t_{X,Y}:X\ot Y\to X\ot Y,\ x\ot y\mapsto \chi^{i}x\ot\chi_{i}y-x\chi^{i}\ot y\chi_{i}
\end{equation}
for all $X,Y$ in ${}_{H}\mm_{H}$.
\end{lemma}
\begin{proof}
    As observed in Remark \ref{rmk:braidingbimodquasi}, the braiding $\sigma$ corresponds to a quasitriangular structure on $H\ot H^{\mathrm{op}}$ obtained using the quasitriangular structure $(\Rr^{-1})^{\mathrm{op}}$ on $H$ and $\Rr^{\mathrm{op}}$ on $H^{\mathrm{op}}$. Since $(H,(\Rr^{-1})^{\mathrm{op}},\chi)$ is pre-Cartier, we have that $(H^{\mathrm{op}},\Rr^{\mathrm{op}},\chi)$ is also pre-Cartier, as $\chi\cdot_{\op}\Delta(\cdot)=\Delta(\cdot)\chi=\chi\Delta(\cdot)=\Delta(\cdot)\cdot_{\op}\chi$ and
\begin{align*}
(\id\ot\Delta)(\chi)&=\chi_{12}+\Rr^{\mathrm{op}}_{12}\chi_{13}(\Rr^{\mathrm{op}})^{-1}_{12}=\chi_{12}+(\Rr^{\mathrm{op}})^{-1}_{12}\cdot_{\op}\chi_{13}\cdot_{\op}\Rr^{\mathrm{op}}_{12},\\
(\Delta\ot\id)(\chi)&=\chi_{23}+\Rr^{\mathrm{op}}_{23}\chi_{13}(\Rr^{\mathrm{op}})^{-1}_{23}=\chi_{23}+(\Rr^{\mathrm{op}})^{-1}_{23}\cdot_{\op}\chi_{13}\cdot_{\op}\Rr^{\mathrm{op}}_{23}.
\end{align*}
Therefore, by \cite[Proposition 2.15 and the remark afterwards]{ABSW} we know that the quasitriangular bialgebra $\big(H\ot H^{\mathrm{op}},(\id\ot\tau\ot\id)((\Rr^{-1})^{\mathrm{op}}\ot\Rr^{\mathrm{op}})\big)$ is moreover pre-Cartier with infinitesimal $\Rr$-matrix $\chi^{i}\ot1_{H}\ot\chi_{i}\ot1_{H}-1_{H}\ot\chi^{i}\ot1_{H}\ot\chi_{i}$. The latter pre-Cartier structure clearly corresponds to the infinitesimal braiding $t$ given in \eqref{infbraidingbimod}. 
\begin{invisible}
\[
\begin{split}
(\chi^{i}\ot1_{H}\ot\chi_{i}\ot1_{H}-1_{H}\ot\chi^{i}\ot1_{H}\ot\chi_{i})(x\ot y)&=(\chi^{i}\ot1_{H})x\ot(\chi_{i}\ot1_{H})y-(1_{H}\ot\chi^{i})x\ot(1_{H}\ot\chi_{i})y\\&=\chi^{i}x\ot\chi_{i}y-x\chi^{i}\ot y\chi_{i}
\end{split}
\]
\end{invisible}
\end{proof}
\begin{remark}\label{remark: why this infinitesimal braiding}
    As remarked in \cite[Proposition 2.15 and the remark afterwards]{ABSW}, any linear combination $\alpha\chi_{13}+\beta\chi_{24}$ is an infinitesimal $R$-matrix for the quasitriangular bialgebra $\big(H\ot H^{\mathrm{op}},(\id\ot\tau\ot\id)((\Rr^{-1})^{\mathrm{op}}\ot\Rr^{\mathrm{op}})\big)$. However, in view of Proposition \ref{prop:bimodulespreCartier} we are interested in the specific infinitesimal $R$-matrix $\chi_{13}-\chi_{24}$. This is explained more thoroughly in Remark \ref{remark: only linear combination}.
\end{remark}

\begin{corollary}\label{cor: preCartier for triangular}
    Let $(H,\Rr,\chi)$ be a pre-Cartier triangular bialgebra. Then $(_{H}\mm_{H},\ot,\Bbbk,\sigma)$ is pre-Cartier with infinitesimal braiding $t$ as in \eqref{infbraidingbimod}.
\end{corollary}

Dually, we obtain the following infinitesimal braiding for bicomodules. For these results we rely on the classification of coquasitriangular structures on the tensor product of bialgebras \cite[Theorem 2.10]{chen2}.

\begin{lemma}\label{lem:infbraidbicomod}
    Let $(H,\Rr,\chi)$ be a pre-Cartier coquasitriangular bialgebra. Then, the braided monoidal category $({}^H\mm^H, \ot,  \Bbbk, \sigma^{\circ})$, where $\sigma^{\circ}$ is defined as in 2) of Proposition \ref{prop:braidedduoidalbimod}, is pre-Cartier with infinitesimal braiding defined by
    \begin{equation}\label{infbraidingbicomod}
        t_{X,Y}\colon X\ot Y\to X\ot Y,\ x\ot y\mapsto \chi(x_{-1}\ot y_{-1})x_0\ot y_0-\chi(x_1\ot y_1)x_0\ot y_0
    \end{equation}
    for all $X,Y$ in ${}^H\mm^H$.
\end{lemma}
\begin{proof}
Dual to Lemma \ref{lem:infbraidbimod}.
\begin{invisible}
    As observed in Remark \ref{rmk:braidingbimodquasi}, the braiding $\sigma$ corresponds to a coquasitriangular structure on $H\ot H^{\cop}$ obtained using the coquasitriangular structure $(\Rr^{-1})^{\op}$ on $H$ and $\Rr^{\op}$ on $H^{\cop}$. Since $(H, (\Rr^{-1})^{\op}, \chi)$ is pre-Cartier, it holds that $(H^{\cop}, \Rr^{\op}, \chi)$ is also pre-Cartier, as
    \begin{align*}(u_{H^{\cop}}\chi)\ast^{\cop}m_{H^{\cop}} &= (u_H\chi)\ast^{\cop} m_H = m_H\ast (u_H\chi) = (u_H\chi)\ast m_H = m_{H^{\cop}} \ast^{\cop}(u_{H^{\cop}}\chi),
        \\\chi(\id\ot m_{H^{\cop}}) &= \chi(\id\ot m_H)
        \\&= \chi_{12}+ ((\Rr^{\op})^{-1})^{-1}_{12}\ast \chi_{13}\ast (\Rr^{\op})^{-1}_{12}
        \\&= \chi_{12} + \Rr^{\op}_{12}\ast \chi_{13}\ast (\Rr^{\op})^{-1}_{12}
        \\&= \chi_{12} + (\Rr^{\op})^{-1}_{12}\ast^{\cop}\chi_{13}\ast^{\cop} \Rr^{\op}_{12},
        \\\chi(m_{H^{\cop}}\ot \id)&= \chi(m_H\ot \id)
        \\&= \chi_{23}+ ((\Rr^{\op})^{-1})^{-1}_{23}\ast \chi_{13}\ast (\Rr^{\op})^{-1}_{23}
        \\&= \chi_{23} + \Rr^{\op}_{23}\ast \chi_{13}\ast (\Rr^{\op})^{-1}_{23}
        \\&= \chi_{23} + (\Rr^{\op})^{-1}_{23}\ast^{\cop}\chi_{13}\ast^{\cop} \Rr^{\op}_{23}.
    \end{align*}
    Therefore, by \cite[Theorem 2.10]{chen2} we know that the coquasitriangular bialgebra $(H\ot H^{\cop}, ((\Rr^{-1})^{\op}\ot \Rr^{\op})(H\ot \tau\ot \id))$ is pre-Cartier with infinitesimal $\Rr$-form $\chi_{13}-\chi_{24}$. This corresponds to the infinitesimal braiding $t$ given in \eqref{infbraidingbicomod}.
    \[
         (\chi_{13}-\chi_{24})(x_{-1}\ot  x_1 \ot y_{-1}\ot y_1)x_0\ot y_0 = \chi(x_{-1}\ot y_{-1})x_0\ot y_0-\chi(x_1\ot y_1)x_0\ot y_0.\qedhere
         \]
\end{invisible}
\end{proof}

We moreover find the following immediate corollary, dual to Corollary \ref{cor: preCartier for triangular}.

\begin{corollary}
    Let $(H,\Rr, \chi)$ be a pre-Cartier cotriangular bialgebra. Then $({}^H\mm^H, \ot, \Bbbk, \sigma)$ is pre-Cartier with infinitesimal braiding $t$ as in \eqref{infbraidingbicomod}.
\end{corollary}

We now obtain the following results on the pre-Cartier structures of the duoidal categories of bi(co)modules.

\begin{proposition}\label{prop:bimodulespreCartier}
Under the assumptions of Lemma \ref{lem:infbraidbimod}, $(_{H}\mm_{H},\ot_{H},H,\ot,\Bbbk)$ is $\bullet$-pre-Cartier with braiding \eqref{eq:braidingbimod} and infinitesimal braiding \eqref{infbraidingbimod}.\\
Dually, under the assumptions of Lemma \ref{lem:infbraidbicomod}, $({}^H\mm^{H},\ot,\Bbbk,\square^H,H)$ is $\circ$-pre-Cartier with braiding \eqref{braidingbicom} and infinitesimal braiding \eqref{infbraidingbicomod}.
\end{proposition}

\begin{proof}
    By Lemma \ref{lem:infbraidbimod}, the braided monoidal category $(_{H}\mm_{H},\ot,\Bbbk,\sigma)$ is pre-Cartier with infinitesimal braiding $t^{\bullet}$ defined as in \eqref{infbraidingbimod}. We compute
\[
\begin{split}
    t^{\bullet}_{A\ot_{H} C,B\ot_{H} D}\zeta_{A,B,C,D}((a\ot b)&\ot_{H}(c\ot d))=t^{\bullet}_{A\ot_{H} C,B\ot_{H} D}((a\ot_{H} c)\ot(b\ot_{H}d))\\&=\chi^{i}(a\ot_{H} c)\ot\chi_{i}(b\ot_{H}d)-(a\ot_{H}c)\chi^{i}\ot(b\ot_{H}d)\chi_{i}\\&=(\chi^{i}a\ot_{H} c)\ot(\chi_{i}b\ot_{H}d)-(a\ot_{H}c\chi^{i})\ot(b\ot_{H}d\chi_{i})\\&=(a\ot_{H}\chi^{i}c)\ot(b\ot_{H}\chi_{i}d)-(a\ot_{H} c\chi^{i})\ot(b\ot_{H} d\chi_{i})\\&+(\chi^{i}a\ot_{H} c)\ot(\chi_{i}b\ot_{H}d)-(a\chi^{i}\ot_{H}c)\ot(b\chi_{i}\ot_{H} d)\\&=\zeta_{A,B,C,D}((a\ot b)\ot_{H}(\chi^{i}c\ot\chi_{i}d)-(a\ot b)\ot_{H}(c\chi^{i}\ot d\chi_{i}))\\&+\zeta_{A,B,C,D}((\chi^{i}a\ot\chi_{i}b)\ot_{H}(c\ot d)-(a\chi^{i}\ot b\chi_{i})\ot_{H}(c\ot d))\\&=\zeta_{A,B,C,D}((A\ot B)\ot_{H}t^{\bullet}_{C,D}+t^{\bullet}_{A,B}\ot_{H}(C\ot D))((a\ot b)\ot_{H}(c\ot d))
\end{split}
\]
and, utilizing \eqref{cqtr1}, $t^{\bullet}_{H,H}\Delta^{\bullet}_{H}(x)=t^{\bullet}_{H,H}(x_{1}\ot x_{2})=\chi^{i}x_{1}\ot\chi_{i}x_{2}-x_{1}\chi^{i}\ot x_{2}\chi_{i}=0$.

The bicomodule statement follows by duality: the braiding \eqref{braidingbicom} and infinitesimal braiding \eqref{infbraidingbicomod} satisfy the two compatibility
conditions in \eqref{diagcircprecartier}.
\begin{invisible}
    By Theorem \ref{thm:infbraidbicomod} the braided monoidal category $({}^H\mm^H, \square^H, H, \sigma^{\bullet})$ has no non-trivial infinitesimal braiding. Equation \eqref{diagbulletprecartier} is clearly satisfied with $t^{\bullet} = 0$. By Lemma \ref{lem:infbraidbicomod} $({}^H\mm^H, \ot, \Bbbk, \sigma)$ is pre-Cartier with infinitesimal braiding $t^{\circ}$ defined as in \eqref{infbraidingbicomod}. Let $A,B,C,D$ be in ${}^H\mm^H$ and $a\square b\in A\square B$, $c\square d\in C\square D$, then it follows that $a_0\ot a_1\ot b = a\ot b_{-1} \ot b_0$ and $c_0\ot c_1\ot d = c\ot d_{-1}\ot d_0$. Hence
        \begin{align*}
            \zeta_{A,B,C,D}t^{\circ}_{A\square B, C\square D}(a\square b &\ot c\square d)
             =\zeta_{A,B,C,D}(\chi(a_{-1}\ot c_{-1})a_0\square b\ot c_0\square d - \chi(b_1\ot d_1)a\square b_0\ot c\square d_0)
            \\&= \chi(a_{-1}\ot c_{-1})a_0\ot c_0\square b\ot d - \chi(b_1\ot d_1)a\ot c\square b_0\ot d_0
            \\&= \chi(a_{-1}\ot c_{-1})a_0\ot c_0\square b\ot d + \chi(a_1\ot d_{-1})a_0\ot c\square b\ot d_0
            \\&\hspace{1em} - \chi(a_1\ot d_{-1})a_0\ot c\square b\square d_0 - \chi(b_1\ot d_1)a\ot c\square b_0\ot d_0
            \\&= \chi(a_{-1}\ot c_{-1})a_0\ot c_0\square b\ot d + \chi(b_{-1}\ot d_{-1})a\ot c\square b_0\ot d_0
            \\&\hspace{1em} - \chi(a_1\ot c_{1})a_0\ot c_0\square b\square d - \chi(b_1\ot d_1)a\ot c\square b_0\ot d_0
            \\&= (t^{\circ}_{A,C}\square B\ot D + A\ot C\square t^{\circ}_{B,D})\zeta_{A,B,C,D}(a\square b\ot c\square d).
        \end{align*}
        Hence the first identity of \eqref{diagcircprecartier} holds. Moreover, utilizing \eqref{pcct1},
        \[
        m_H^{\circ}t^{\circ}_{H,H}(x\ot y) = \chi(x_{-1}\ot y_{-1})x_0y_0 - \chi(x_1\ot y_1)x_0y_0 = \chi(x_1\ot y_1)x_2y_2 - \chi(x_2\ot y_2)x_1y_1 = 0.
        \]
\end{invisible}
\end{proof}

\begin{remark}\label{remark: only linear combination}
    In Remark \ref{remark: why this infinitesimal braiding} we noticed that every linear combination $\alpha\chi_{13}+\beta\chi_{24}$ induces an infinitesimal braiding on $({}_H\mm_H, \ot, \Bbbk,\sigma)$, namely by
    \begin{equation}\label{eqn: general infinitesimal braiding}
    t_{X,Y}^{\alpha,\beta}:X\ot Y \to X\ot Y,\ x\ot y\mapsto \alpha\chi^ix\ot \chi_iy +\beta x\chi^i\ot y\chi_i.
    \end{equation}
If $\chi\not=0$, it defines a $\bullet$-pre-Cartier structure on the bimodule duoidal category if and only if $\alpha+\beta=0$. Indeed, the unit condition evaluated at $1_{H}$ gives $t^{\alpha,\beta}_{H,H
}\Delta(1_{H})=(\alpha+\beta)\chi=0$, hence $\alpha+\beta=0$. The vice versa follows from the previous proposition by scalar multiplication.

Similarly, for bicomodules one can define
\[
t^{\alpha,\beta}_{M,N}(m\ot n):=\alpha\chi(m_{(-1)}\ot n_{(-1)})m_{(0)}\ot n_{(0)}+\beta m_{(0)}\ot n_{(0)}\chi(m_{(1)}\ot n_{(1)}).
\]
Applying $\varepsilon$ to $m_{H}t^{\alpha,\beta}_{H,H}(a\ot b)=0$ gives $(\alpha+\beta)\chi(a\ot b)=0$. If $\chi=0$, every choice of coefficients gives the same zero structure. The uniqueness up to scalar just described is only within these specified two-term families; it is not a classification of all infinitesimal braidings.
\end{remark}

We provide two easy explicit examples of one-sided pre-Cartier structures for bimodules and bicomodules.

\begin{example}\label{ex:1}
Let $H=\Bbbk[x]$ with $x$ primitive, $\Rr=1\ot1$, and $\chi=x\ot x$. Commutativity and primitivity give the pre-Cartier identities, so $(H,\Rr,\chi)$ is Cartier triangular. An $H$-bimodule $M$ is a vector space with two commuting endomorphisms $L_{M}(m)=xm$ and $R_{M}(m)=mx$. Proposition \ref{prop:bimodulespreCartier} makes $(_{H}\mm_{H},\ot_{H},H,\ot,\Bbbk)$ $\bullet$-pre-Cartier with the ordinary flip and 
\[
t^{\bullet}_{M,N}=L_{M}\ot L_{N}-R_{M}\ot R_{N}.
\]
In particular, take the one-dimensional bimodule $M=\Bbbk m$ with $xm=m$ and $mx=0$. Then $t^{\bullet}_{M,M}(m\ot m)=m\ot m$, so the infinitesimal braiding is nonzero. The interchange compatibility amounts to the cancellation of the middle actions using $ax\ot_{H}c=a\ot_{H}xc$.
\end{example}

\begin{example}\label{ex:2}
Let $H=\Bbbk[g,g^{-1}]$, with $g$ grouplike, and put
\[
\Rr(a\ot b)=\varepsilon(a)\varepsilon(b),\qquad \chi(g^{p}\ot g^{q})=pq\qquad (p,q\in\mathbb{Z}),
\]
extending the second formula bilinearly; integers are viewed as elements of $\Bbbk$. The identities $p(q+r)=pq+pr$ and $(p+q)r=pr+qr$ show that $(H,\Rr,\chi)$ is pre-Cartier cotriangular.\\
An $H$-bicomodule is a $\mathbb{Z}^{2}$-graded vector space $M=\bigoplus_{p,q}{M_{p,q}}$, where $m\in M_{p,q}$ has coactions $g^{p}\ot m$ and $m\ot g^{q}$. For $m\in M_{p,q}$ and $n\in N_{r,s}$, by proposition \ref{prop:bimodulespreCartier} the infinitesimal braiding is
\[
t^{\circ}_{M,N}(m\ot n)=(pr-qs)m\ot n.
\]
Together with the flip, it makes $(^{H}\mm^{H},\ot,\Bbbk,\square^{H},H)$ $\circ$-pre-Cartier. For a one-dimensional $M$ concentrated in degree $(1,0)$, $t^{\circ}_{M,M}=\mathrm{id}_{M\ot M}\not=0$.
\end{example}

We observe that for the given quasitriangular structure $(\id\ot\tau\ot\id)((\Rr^{-1})^{\mathrm{op}}\ot\Rr^{\mathrm{op}})$ on $H\ot H^{\mathrm{op}}$, one could find other infinitesimal $\Rr$-matrices in addition to $\chi_{13}-\chi_{24}$, hence other infinitesimal braidings besides $t^{\bullet}$ defined as in \eqref{infbraidingbimod}. Dually, the pre-Cartier coquasitriangular structure $(\Rr\ot\Rr^{-1})(\id\ot\tau\ot\id)$ on $H\ot H^{\cop}$ allows for other infinitesimal $\Rr$-forms in addition to $\chi_{24}-\chi_{13}$, and thus other infinitesimal braidings besides $t^{\circ}$ defined as in \eqref{infbraidingbicomod}.

\begin{lemma}\label{lem:newchiconjugation}
    Let $(H,\Rr,\chi)$ be a pre-Cartier triangular bialgebra. Then $(H,\Rr,\chi^{\Rr}=\Rr^{-1}\chi^{\mathrm{op}}\Rr)$ is a pre-Cartier triangular bialgebra. Moreover, $(H,\Rr,\chi^{\Rr})$ is Cartier if and only if $(H,\Rr,\chi)$ is Cartier, in which case $\chi^{\Rr}=\chi$.

    Dually, if $(H,\Rr, \chi)$ is a pre-Cartier cotriangular bialgebra, then $(H,\Rr, \chi^{\Rr} = \Rr^{-1}\ast \chi^{\op}\ast \Rr)$ is a pre-Cartier cotriangular bialgebra. Moreover, $(H,\Rr, \chi^{\Rr})$ is Cartier if and only if $(H,\Rr,\chi)$ is Cartier, in which case $\chi^{\Rr} = \chi$.
\end{lemma}

\begin{proof}
    We compute
\[
\chi^{\Rr}\Delta(\cdot)=\Rr^{-1}\chi^{\mathrm{op}}\Rr\Delta(\cdot)\overset{\eqref{qtr1}}{=}\Rr^{-1}\chi^{\mathrm{op}}\Delta^{\mathrm{op}}(\cdot)\Rr\overset{\eqref{cqtr1}}{=}\Rr^{-1}\Delta^{\mathrm{op}}(\cdot)\chi^{\mathrm{op}}\Rr\overset{\eqref{qtr1}}{=}\Delta(\cdot)\Rr^{-1}\chi^{\mathrm{op}}\Rr=\Delta(\cdot)\chi^{\Rr},
\]
so that \eqref{cqtr1} holds for $(H,\Rr,\chi^{\Rr})$. Now, we compute
\[
\begin{split}
(\id\ot\Delta)(\chi^{\Rr})&=(\id\ot\Delta)(\Rr^{-1}\chi^{\mathrm{op}}\Rr)=(\id\ot\Delta)(\Rr)^{-1}(\id\ot\Delta)(\chi^{\mathrm{op}})(\id\ot\Delta)(\Rr)\\&\overset{\eqref{qtr2},\eqref{cqtr3}}{=}(\Rr_{13}\Rr_{12})^{-1}(\chi^{\mathrm{op}}_{13}+(\Rr^{-1})^{\mathrm{op}}_{13}\chi_{12}^{\mathrm{op}}\Rr^{\mathrm{op}}_{13})(\Rr_{13}\Rr_{12})\\&\overset{(!)}{=}\Rr_{12}^{-1}\Rr_{13}^{-1}\chi^{\mathrm{op}}_{13}\Rr_{13}\Rr_{12}+\Rr_{12}^{-1}\Rr_{13}^{-1}\Rr_{13}\chi_{12}^{\mathrm{op}}\Rr^{-1}_{13}\Rr_{13}\Rr_{12}\\&=\Rr_{12}^{-1}\Rr_{13}^{-1}\chi^{\mathrm{op}}_{13}\Rr_{13}\Rr_{12}+\Rr_{12}^{-1}\chi_{12}^{\mathrm{op}}\Rr_{12}\\&=\chi^{\Rr}_{12}+\Rr^{-1}_{12}\chi^{\Rr}_{13}\Rr_{12},
\end{split}
\]
where $(!)$ follows since $\Rr^{\mathrm{op}}=\Rr^{-1}$. Hence \eqref{cqtr2} holds for $(H,\Rr,\chi^{\Rr})$. Similarly, using that \eqref{qtr3} and \eqref{cqtr2} hold for $(H,\Rr,\chi)$, we get that \eqref{cqtr3} holds for $(H,\Rr,\chi^{\Rr})$.
\begin{invisible}
\[
\begin{split}
    (\Delta\ot\id)(\chi^{\Rr})&=(\Delta\ot\id)(\Rr^{-1}\chi^{\mathrm{op}}\Rr)=(\Delta\ot\id)(\Rr)^{-1}(\Delta\ot\id)(\chi^{\mathrm{op}})(\Delta\ot\id)(\Rr)\\&\overset{\eqref{qtr3},\eqref{cqtr2}}{=}(\Rr_{13}\Rr_{23})^{-1}(\chi^{\mathrm{op}}_{13}+(\Rr^{-1})^{\mathrm{op}}_{13}\chi^{\mathrm{op}}_{23}\Rr^{\mathrm{op}}_{13})(\Rr_{13}\Rr_{23})\\&\overset{(!)}{=}\Rr^{-1}_{23}\Rr_{13}^{-1}\chi^{\mathrm{op}}_{13}\Rr_{13}\Rr_{23}+\Rr^{-1}_{23}\Rr_{13}^{-1}\Rr_{13}\chi^{\mathrm{op}}_{23}\Rr^{-1}_{13}\Rr_{13}\Rr_{23}\\&=\Rr^{-1}_{23}\Rr_{13}^{-1}\chi^{\mathrm{op}}_{13}\Rr_{13}\Rr_{23}+\Rr^{-1}_{23}\chi^{\mathrm{op}}_{23}\Rr_{23}\\&=\chi^{\Rr}_{23}+\Rr^{-1}_{23}\chi^{\Rr}_{13}\Rr_{23}
\end{split}
\]
\end{invisible}
Therefore, $(H,\Rr,\chi^{\Rr})$ is a pre-Cartier triangular bialgebra. Moreover, we have $\Rr\chi^{\Rr}=\Rr\Rr^{-1}\chi^{\mathrm{op}}\Rr=\chi^{\mathrm{op}}\Rr$ and $(\chi^{\Rr})^{\mathrm{op}}\Rr=(\Rr^{-1})^{\mathrm{op}}\chi\Rr^{\mathrm{op}}\Rr=\Rr\chi$. Therefore, $(H,\Rr,\chi^{\Rr})$ is Cartier if and only if $(H,\Rr,\chi)$ is Cartier. In that case $\chi^{\Rr}=\chi$.

The case for $(H,\Rr,\chi)$ a pre-Cartier cotriangular bialgebra is completely dual.
\begin{invisible}
    Let $a,b,c\in H$, then condition \eqref{pcct1} holds, since
        \begin{align*}
            \chi^{\Rr}(a_1\ot b_1)a_2b_2&= \Rr^{-1}(a_1\ot b_1)\chi(b_2\ot a_2)\Rr(a_3\ot b_3)a_4b_4
            \\&\overset{\eqref{CQT3}}{=}\Rr^{-1}(a_1\ot b_1)\chi(b_2\ot a_2)b_3a_3\Rr(a_4\ot b_4)
            \\&\overset{\eqref{pcct1}}{=} \Rr^{-1}(a_1\ot b_1)b_2a_2\chi(b_3\ot a_3)\Rr(a_4\ot b_4)
            \\&\overset{\eqref{CQT3'}}{=} a_1b_1\Rr^{-1}(a_2\ot b_2)\chi(b_3\ot a_3)\Rr(a_4\ot _4)
            \\&= a_1b_1\chi^{\Rr}(a_2\ot b_2).
        \end{align*}
    Condition \eqref{pcct2} holds by the following computation.
    \begin{align*}
        \chi^{\Rr}(a\ot bc) &= \Rr^{-1}(a_1\ot b_1c_1)\chi(b_2c_2\ot a_2)\Rr(a_3\ot b_3c_3)
        \\&\overset{\eqref{pcct3}, \eqref{CQT1}, \eqref{CQT1'}}{=} \Rr^{-1}(a_1\ot b_1)\Rr^{-1}(a_2\ot c_1)\chi(c_2\ot a_3)\Rr(a_4\ot c_3)\Rr(a_5\ot b_2)
        \\&\qquad + \Rr^{-1}(a_1\ot b_1)\Rr^{-1}(a_2\ot c_1)\Rr^{-1}(c_2\ot a_3)\chi(b_2\ot a_4)
        \\&\hspace{13em} \Rr(c_3\ot a_5)\Rr(a_6\ot c_4)\Rr(a_7\ot b_3)
        \\&\overset{(!)}{=} \Rr^{-1}(a_1\ot b_1)\chi^{\Rr}(a_2\ot c)\Rr(a_3\ot b_2)
        \\&\qquad + \Rr^{-1}(a_1\ot b_1)\ep(a_2)\ep(c_1)\chi(b_2\ot a_3)\ep(c_2)\ep(a_4)\Rr(a_5\ot b_3)
        \\&= \Rr^{-1}(a_1\ot b_1)\chi^{\Rr}(a_2\ot c)\Rr(a_3\ot b_2) 
        \\&\qquad + \Rr^{-1}(a_1\ot b_1)\chi(b_2\ot a_2)\Rr(a_3\ot b_3)\ep(c)
        \\&= \Rr^{-1}(a_1\ot b_1)\chi^{\Rr}(a_2\ot c)\Rr(a_3\ot b_2) + \chi^{\Rr}(a\ot b)\ep(c).
    \end{align*}
    In the equality marked by $(!)$ we utilized the cotriangularity of $(H,\Rr)$ twice. Condition \eqref{pcct3} holds by analogous arguments. Notice that
    \begin{align*}
        \Rr(a_1\ot b_1)\chi^{\Rr}(a_2\ot b_2) &= \Rr(a_1\ot b_1)\Rr^{-1}(a_2\ot b_2)\chi(b_3\ot a_3)\Rr(a_4\ot b_4)
        \\&= \chi(b_1\ot a_1)\Rr(a_2\ot b_2),
        \\\chi^{\Rr}(b_1\ot a_1)\Rr(a_2\ot b_2)&= \Rr^{-1}(b_1\ot a_1)\chi(a_2\ot b_2)\Rr(b_3\ot a_3)\Rr(a_4\ot b_4)
        \\&\overset{(!)}{=} \Rr^{-1}(b_1\ot a_1)\chi(a_2\ot b_2),
    \end{align*}
    where the equality marked by $(!)$ follows again by the cotriangularity of $(H,\Rr)$. Hence $(H,\Rr, \chi)$ is Cartier if and only if $(H,\Rr, \chi^{\Rr})$ is.
\end{invisible}
\end{proof}

The preceding lemma yields further one-sided pre-Cartier structures by replacing the infinitesimal datum with its conjugate transpose.

\begin{corollary}\label{cor:otherinfbraidingbimod}
Let $(H,\Rr,\chi)$ be a pre-Cartier triangular bialgebra and set $\chi^{\Rr}=\Rr^{-1}\chi^{\mathrm{op}}\Rr$. Then $(_{H}\mm_{H},\ot_{H},H,\ot,\Bbbk)$ is $\bullet$-pre-Cartier with braiding $\sigma^{\bullet}$ defined as in \eqref{eq:braidingbimod} and infinitesimal braiding $t^{\bullet}$ defined by
\begin{equation}\label{newinfbraidingbimod}
t^{\bullet}_{X,Y}:X\ot Y\to X\ot Y,\ x\ot y\mapsto 
(\chi^{\Rr})^{i}x\ot(\chi^{\Rr})_{i}y-x(\chi^{\Rr})^{i}\ot y(\chi^{\Rr})_{i}
\end{equation}
for all $X,Y\in{}_{H}\mm_{H}$.\\
Dually, let $(H,\Rr,\chi)$ be a pre-Cartier cotriangular bialgebra and $\chi^{\Rr}=\Rr^{-1}*\chi^{\mathrm{op}}*\Rr$. Then $({}^{H}\mm^{H},\ot,\Bbbk,\square^{H},H)$ is $\circ$-pre-Cartier with braiding $\sigma^{\circ}$ defined as in \eqref{braidingbicom} and infinitesimal braiding $t^{\circ}$ defined by
\begin{equation}\label{newinfbraidingbicomod}
t^{\circ}_{X,Y}:X\ot Y\to X\ot Y,\ x\ot y\mapsto 
\chi^{\Rr}(x_{-1}\ot y_{-1})x_{0}\ot y_{0}-x_{0}\ot y_{0}\chi^{\Rr}(x_{1}\ot y_{1}).
\end{equation}
for all $X,Y\in{}^{H}\mm^{H}$.
These structures are nonzero whenever $\chi\not=0$. In the Cartier case they coincide with the structures associated to $\chi$.
\end{corollary}
\begin{proof}
Lemma \ref{lem:newchiconjugation} gives a pre-Cartier structure with infinitesimal datum $\chi^{\Rr}$. We can then apply Proposition \ref{prop:bimodulespreCartier}. Since conjugation and the flip are invertible, $\chi^{\Rr}=0$ if and only if $\chi=0$. In the Cartier case, Lemma \ref{lem:newchiconjugation} gives $\chi^{\Rr}=\chi$. To verify nonvanishing in the bimodule case, take $X=Y=H$
with the regular left action and the right action induced by
$\varepsilon$. Then $t^\bullet_{H,H}(1\otimes1)=\chi_R$.
In the bicomodule case, take $X=Y=H$ with left coaction
$\Delta$ and right coaction $h\mapsto h\otimes1$.
Then $(\varepsilon\otimes\varepsilon)
t^\circ_{H,H}(a\otimes b)=\chi_R(a\otimes b)$. Thus, the displayed infinitesimal braidings are nonzero whenever
$\chi_R\ne0$.
\end{proof}

\begin{invisible}
\begin{proof}
Dual to the proof of Proposition \ref{cor:otherinfbraidingbimod}.
By Lemma \ref{lem:newchiconjugation}, $(H,\Rr, \chi^{\Rr})$ is a pre-Cartier cotriangular bialgebra, where $\chi^{\Rr} = \Rr^{-1}\ast \chi^{\op}\ast \Rr$. As observed in Remark \ref{rmk:braidingbimodquasi}, $(H^{\cop}, \Rr^{-1} = \Rr^{\op}, \chi^{\Rr})$ is again a pre-Cartier cotriangular bialgebra. Hence again by \cite[Theorem 2.10]{chen2}, $(H\ot H^{\cop}, (\Rr\ot \Rr^{\op})(\id\ot\tau\ot\id))$ is pre-Cartier with infinitesimal $\Rr$-form $\chi^{\Rr}_{13}-\chi^{\Rr}_{24}$. The latter pre-Cartier structure corresponds to the infinitesimal braiding $t$ given in \eqref{newinfbraidingbicomod}.

    By Theorem \ref{thm:infbraidbicomod}, the braided monoidal category $({}^H\mm^H, \square^H, H, \sigma^{\bullet})$ has no non-trivial infinitesimal braiding, and equation \eqref{diagbulletprecartier} is clearly satisfied with $t^{\bullet} = 0$. We have proven that the braided monoidal category $({}^H\mm^H, \ot, \Bbbk, \sigma^{\circ})$ is pre-Cartier with infinitesimal braiding $t^{\circ}$ defined as in \eqref{newinfbraidingbicomod}. Alike in the proof of Proposition \ref{prop:bimodulespreCartier}, one can show that the conditions in \eqref{diagcircprecartier} are satisfied. Let $A,B,C,D$ be $H$-bicomodules, and $a\square b \in A\square B$, $c\square d \in C\square D$. Then $a_{(0)}\ot a_{(1)}\ot b = a\ot b_{(-1)}\ot b_{(0)}$ and $c_{(0)}\ot c_{(1)}\ot d = c\ot d_{(-1)}\ot d_{(0)}$, hence
    \begin{align*}
        \zeta_{A,B,C,D}&t^{\circ}_{A\square B, C\square D}(a\square b \ot c\square d)
        \\&=\Rr^{-1}(a_{(-3)}\ot c_{(-3)})\chi(c_{(-2)}\ot a_{(-2)})\Rr(a_{(-1)}\ot c_{(-1)})a_{(0)}\ot c_{(0)}\square b\ot d
        \\&\quad - \Rr^{-1}(b_{(1)}\ot d_{(1)})\chi(d_{(2)}\ot b_{(2)})\Rr(b_{(3)}\ot d_{(3)})a\ot c\square b_{(0)}\ot d_{(0)}
        \\&=\Rr^{-1}(a_{(-3)}\ot c_{(-3)})\chi(c_{(-2)}\ot a_{(-2)})\Rr(a_{(-1)}\ot c_{(-1)})a_{(0)}\ot c_{(0)}\square b\ot d
        \\&\quad - \Rr^{-1}(a_{(1)}\ot c_{(1)})\chi(c_{(2)}\ot a_{(2)})\Rr(a_{(3)}\ot c_{(3)})a_{(0)}\ot c_{(0)}\square b\ot d
        \\&\quad + \Rr^{-1}(a_{(1)}\ot c_{(1)})\chi(c_{(2)}\ot a_{(2)})\Rr(a_{(3)}\ot c_{(3)})a_{(0)}\ot c_{(0)}\square b\ot d
        \\&\quad - \Rr^{-1}(b_{(1)}\ot d_{(1)})\chi(d_{(2)}\ot b_{(2)})\Rr(b_{(3)}\ot d_{(3)})a\ot c\square b_{(0)}\ot d_{(0)}
        \\&=\Rr^{-1}(a_{(-3)}\ot c_{(-3)})\chi(c_{(-2)}\ot a_{(-2)})\Rr(a_{(-1)}\ot c_{(-1)})a_{(0)}\ot c_{(0)}\square b\ot d
        \\&\quad - \Rr^{-1}(a_{(1)}\ot c_{(1)})\chi(c_{(2)}\ot a_{(2)})\Rr(a_{(3)}\ot c_{(3)})a_{(0)}\ot c_{(0)}\square b\ot d
        \\&\quad + \Rr^{-1}(b_{(-3)}\ot d_{(-3)})\chi(d_{(-2)}\ot b_{(-2)})\Rr(b_{(-1)}\ot d_{(-1)})a\ot c\square b_{(0)}\ot d_{(0)}
        \\&\quad - \Rr^{-1}(b_{(1)}\ot d_{(1)})\chi(d_{(2)}\ot b_{(2)})\Rr(b_{(3)}\ot d_{(3)})a\ot c\square b_{(0)}\ot d_{(0)}
        \\&= (t^{\circ}_{A,C}\square B\ot D + A\ot C \square t^{\circ}_{B,D})(a\ot c\square b\ot d)
        \\&= (t^{\circ}_{A,C}\square B\ot D + A\ot C\square t^{\circ}_{B,D})\zeta_{A,B,C,D}(a\square b\ot c\square d).
    \end{align*}
    Lastly, we compute that for any $x,y\in H$
    \begin{align*}
      m^{\circ}_H&t^{\circ}_{H,H}(x\ot y)
      \\&= \Rr^{-1}(x_{(-3)}\ot y_{(-3)})\chi(y_{(-2)}\ot x_{(-2)})\Rr(x_{(-1)}\ot y_{(-1)})x_{(0)}y_{(0)} 
      \\&\quad- \Rr^{-1}(x_{(1)}\ot y_{(1)})\chi(y_{(2)}\ot x_{(2)})\Rr(x_{(3)}\ot y_{(3)})x_{(0)}y_{(0)}
      \\&= \Rr^{-1}(x_{1}\ot y_{1})\chi(y_{2}\ot x_{2})\Rr(x_{3}\ot y_{3})x_{4}y_{4} 
      \\&\quad- \Rr^{-1}(x_{2}\ot y_{2})\chi(y_{3}\ot x_{3})\Rr(x_{4}\ot y_{4})x_{1}y_{1}
      \\&\overset{\eqref{CQT3}}{=} \Rr^{-1}(x_{1}\ot y_{1})\chi(y_{2}\ot x_{2})\Rr(x_{4}\ot y_{4})y_{3}x_{3} 
      \\&\quad- \Rr^{-1}(x_{2}\ot y_{2})\chi(y_{3}\ot x_{3})\Rr(x_{4}\ot y_{4})x_{1}y_{1}
      \\&\overset{\eqref{pcct1}}{=} \Rr^{-1}(x_{1}\ot y_{1})\chi(y_{3}\ot x_{3})\Rr(x_{4}\ot y_{4})y_{2}x_{2} 
      \\&\quad- \Rr^{-1}(x_{2}\ot y_{2})\chi(y_{3}\ot x_{3})\Rr(x_{4}\ot y_{4})x_{1}y_{1}
      \\&\overset{\eqref{CQT3'}}{=} \Rr^{-1}(x_{2}\ot y_{2})\chi(y_{3}\ot x_{3})\Rr(x_{4}\ot y_{4})x_{1}y_{1} 
      \\&\quad- \Rr^{-1}(x_{2}\ot y_{2})\chi(y_{3}\ot x_{3})\Rr(x_{4}\ot y_{4})x_{1}y_{1} = 0.\qedhere
    \end{align*}
\end{proof}
\end{invisible}

We end this paper by providing an example of pre-Cartier structure which is not simply one-sided.

\begin{example}\label{ex:fullpreCartier}
Let $\Dd$ be the category whose objects are triples $A=(A_{0},A_{1},D_{A})$, with $A_{0}$, $A_{1}$ vector spaces and $D_{A}\in\mathrm{End}_{\Bbbk}(A_{1})$. A morphism $f\colon A\to B$ is a pair $(f_{0}:A_{0}\to B_{0},f_{1}:A_{1}\to B_{1})$ such that $f_{1}D_{A}=D_{B}f_{1}$. Equivalently, $\Dd=\mathsf{Vec}_{\Bbbk}\times \prescript{}{\Bbbk[x]}{\mm}$. Define two products by
\begin{align*}
    &(A\circ B)_{0}=A_{0}\ot B_{0},\qquad(A\circ B)_{1}=(A_{0}\ot B_{1})\oplus(A_{1}\ot B_{0}),\\
    &(A\bullet B)_{0}=A_{0}\ot B_{0},\qquad (A\bullet B)_{1}=A_{1}\ot B_{1}.
\end{align*}
The corresponding endomorphisms are
\[
D_{A\circ B}=(\mathrm{id}_{A_{0}}\ot D_{B})\oplus(D_{A}\ot\mathrm{id}_{B_{0}}),\qquad D_{A\bullet B}=D_{A}\ot\mathrm{id}_{B_{1}}+\mathrm{id}_{A_{1}}\ot D_{B}.
\]
On morphisms, the products use the same tensor-product and direct-sum formulas. They are additive in each variable. Their units are $I=(\Bbbk,0,0)$, $J= (\Bbbk,\Bbbk,0)$. Both products are symmetric monoidal: their symmetries are induced by the ordinary flip, with the two summands interchanged for the degree-one part of $\sigma^{\circ}$. The interchange 
\[
\zeta_{A,B,C,E}:(A\bullet B)\circ(C\bullet E)\longrightarrow(A\circ C)\bullet(B\circ E)
\]
is the middle-factor flip in degree zero. Its source in degree one is
\[
(A_{0}\ot B_{0}\ot C_{1}\ot E_{1})\oplus(A_{1}\ot B_{1}\ot C_{0}\ot E_{0}).
\]
It sends these two summands, respectively, into the summands
\[
(A_{0}\ot C_{1})\ot(B_{0}\ot E_{1}),\qquad (A_{1}\ot C_{0})\ot(B_{1}\ot E_{0})
\]
of the target, again by the middle-factor flip. The remaining two target summands receive no contribution. The unit structure maps, using canonical identifications, are
\[
\Delta^{\bullet}_{I}=(\mathrm{id}_{\Bbbk},0):I\longrightarrow I\bullet I,\qquad \varepsilon^{\bullet}_{I}=(\mathrm{id}_{\Bbbk},0):I\longrightarrow J,\qquad m^{\circ}_{J}=(\mathrm{id}_{\Bbbk},+):(\Bbbk,\Bbbk\oplus\Bbbk,0)\to(\Bbbk,\Bbbk,0).
\]
These data make $\Dd$ a braided duoidal category. For the underlying duoidal structure one may regard the objects as graded vector spaces in degrees zero and one, carrying an endomorphism that is zero in degree zero. The product $\circ$ is the Cauchy product truncated above degree one, whereas $\bullet$ is the Hadamard product. The interchange is the usual inclusion of matching degree decompositions; compare the untruncated construction in \cite[Example 6.22]{Aguiar}.
More explicitly, the summands of an iterated $\circ$-product are indexed by choices of degrees in $\{0,1\}$ whose sum is at most one. Both paths in either interchange associativity diagram send each such source summand to the same target summand by the same permutation of tensor factors. The unit diagrams reduce to the canonical unit identities and the associativity and unit identities for the codiagonal $\Bbbk\oplus\Bbbk\to\Bbbk$. All these maps commute with the prescribed endomorphisms. The compatibility diagrams for both symmetries hold by the same summand-wise permutation check; the unit maps are symmetric as well. This verifies the braided duoidal axioms.\\
Define
\begin{equation}
t^{\circ}_{A,B}=0,\qquad t^{\bullet}_{A,B}=(0_{A_{0}\ot B_{0}},D_{A}\ot D_{B})
\end{equation}
The maps $t^{\bullet}$ are natural and commute with $D_{A\bullet B}$. Their infinitesimal braid identities reduce in degree one to the two distributive identities
\begin{align*}
    D_{A}\ot D_{B\bullet C}&=D_{A}\ot D_{B}\ot\mathrm{id}+D_{A}\ot\mathrm{id}\ot D_{C},\\
    D_{A\bullet B}\ot D_{C}&=D_{A}\ot\mathrm{id}\ot D_{C}+\mathrm{id}\ot D_{B}\ot D_{C}.
\end{align*}
In degree zero every term is zero. Thus the $\bullet$-monoidal category is pre-Cartier. \\
To check \eqref{diagbulletprecartier}, first consider the source summand $A_{0}\ot B_{0}\ot C_{1}\ot E_{1}$. After applying $\zeta$, the target operator acts as $D_{C}\ot D_{E}$ on its two degree-one factors. This is exactly the contribution of $\mathrm{id}_{A\bullet B}\circ t^{\bullet}_{C,E}$; the other term is zero. On $A_{1}\ot B_{1}\ot C_{0}\ot E_{0}$, the same verification uses $D_{A}\ot D_{B}$ and the contribution of $t^{\bullet}_{A,B}\circ\mathrm{id}_{C\bullet E}$. In degree zero both sides vanish. This proves the interchange compatibility. Since $I_{1}=0$, $t^{\bullet}_{I,I}\Delta^{\bullet}_{I}=0$. The $\circ$-infinitesimal braiding and its two duoidal compatibilities are zero; here additivity of both products ensures that tensoring
a zero morphism gives zero.\\
Finally, take $X=(0,\Bbbk,\mathrm{id}_{\Bbbk})$. We observe that the degree-one component of $t^{\bullet}_{X,X}$ is $\mathrm{id}_{\Bbbk}\ot\mathrm{id}_{\Bbbk}$, so $t^{\bullet}_{X,X}=\mathrm{id}_{X\bullet X}\not=0$.
\end{example}

The notion of pre-Cartier duoidal category will be investigated further in another project. Moreover, since duoidal categories generalize braided monoidal categories by replacing the braiding with an interchange law, one could be interested in introducing a notion of \textit{infinitesimal duoidal category}, generalizing pre-Cartier categories in the same spirit. This would require a great deal of work and goes beyond the scope of this paper. We leave this as a possible future project.
\medskip

\noindent\textbf{Acknowledgments.}
The authors thank J.~Vercruysse for useful suggestions. A.~Sciandra is member of the “National Group for Algebraic and Geometric Structures and their Applications” (GNSAGA-INdAM) and is supported by a postdoctoral fellowship at the ULB within the framework of the PDR project “Reconstruction of modules and algebraic objects from closed and monoidal structures on their representation categories” funded by the FNRS under the grant number T.0318.25F (PI J. Vercruysse). L.~Simons is supported by a joint pre-doctoral fellowship at the ULB and VUB, funded by the FNRS under the grant number $40031463$.

\noindent\textbf{Use of AI.} ChatGPT-6 (Open AI) was used for a final proof-reading of the work and suggested Example \ref{ex:fullpreCartier}. The authors verified all AI-assisted material and take full responsibility for the final text. 

\newpage

\appendix
\section{Proofs of Section \ref{sec:braidingbicomodules}}\label{sec:appendix}


The next lemma provides a bicomodule morphism which will be useful in the classification of braidings on $({}^C\mm^C,\square^{C},C)$.

\begin{lemma}\label{lemma: induced bicomodule morphism}
    Let $V$ be a $C$-bicomodule, and $f\in V^*$. Define
    \begin{equation}\label{eqn: induced bimodule morphism}
        \hat{f}\colon V\to C\otimes C\colon v\mapsto v_{(-1)}\otimes f(v_{(0)})v_{(1)}.
    \end{equation}
    Then $(\ep\ot\ep)\hat{f} = f$ and $\hat{f}$ is a $C$-bicomodule morphism, considering $C\ot C$ in ${}^C\mm^{C}$ with coactions $\Delta\ot\id_{C}$ and $\id_{C}\ot\Delta$.
\end{lemma}
\begin{proof}
    The first claim is immediate by definition of left and right coactions. Let us verify that $\hat{f}$ is a morphism of $C$-bicomodules. 
By definition of the tensor $C$-bicomodule structure it follows that
    \begin{align*}
        \hat{f}(v)_{(-1)}\otimes \hat{f}(v)_{(0)} &= (v_{(-1)}\otimes v_{(1)})_{(-1)} \otimes (v_{(-1)}\otimes v_{(1)})_{(0)} f(v_{(0)})
        \\&= v_{(-1)1}\otimes v_{(-1)2}\otimes v_{(1)}f(v_{(0)})
        \\&= v_{(-1)}\otimes v_{(0)(-1)}\otimes v_{(0)(1)}f(v_{(0)(0)})
        \\&= v_{(-1)}\otimes \hat{f}(v_{(0)})
\end{align*}  
and, similarly, $\hat{f}(v)_{(0)}\otimes \hat{f}(v)_{(1)}=\hat{f}(v_{(0)})\ot v_{(1)}$.
\end{proof}

\begin{myproof}[Proof of Proposition \ref{prop: equivalent conditions canonical R-comatrix}.]
Notice that conditions (\ref{cond: comatrix 1})--(\ref{cond: comatrix 5}) are one-by-one equivalent to stating that, for all $p,q,r,u\in C$, the following equalities hold: 
\begin{align}
        &\RR(p\ot q\ot r_2)r_1 = \RR(p_1\ot q \ot r)p_2,\tag{20'}
        \\&\RR(p_2\ot q \ot r)p_1 =\RR(p\ot q_1\ot r)q_2,\tag{21'}
        \\&\RR(p\ot q_2\ot r)q_1 = \RR(p\ot q \ot r_1)r_2.\tag{22'}
        \\&\RR(p\ot q \ot u)\ep(r) = \RR(p_1\ot q\ot r_1)\RR(p_2\ot r_2\ot u),\tag{23'}
        \\&\RR(p\ot r\ot u)\ep(q) = \RR(p\ot q_1\ot u_1)\RR(q_2\ot r\ot u_2)\tag{24'}.
    \end{align}
    Assume $\RR$ to be convolution-invertible and satisfying conditions (\ref{cond: comatrix 1})--(\ref{cond: comatrix 5}). It suffices to verify the normalizing conditions (\ref{cond: normalizing condition}). By condition (\ref{cond: comatrix 5}) it holds that
    \begin{align*}
        \RR &= (\ep\ot \RR)(\tau\ot\id\ot\id)(\id\ot\Delta\ot\id) 
        \\&= ((\RR\ot \ep)(\id\ot\id\ot\tau)\ast (\ep\ot\RR))(\id\ot\Delta\ot\id)
        \\\iff\ep^{\ot 3} &= \RR^{-1}\ast ((\RR\ot \ep)(\id\ot\id\ot\tau)\ast (\ep\ot\RR))(\id\ot\Delta\ot\id).
    \end{align*}
    Evaluating in an element $p\ot q\ot r\in C\ot C\ot C$ we obtain that
    \begin{align*}
        \ep(p)\ep(q)\ep(r) &= \RR^{-1}(p_1\ot q_1\ot r_1)\RR(p_2\ot q_2\ot r_2)\RR(q_3\ot q_4\ot r_3)
        \\&= \ep(p)\ep(q_1)\ep(r_1)\RR(q_2\ot q_3\ot r_2)
        \\&= \RR(q_1\ot q_2\ot r)\ep(p).
    \end{align*}
    For $p\ot q \ot r= x_2\ot y \ot x_1$ this implies that $\RR(y_1\ot y_2\ot x) = \ep(x)\ep(y)$.\newline
    Similarly, by condition (\ref{cond: comatrix 5})
    \begin{align*}
        \RR &= (\ep\ot \RR)(\tau\ot\id\ot\id)(\Delta\ot\id\ot\id) 
        \\&= ((\RR\ot \ep)(\id\ot\id\ot\tau)\ast (\ep\ot\RR))(\Delta\ot\id\ot\id)
        \\\iff\ep^{\ot 3} &= \RR^{-1}\ast ((\RR\ot \ep)(\id\ot\id\ot\tau)\ast (\ep\ot\RR))(\Delta\ot\id\ot\id).
    \end{align*}
    Evaluating on an element $p\ot q\ot r\in C\ot C\ot C$ we obtain that
    \begin{align*}
        \ep(p)\ep(q)\ep(r)&= \RR^{-1}(p_1\ot q_1\ot r_1)\RR(p_2\ot p_3\ot r_2)\RR(p_4\ot q_2\ot r_3)
        \\&= \RR^{-1}(p_1\ot q_1\ot r_1)\RR(p_4\ot q_2\ot \RR(p_2\ot p_3\ot r_{21})r_{22})
        \\&\stackrel{(\ref{cond: comatrix 3 prime})}{=} \RR^{-1}(p_1\ot q_1\ot r_1)\RR(p_4\ot q_2\ot \RR(p_2\ot p_{32}\ot r_2)p_{31})
        \\&=\RR^{-1}(p_1\ot q_1\ot r_1)\RR(p_2\ot \RR(p_5\ot q_2\ot p_{3})p_{4}\ot r_2)
        \\&\stackrel{(\ref{cond: comatrix 3 prime})}{=} \RR^{-1}(p_1\ot q_1\ot r_1)\RR(p_2\ot q_{21}\ot r_2)\RR(p_4\ot q_{22}\ot p_3)
        \\&= \RR^{-1}(p_1\ot q_1\ot r_1)\RR(p_2\ot q_2\ot r_2)\RR(p_4\ot q_3\ot p_3)
        \\&= \ep(p_1)\ep(q_1)\ep(r)\RR(p_3\ot q_2\ot p_2)
        \\&=\varepsilon(r) \RR(p_2\ot q \ot p_1).
    \end{align*}
    Hence for $p\ot q \ot r = y\ot x_1\ot x_2$ it follows that $\RR(y_2\ot x\ot y_1) = \ep(x)\ep(y)$.
    Assume conversely that $\RR$ satisfies condition (\ref{cond: comatrix 3}) and the normalizing conditions (\ref{cond: normalizing condition}). We first prove that $\RR$ satisfies the cyclic conditions (\ref{cond: cyclic conditions}). For any $p,q,r\in C$ it holds that
    \begin{align*}
        \RR(r\ot p\ot q) &\stackrel{(\ref{cond: normalizing condition})}{=} \RR(p_1\ot p_2\ot r_1)\RR(r_2\ot p_3\ot q)
        \\&= \RR(p_1\ot p_{21}\ot r_1)\RR(r_2\ot p_{22}\ot q)
        \\&\stackrel{(\ref{cond: comatrix 3 prime})}{=} \RR(p_1\ot q_2\ot r_1)\RR(r_2\ot p_2\ot q_1)
        \\&\stackrel{(\ref{cond: comatrix 3 prime})}{=} \RR(p_1\ot q_{22}\ot r)\RR(q_{21}\ot p_2\ot q_1)
        \\&= \RR(p_1\ot q_3\ot r)\RR(q_2\ot p_2\ot q_1)
        \\&\stackrel{(\ref{cond: normalizing condition})}{=} \RR(p_1\ot q_2\ot r)\ep(q_1)\ep(p_2)
        = \RR(p\ot q\ot r).
    \end{align*}
    Hence $\RR$ satisfies the cyclic condition (\ref{cond: cyclic conditions}). From this the additional normalizing condition (\ref{cond: third normalizing condition}) immediately follows by combining \eqref{cond: normalizing condition} and \eqref{cond: cyclic conditions}. We now show that $\RR$ is a convolution-invertible satisfying conditions (\ref{cond: comatrix 1})--(\ref{cond: comatrix 5}), proving that a) and b) are equivalent. Conditions (\ref{cond: comatrix 1}) and (\ref{cond: comatrix 2}) follow immediately by applying \eqref{cond: cyclic conditions} to (\ref{cond: comatrix 3}). For condition (\ref{cond: comatrix 4}) we compute that
    \begin{align*}
        \RR(p_1\ot q\ot r_1)\RR(p_2\ot r_2\ot u)&\stackrel{(\ref{cond: comatrix 1 prime})}{=} \RR(p\ot q\ot r_{2})\RR(r_{1}\ot r_3\ot u)
        \\&\stackrel{(\ref{cond: comatrix 3 prime})}{=} \RR(p\ot q\ot u_2)\RR(r_1\ot r_2\ot u_1)
        \\&\stackrel{(\ref{cond: normalizing condition})}{=} \RR(p\ot q\ot u_2)\ep(r)\ep(u_{1})
        = \RR(p\ot q\ot u)\ep(r).
    \end{align*}
    Similarly, condition (\ref{cond: comatrix 5}) follows since
    \begin{align*}
        \RR(p\ot q_1\ot u_1)\RR(q_2\ot r\ot u_2) &\stackrel{(\ref{cond: comatrix 3 prime})}{=} \RR(p\ot q_{12}\ot u)\RR(q_2\ot r\ot q_{11})
        \\&= \RR(p\ot q_{21}\ot u)\RR(q_{22}\ot r\ot q_1)
        \\&\stackrel{(\ref{cond: comatrix 2 prime})}{=} \RR(p\ot r_2\ot u)\RR(q_2\ot r_1\ot q_1)
        \\&\stackrel{(\ref{cond: normalizing condition})}{=} \RR(p\ot r_2\ot u)\ep(q)\ep(r_1)
        = \RR(p\ot r\ot u)\ep(q).
    \end{align*}
    It remains to prove the invertibility of $\RR$ with respect to the convolution. We claim $\RR(\tau\ot\id)$ to be the inverse of $\RR$. Let $p,q,r\in C$, then
    \begin{align*}
        (\RR\ast \RR(\tau\ot\id))(p\ot q \ot r) &= \RR(p_1\ot q_1\ot r_1)\RR(q_2\ot p_2\ot r_2)
        \\&\stackrel{(\ref{cond: cyclic conditions})}{=} \RR(p_1\ot q_1\ot r_1)\RR(p_2\ot r_2\ot q_2)
        \\&\stackrel{(\ref{cond: comatrix 4 prime})}{=}
        \RR(p\ot q_1\ot q_2)\ep(r)
        \stackrel{(\ref{cond: third normalizing condition})}{=} \ep(p)\ep(q)\ep(r),
        \\(\RR(\tau\ot\id)\ast \RR)(p\ot q\ot r)&= \RR(q_1\ot p_1\ot r_1)\RR(p_2\ot q_2\ot r_2)
        \\&\stackrel{(\ref{cond: comatrix 5 prime})}{=} \RR(q_1\ot q_2\ot r)\ep(p)
        \stackrel{(\ref{cond: normalizing condition})}{=} \ep(p)\ep(q)\ep(r).\qedhere
    \end{align*}
\end{myproof}

\begin{myproof}[Proof of Proposition \ref{prop:solutionYangBaxter}.]
Let $M$ be a fixed $C$-bicomodule, and $p,q,r\in M$. We compute that
    \begin{align*}
        \Upsilon^{12}\Upsilon^{13}\Upsilon^{23}(p\ot q\ot r) 
        &= \RR(q_{(-1)1}\ot q_{(1)2}\ot r_{(-1)1})\RR(p_{(-1)}\ot p_{(1)}\ot q_{(-1)2})
        \\&\hspace{12em}\RR(q_{(-1)3}\ot q_{(1)1}\ot r_{(-1)2})r_{(0)}\ot q_{(0)}\ot p_{(0)}
        \\&= \RR(q_{(-1)11}\ot q_{(1)2}\ot r_{(-1)1})\RR(p_{(-1)}\ot p_{(1)}\ot q_{(-1)12})
        \\&\hspace{12em}\RR(q_{(-1)2}\ot q_{(1)1}\ot r_{(-1)2})r_{(0)}\ot q_{(0)}\ot p_{(0)}
        \\&\overset{(\ref{cond: comatrix 1 prime})}{=} \RR(p_{(-1)2}\ot q_{(1)2}\ot r_{(-1)1})\RR(p_{(-1)1}\ot p_{(1)}\ot q_{(-1)1})
        \\&\hspace{12em}\RR(q_{(-1)2}\ot q_{(1)1}\ot r_{(-1)2})r_{(0)}\ot q_{(0)}\ot p_{(0)}
        \\&\overset{(\ref{cond: comatrix 1 prime})}{=} \RR(p_{(-1)2}\ot q_{(1)2}\ot q_{(-1)22})\RR(p_{(-1)1}\ot p_{(1)}\ot q_{(-1)1})
        \\&\hspace{12em}\RR(q_{(-1)21}\ot q_{(1)1}\ot r_{(-1)})r_{(0)}\ot q_{(0)}\ot p_{(0)}
        \\&\overset{(\ref{cond: comatrix 2 prime})}{=}\RR(p_{(-1)2}\ot q_{(-1)21}\ot q_{(-1)3})\RR(p_{(-1)1}\ot p_{(1)}\ot q_{(-1)1})
        \\&\hspace{12em}\RR(q_{(-1)22}\ot q_{(1)}\ot r_{(-1)})r_{(0)}\ot q_{(0)}\ot p_{(0)}
        \\&=\RR(p_{(-1)2}\ot q_{(-1)12}\ot q_{(-1)3})\RR(p_{(-1)1}\ot p_{(1)}\ot q_{(-1)11})
        \\&\hspace{12em}\RR(q_{(-1)2}\ot q_{(1)}\ot r_{(-1)})r_{(0)}\ot q_{(0)}\ot p_{(0)}
        \\&\overset{(\ref{cond: comatrix 3 prime})}{=} \RR(p_{(-1)2}\ot p_{(1)1}\ot q_{(-1)3})\RR(p_{(-1)1}\ot p_{(1)2}\ot q_{(-1)1})
        \\&\hspace{12em}\RR(q_{(-1)2}\ot q_{(1)}\ot r_{(-1)})r_{(0)}\ot q_{(0)}\ot p_{(0)}
        \\&= \Upsilon^{23}\Upsilon^{13}\Upsilon^{12}(p\ot q\ot r).
    \end{align*}
    Hence the quantum Yang--Baxter equation holds. Analogously, we prove that $\Upsilon$ satisfies the braid equation. Indeed, on the one hand, we get
    \begin{align*}
    \Upsilon^{12}&\Upsilon^{23}\Upsilon^{12}(p\ot q\ot r)
    \\&= \RR(p_{(-1)1}\ot p_{(1)2}\ot q_{(-1)1})\RR(p_{(-1)2}\ot p_{(1)1}\ot r_{(-1)1})
    \\&\hspace{12em}\RR(q_{(-1)2}\ot q_{(1)}\ot r_{(-1)2})r_{(0)}\ot q_{(0)}\ot p_{(0)}
    \\&\overset{(\ref{cond: comatrix 2 prime})}{=} \RR(p_{(-1)1}\ot p_{(-1)21}\ot q_{(-1)1})\RR(p_{(-1)22}\ot p_{(1)}\ot r_{(-1)1})
    \\&\hspace{12em}\RR(q_{(-1)2}\ot q_{(1)}\ot r_{(-1)2})r_{(0)}\ot q_{(0)}\ot p_{(0)}
    \\&= \RR(p_{(-1)11}\ot p_{(-1)12}\ot q_{(-1)1})\RR(p_{(-1)2}\ot p_{(1)}\ot r_{(-1)1})
    \\&\hspace{12em}\RR(q_{(-1)2}\ot q_{(1)}\ot r_{(-1)2})r_{(0)}\ot q_{(0)}\ot p_{(0)}
    \\&\overset{(\ref{cond: normalizing condition})}{=} \ep(p_{(-1)1})\ep(q_{(-1)1})\RR(p_{(-1)2}\ot p_{(1)}\ot r_{(-1)1})\RR(q_{(-1)2}\ot q_{(1)}\ot r_{(-1)2})r_{(0)}\ot q_{(0)}\ot p_{(0)}
    \\&= \RR(p_{(-1)}\ot p_{(1)}\ot r_{(-1)1})\RR(q_{(-1)}\ot q_{(1)}\ot r_{(-1)2})r_{(0)}\ot q_{(0)}\ot p_{(0)}.
    \end{align*}
    On the other hand,
    \begin{align*}
        \Upsilon^{23}\Upsilon^{12}\Upsilon^{23}(p\ot q\ot r)
        &=\RR(q_{(-1)1}\ot q_{(1)}\ot r_{(-1)1})\RR(p_{(-1)1}\ot p_{(1)2}\ot r_{(-1)2})
        \\&\hspace{12em} \RR(p_{(-1)2}\ot p_{(1)1}\ot q_{(-1)2})r_{(0)}\ot q_{(0)}\ot p_{(0)}
        \\&\overset{(\ref{cond: comatrix 2 prime})}{=} \RR(q_{(-1)1}\ot q_{(1)}\ot r_{(-1)1})\RR(p_{(-1)1}\ot p_{(-1)21}\ot r_{(-1)2})
        \\&\hspace{12em} \RR(p_{(-1)22}\ot p_{(1)}\ot q_{(-1)2})r_{(0)}\ot q_{(0)}\ot p_{(0)}
        \\&= \RR(q_{(-1)1}\ot q_{(1)}\ot r_{(-1)1})\RR(p_{(-1)11}\ot p_{(-1)12}\ot r_{(-1)2})
        \\&\hspace{12em} \RR(p_{(-1)2}\ot p_{(1)}\ot q_{(-1)2})r_{(0)}\ot q_{(0)}\ot p_{(0)}
        \\&\overset{(\ref{cond: normalizing condition})}{=} \ep(p_{(-1)1})\ep(r_{(-1)2})\RR(q_{(-1)1}\ot q_{(1)}\ot r_{(-1)1})
        \\&\hspace{12em} \RR(p_{(-1)2}\ot p_{(1)}\ot q_{(-1)2})r_{(0)}\ot q_{(0)}\ot p_{(0)}
        \\&= \RR(q_{(-1)1}\ot q_{(1)}\ot r_{(-1)})\RR(p_{(-1)}\ot p_{(1)}\ot q_{(-1)2})r_{(0)}\ot q_{(0)}\ot p_{(0)}
        \\&\overset{(\ref{cond: comatrix 1 prime})}{=} \RR(q_{(-1)}\ot q_{(1)}\ot r_{(-1)2})\RR(p_{(-1)}\ot p_{(1)}\ot r_{(-1)1})r_{(0)}\ot q_{(0)}\ot p_{(0)}.
    \end{align*}
    
    Finally, we compute
    \begin{align*}
        \Upsilon&\Upsilon\Upsilon(p\ot q)
        \\&= \RR(p_{(-1)1}\ot p_{(1)2}\ot q_{(-1)1})\RR(q_{(-1)2}\ot q_{(1)}\ot p_{(-1)2})\RR(p_{(-1)3}\ot p_{(1)1}\ot q_{(-1)3})q_{(0)}\ot p_{(0)}
        \\&= \RR(p_{(-1)11}\ot p_{(1)2}\ot q_{(-1)1})\RR(q_{(-1)2}\ot q_{(1)}\ot p_{(-1)12})\RR(p_{(-1)2}\ot p_{(1)1}\ot q_{(-1)3})q_{(0)}\ot p_{(0)}
        \\&\overset{(\ref{cond: comatrix 1 prime})}{=} \RR(p_{(-1)1}\ot p_{(1)2}\ot q_{(-1)12})\RR(q_{(-1)2}\ot q_{(1)}\ot q_{(-1)11})\RR(p_{(-1)2}\ot p_{(1)1}\ot q_{(-1)3})q_{(0)}\ot p_{(0)}
        \\&= \RR(p_{(-1)1}\ot p_{(1)2}\ot q_{(-1)21})\RR(q_{(-1)22}\ot q_{(1)}\ot q_{(-1)1})\RR(p_{(-1)2}\ot p_{(1)1}\ot q_{(-1)3})q_{(0)}\ot p_{(0)}
        \\&\overset{(\ref{cond: comatrix 2 prime})}{=} \RR(p_{(-1)1}\ot p_{(1)2}\ot q_{(1)2})\RR(q_{(-1)2}\ot q_{(1)1}\ot q_{(-1)1})\RR(p_{(-1)2}\ot p_{(1)1}\ot q_{(-1)3})q_{(0)}\ot p_{(0)}
        \\&\overset{(\ref{cond: normalizing condition})}{=} \ep(q_{(-1)1})\ep(q_{(1)1})\RR(p_{(-1)1}\ot p_{(1)2}\ot q_{(1)2})\RR(p_{(-1)2}\ot p_{(1)1}\ot q_{(-1)2})q_{(0)}\ot p_{(0)}
        \\&= \RR(p_{(-1)1}\ot p_{(1)2}\ot q_{(1)})\RR(p_{(-1)2}\ot p_{(1)1}\ot q_{(-1)})q_{(0)}\ot p_{(0)}
        \\&\overset{(\ref{cond: comatrix 1 prime})}{=} \RR(p_{(-1)}\ot p_{(1)2}\ot q_{(1)2})\RR(q_{(1)1}\ot p_{(1)1}\ot q_{(-1)})q_{(0)}\ot p_{(0)}
        \\&\overset{(\ref{cond: cyclic conditions})}{=} \RR(p_{(1)2}\ot q_{(1)2}\ot p_{(-1)})\RR(p_{(1)1}\ot q_{(-1)}\ot q_{(1)1})q_{(0)}\ot p_{(0)}
        \\&\overset{(\ref{cond: comatrix 4 prime})}{=} \RR(p_{(1)}\ot q_{(-1)}\ot p_{(-1)})\ep(q_{(1)})q_{(0)}\ot p_{(0)}
        \\&\overset{(\ref{cond: cyclic conditions})}{=} \RR(p_{(-1)}\ot p_{(1)}\ot q_{(-1)})q_{(0)}\ot p_{(0)} = \Upsilon(p\ot q).\qedhere
    \end{align*}
\end{myproof}

\begin{myproof}[Proof of Theorem \ref{theorem: braidings on bicomodules}]
    Let $V$ and $W$ be $C$-bicomodules, and assume the existence of a braiding $\sigma$ on $(^C\mm^C, \square^C, C)$. Define $\RR$ as in (\ref{eqn: canonical r comatrix}). For any $f\in V^*$ and $g\in W^*$ we consider the morphisms $\hat{f}\colon V\to C\ot C$ and $\hat{g}\colon W\to C\ot C$ in $^C\mm^C$ defined as in Lemma \ref{lemma: induced bicomodule morphism}. We observe that
\begin{align*}
        (g\ot f)e_{W,V}\sigma_{V,W}&= \ep^{\ot 4}(\hat{g}\ot \hat{f})e_{W,V}\sigma_{V,W}
        = \ep^{\ot 4}e_{C^{\ot2},C^{\ot 2}}(\hat{g}\square \hat{f})\sigma_{V,W}
        = \ep^{\ot 4}e_{C^{\ot2},C^{\ot 2}}\sigma_{C^{\ot2},C^{\ot 2}}(\hat{f}\square\hat{g})
\end{align*}
For notational ease, we will refer by slight abuse of notation to $\sigma_{C^{\ot2}, C^{\ot2}}$ simply as $\sigma$. We will also avoid writing inclusions $e_{C^{\ot2},C^{\ot2}}$. Using
    the explicit isomorphism $C\cong C\square^C C$ from Remark \ref{lemma: left right unitors for vec}, it follows that
\begin{align*}
        (g\ot f)e_{W,V}\sigma_{V,W}(v\square w)&=\ep^{\ot 4}\sigma(v_{(-1)}\ot v_{(1)}\square w_{(-1)}\ot w_{(1)})f(v_{(0)})g(w_{(0)})
        \\&= \ep^{\ot 4}\sigma(v_{(-1)}\ot v_{(1)1}\square v_{(1)2}\ot w_{(1)})\ep(w_{(-1)})(g\ot f)e_{W,V}(w_{(0)}\square v_{(0)})\\&= \ep^{\ot 4}\sigma(\id\ot\Delta\ot\id)(v_{(-1)}\ot v_{(1)}\ot w_{(1)})(g\ot f)e_{W,V}(w_{(0)}\square v_{(0)})
        \\&= (g\ot f)e_{W,V}(\RR(v_{(-1)}\ot v_{(1)}\ot w_{(1)})w_{(0)}\square v_{(0)}).
    \end{align*}
    By the non-degeneracy of the evaluation $W^*\ot V^*\ot W\ot V\to \Bbbk$, the fact that $f$ and $g$ were chosen arbitrarily implies that $e_{W,V}\sigma_{V,W}(v\square w)=e_{W,V}(\RR(v_{(-1)}\ot v_{(1)}\ot w_{(1)})w_{(0)}\square v_{(0)})$. Since $e_{W,V}$ is a monomorphism, we obtain the first equality of (\ref{eqn: definition of braiding}). The second equality holds likewise, since $v_{(1)}\square w_{(-1)} = \ep(v_{(1)})w_{(-1)1}\square w_{(-1)2}$. 
    \newline
    The fact that $\sigma_{V,W}(V\square W)\subseteq W\square V$ states that for all $v\square w\in V\square^C W$
    \[
    \RR(v_{(-1)}\ot v_{(1)}\ot w_{(1)})w_{(0)(0)}\ot w_{(0)(1)}\ot v_{(0)} = \RR(v_{(-1)}\ot v_{(1)}\ot w_{(1)})w_{(0)}\ot v_{(0)(-1)}\ot v_{(0)(0)}.
    \]
    In particular, for $p,q,r\in C$, it holds that $p\ot q_1\square q_2\ot r\in C^{\ot 2}\square^C C^{\ot 2}$, and this becomes the condition that
\[
\RR(p_1\ot q_2\ot r_3)q_3\ot r_1\ot r_2\ot p_2\ot q_1 = \RR(p_1\ot q_2\ot r_2)q_3\ot r_1\ot p_2\ot p_3\ot q_1.
\]
    Applying $\ep\ot\ep\ot\id\ot \ep\ot\ep$ results in identity (\ref{cond: comatrix 1 prime}). Let $v\in V$, $w\in W$, and $u\in U$, with $V,W,U$ three $C$-bicomodules, and $p,q,r,u\in C$ from now on.
    The fact that $\sigma_{V,W}$ is a morphism in $^C\mm^C$ translates to
    \[
    \sigma_{V,W}(v\square w)_{(-1)}\ot \sigma_{V,W}(v\square w)_{(0)}\ot \sigma_{V,W}(v\square w)_{(1)}\ = v_{(-1)}\ot \sigma_{V,W}(v_{(0)}\square w_{(0)})\ot w_{(1)}.
    \]
    In the case of $V=W=C\ot C$ and $p\ot q_1\square q_2\ot r \in C^{\ot 2}\square^C C^{\ot 2}$ this becomes
    \begin{align*}
        \RR(p_1\ot q_3\ot r_2)q_4\ot q_5\ot r_1\square p_2\ot q_1\ot q_2 = \RR(p_2\ot q_2\ot r_2)p_1\ot q_3\ot r_1\square p_3\ot q_1\ot r_3.
    \end{align*}
    Applying $\id\ot\ep^{\ot 5}$ respectively $\ep^{\ot 5}\ot \id$ results in condition (\ref{cond: comatrix 2 prime}) respectively (\ref{cond: comatrix 3 prime}).
    The braid relation $\sigma_{V, W\square^C U} = (\id\square \sigma_{V,U})(\sigma_{V,W}\square\id)$ translates, considered for $V=W=U=C\ot C$ and $p\ot q_1\square q_2\ot r_1\square r_2\ot u\in C^{\ot 2}\square^C C^{\ot 2}\square^C C^{\ot 2}$, to
    \begin{align*}
        \RR(p_1\ot q_2\ot u_2)&q_3\ot r_1\square r_2\ot u_1\square p_2\ot q_1
        \\&= \RR(p_1\ot q_2\ot r_2)\RR(p_2\ot r_3\ot u_2)q_3\ot r_1\square r_4\ot u_1\square p_3\ot q_1.
    \end{align*}
    Applying $\ep^{\ot 6}$ results in condition (\ref{cond: comatrix 4 prime}).
    Similarly, the braid relation $\sigma_{V\square^C W, U} = (\sigma_{V,U}\square \id)(\id\square \sigma_{W,U})$
    translates to
    \begin{align*}
        \RR(p_1\ot r_2\ot u_2)&r_3\ot u_1\square p_2\ot q_1\square q_2\ot r_1
        \\&= \RR(q_2\ot r_2\ot u_3)\RR(p_1\ot r_3\ot u_2)r_4\ot u_1\square p_2\ot q_1\square q_3\ot r_1.
    \end{align*}
    Applying $\ep^{\ot 6}$ we obtain condition (\ref{cond: comatrix 5 prime}), since
    \[
        \RR(p\ot r\ot u)\ep(q) = \RR(p\ot r_2\ot u_1)\RR(q\ot r_1\ot u_2) \overset{(\ref{cond: comatrix 2 prime})}{=} \RR(p\ot q_1\ot u_1)\RR(q_2\ot r\ot u_2).
    \]
    Notice that $\sigma_{V,W}^{-1}$ is a braiding on $(^C\mm^C, \square^C, C)$ as well and hence by the above reasoning there exists a linear functional $\Tt\colon C\ot C\ot C\to \Bbbk$
    satisfying conditions (\ref{cond: comatrix 1})--(\ref{cond: comatrix 5}). Then
    \begin{align*}
        v\square w = \sigma_{V,W}^{-1}\sigma_{V,W}(v\square w) = \RR(v_{(-1)}\ot v_{(1)}\ot w_{(1)})\Tt(w_{(0)(-1)}\ot v_{(0)(-1)}\ot v_{(0)(1)})v_{(0)(0)}\square w_{(0)(0)}.
    \end{align*}
    For $V = W = C\ot C$ and an element $p\ot q_1\square q_2\ot r \in C^{\ot 2}\square^C C^{\ot 2}$, this becomes
    \[
    p\ot q_1\square q_2\ot r= \RR(p_1\ot q_3\ot r_2)\Tt(q_4\ot p_2\ot q_2)p_3\ot q_1\square q_5\ot r_1.
    \]
    Applying $\ep^{\ot 4}$ we obtain that
    \begin{align*}
    \ep^{\ot 3}(p\ot q\ot r) &= \RR(p_1\ot q_2\ot r)\Tt(q_3\ot p_2\ot q_1) 
    \\&= \RR(p_1\ot q_{12}\ot r)\Tt(q_2\ot p_2\ot q_{11})\overset{(\ref{cond: comatrix 3 prime})}{=} \RR(p_1\ot q_1\ot r_1)\Tt(q_2\ot p_2\ot r_2).
    \end{align*}
    Hence $\Tt(\tau\ot \id)$ is the right convolution inverse to $\RR$. By symmetry, it holds that $\RR(\tau\ot\id)$ is right inverse to $\Tt$, which is equivalent to stating that $\Tt(\tau\ot \id)$ is left inverse to $\RR$, whence $\RR$ is convolution invertible. Moreover, by condition (\ref{cond: symmetry}) of Proposition \ref{prop: equivalent conditions canonical R-comatrix}, $\Tt(\tau\ot\id) = \RR^{-1} = \RR(\tau\ot\id)$, whence $\Tt = \RR$ and $\sigma_{V,W}^{-1} = \sigma_{W,V}$.

Conversely, assume $\RR\colon C\ot C\ot C \to \Bbbk$ to be a canonical $R$-form of $C$. For $C$-bicomodules $V,W$ define the linear map
    \[
    \sigma_{V,W}\colon V\square^C W \to W\ot V \colon v\square w \mapsto \RR(v_{(-1)}\ot v_{(1)}\ot w_{(1)})w_{(0)}\ot v_{(0)},
    \]
    i.e. $\sigma_{V,W}=\tilde{\sigma}_{V,W}|_{V\square^{C}W}$. Since $v\square w \in V\square^C W$, it holds that $v_{(0)}\ot v_{(1)}\ot w = v\ot w_{(-1)}\ot w_{(0)}$, and hence
    \[
    \sigma_{V,W}(v\square w) = \RR(v_{(-1)}\ot v_{(1)}\ot w_{(1)})w_{(0)}\ot v_{(0)} = \RR(v_{(-1)}\ot w_{(-1)}\ot w_{(1)})w_{(0)}\ot v_{(0)}.
    \]
    Utilizing condition \eqref{cond: comatrix 1 prime} one verifies that $\sigma_{V,W}$ has image in $W\square^C V$. Conditions \eqref{cond: comatrix 2 prime} (resp.\ \eqref{cond: comatrix 3 prime}) translate to the left (resp.\ right) $C$-colinearity of $\sigma_{V,W}$.
    
    \begin{invisible}
    \underline{$\sigma_{V,W}$ has image in $W\square^C V$.}\newline
    Let $v\square w \in V\square^C W$, then we compute that
    \begin{align*}
        &\RR(v_{(-1)}\ot v_{(1)}\ot w_{(1)}) w_{(0)(0)}\ot w_{(0)(1)}\ot v_{(0)}
        \\&= \RR(v_{(-1)}\ot v_{(1)}\ot w_{(1)2})w_{(0)}\ot w_{(1)1}\ot v_{(0)}
        \\&\stackrel{(\ref{cond: comatrix 1 prime})}{=} \RR(v_{(-1)1}\ot v_{(1)}\ot w_{(1)})w_{(0)}\ot v_{(-1)2}\ot v_{(0)}
        \\&= \RR(v_{(-1)}\ot v_{(1)}\ot w_{(1)})w_{(0)}\ot v_{(0)(-1)}\ot v_{(0)(0)}.
    \end{align*}
    Hence $\sigma_{V,W}$ has image in $W\square^C V$.\newline
    \underline{$\sigma_{V,W}$ is a $^C\mm^C$-morphism.}\newline
    Let $v\square w \in V\square^C W$, then we compute that
    \begin{align*}
        c_{V,W}&(v\square w)_{(-1)}\ot c_{V,W}(v\square w)_{(0)}
        \\&= \RR(v_{(-1)}\ot v_{(1)}\ot w_{(1)}) w_{(0)(-1)}\ot w_{(0)(0)}\square v_{(0)}
        \\&= \RR(v_{(-1)}\ot w_{(-1)1}\ot w_{(1)})w_{(-1)2}\ot w_{(0)}\square v_{(0)}
        \\&\stackrel{(\ref{cond: comatrix 2 prime})}{=} \RR(v_{(-1)2}\ot w_{(-1)}\ot w_{(1)})v_{(-1)1}\ot w_{(0)}\square v_{(0)}
        \\&= \RR(v_{(0)(-1)}\ot w_{(-1)}\ot w_{(1)})v_{(-1)}\ot w_{(0)}\square v_{(0)(0)}
        \\&= (v\square w)_{(-1)}\ot \sigma_{V,W}((v\square w)_{(0)}),
    \end{align*}
    hence $\sigma_{V,W}$ is left $C$-colinear. Analogously, one proves that $\sigma_{V,W}$ is a right $C$-colinear morphism by using condition (\ref{cond: comatrix 3 prime}).\newline
    \underline{Braid relations.}\newline
    \end{invisible}
    For the braid relations, let $U$ be a $C$-bicomodule, and $v\square w\square u\in V\square^C W\square^C U$, then
    \begin{align*}
        (\id\square \sigma_{V,U})&(\sigma_{V,W}\square\id)(v\square w \square u)
        \\= &\RR(v_{(-1)}\ot v_{(1)}\ot w_{(1)})\RR(v_{(0)(-1)}\ot v_{(0)(1)}\ot u_{(1)})w_{(0)}\square u_{(0)}\square v_{(0)(0)}
        \\= &\RR(v_{(-1)1}\ot v_{(1)2}\ot w_{(1)})\RR(v_{(-1)2}\ot v_{(1)1}\ot u_{(1)})w_{(0)}\square u_{(0)}\square v_{(0)}
        \\\stackrel{(\ref{cond: comatrix 3 prime})}{=} &\RR(v_{(-1)1}\ot v_{(1)}\ot w_{(1)1})\RR(v_{(-1)2}\ot w_{(1)2}\ot u_{(1)})w_{(0)}\square u_{(0)}\square v_{(0)}
        \\\stackrel{(\ref{cond: comatrix 4 prime})}{=} &\RR(v_{(-1)}\ot v_{(1)}\ot u_{(1)})\ep(w_{(1)})w_{(0)}\square u_{(0)}\square v_{(0)}
        \\= &\RR(v_{(-1)}\ot v_{(1)}\ot u_{(1)})w\square u_{(0)}\square v_{(0)}
        \\= &\RR(v_{(-1)}\ot v_{(1)}\ot (w\square u)_{(1)}) (w\square u)_{(0)}\square v_{(0)}
        \\= &\sigma_{V, W\square^C U}(v\square w\square u).
    \end{align*}
    Similarly, one verifies that $\sigma_{V\square W, U} = (\sigma_{V,U}\square\id)(\id\square \sigma_{W,U})$.\newline\begin{invisible}
    \underline{$\sigma$ is a natural isomorphism.}\newline\end{invisible}
    Finally, we verify that $\sigma$ is a natural isomorphism. One verifies immediately that $\sigma$ is a natural transformation from $\square^C \to (\square^C)^{\op}$. It suffices to prove that $\sigma_{V,W}$ is a linear isomorphism for all $C$-bicomodules $V$ and $W$. Inspired by the converse case, we expect the inverse to $\sigma_{V,W}$ to be $\sigma_{W,V}$.
    \begin{align*}
        \sigma_{W,V}\sigma_{V,W}(v\square w)&= \RR(v_{(-1)}\ot v_{(1)}\ot w_{(1)})\sigma_{W,V}(w_{(0)}\square v_{(0)})
        \\&= \RR(v_{(-1)}\ot v_{(1)}\ot w_{(1)})\RR(w_{(0)(-1)}\ot w_{(0)(1)}\ot v_{(0)(1)})v_{(0)(0)}\square w_{(0)(0)}
        \\&= \RR(v_{(-1)}\ot v_{(1)2}\ot w_{(1)2})\RR(w_{(-1)}\ot w_{(1)1}\ot v_{(1)1})v_{(0)}\square w_{(0)}
        \\&\stackrel{(\ref{cond: cyclic conditions})}{=} \RR(v_{(1)2}\ot w_{(1)2}\ot v_{(-1)})\RR(v_{(1)1}\ot w_{(-1)}\ot w_{(1)1})v_{(0)}\square w_{(0)}
        \\&\stackrel{(\ref{cond: comatrix 4 prime})}{=} \RR(v_{(1)}\ot w_{(-1)}\ot v_{(-1)})\ep(w_{(1)})v_{(0)}\square w_{(0)}
        \\&=\RR(v_{(1)}\ot w_{(-1)}\ot v_{(0)(-1)})v_{(0)(0)}\square w_{(0)}
        \\&\stackrel{(\star)}{=} \RR(w_{(-1)}\ot w_{(0)(-1)}\ot v_{(-1)})v_{(0)}\square w_{(0)(0)}
        \\&= \RR(w_{(-1)1}\ot w_{(-1)2}\ot v_{(-1)})v_{(0)}\square w_{(0)}
        \\&\stackrel{(\ref{cond: normalizing condition})}{=} \ep(w_{(-1)})\ep(v_{(-1)})v_{(0)}\square w_{(0)}
        \\&= v\square w,
    \end{align*}
    and hence $\sigma_{W,V}$ is a right inverse as well by symmetry. 
\end{myproof}

\begin{myproof}[Proof of Proposition \ref{prop: unit canonical R-form iff condition}]
        Let $\RR = f\ot g\ot p$ be a decomposable canonical $R$-form of $C$ in $C^*\ot C^*\ot C^*$. By conditions (\ref{cond: comatrix 4}) respectively (\ref{cond: comatrix 5}) it follows that
        \begin{align*}
            &f\ot g\ot \ep\ot p = (\RR\ot \ep)\ast (f\ot \ep\ot g\ot p),
            \\&f\ot \ep\ot g\ot p= (f\ot g\ot\ep\ot p)\ast(\ep\ot\RR).
        \end{align*}
        Since $\RR$ is invertible this implies that
        \[
            \ep^{\ot 4} = f\ot \ep \ot (p\ast g)\ot \ep,
            \hspace{3em}\ep^{\ot 4} = \ep\ot(g\ast f)\ot \ep\ot p.
        \]
        Consequently, by evaluating both sides on an element $p_1\ot p_2\ot q_1\ot q_2\in C^{\ot 4}$, with $p,q\in C$, these imply that
        \begin{align}
            &\ep\ot\ep = f\ot p\ast g,\label{eqn: monomial1}
            \\&\ep\ot\ep = g\ast f\ot p.\label{eqn: monomial2}
        \end{align}
        By evaluating identities (\ref{eqn: monomial1}) respectively (\ref{eqn: monomial2}) on some element $p_1\ot p_2\in C\ot C$, with $p\in C$, it follows that a priori
        \[
    f\ast p\ast g = \ep, \hspace{3em} g\ast\RR^1\ast\RR^3 = \ep.
        \]
        Hence $g\in C^*$ is invertible. 
        Finally, we compute 
        \begin{align*}
            \RR&= \RR\ast (\ep\ot g\ot\ep)\ast (\ep\ot g^{-1}\ot\ep)
            \\&\overset{(\ref{cond: comatrix 2})}{=} (g\ot\ep\ot\ep)\ast\RR\ast(\ep\ot g^{-1}\ot\ep)
            \\&= (g\ast f) \ot \ep\ot p
            \\&\overset{(\ref{eqn: monomial2})}{=} \ep\ot\ep\ot\ep.
        \end{align*}
        For the second part, notice that by Proposition \ref{prop: equivalent conditions canonical R-comatrix}, $\ep\ot\ep\ot\ep$ is a canonical $R$-form of $C$ if and only if it satisfies conditions (\ref{cond: comatrix 3}) and (\ref{cond: normalizing condition}). The latter trivially holds, while the former is equivalent to stating that $f\ot \ep = \ep\ot f$ for all $f\in C^*$. This is moreover equivalent to $\ep^*\colon \Bbbk\to C^*$ being an epimorphism of rings, see \cite[Proposition 3.3]{agore2014braidings}.\qedhere 
\end{myproof}

\bibliography{preCartierYD}
\bibliographystyle{acm}

\end{document}